\documentclass[onseside, final]{amsart}

\newif\ifgitinfo
\ifgitinfo
	\usepackage[mark]{gitinfo2}
	
\fi
\usepackage[numbers, sort, longnamesfirst]{natbib}
\usepackage{amsaddr}
\usepackage{mathtools, amssymb, amsmath, amsthm}
\usepackage{enumitem}

\usepackage{xcolor, graphicx, subcaption}
\graphicspath{{./}}
\makeatletter
\def\input@path{{./}}
\makeatother
\usepackage{dsfont}
\usepackage{setspace, geometry}
\usepackage{hyperref}
\hypersetup{colorlinks = true, linkcolor = blue, citecolor = blue, pdfborder={0 0 0}, pdfpagemode=UseOutlines, bookmarks = true, hypertexnames = false}
\usepackage{cleveref, autonum}
	\crefname{figure}{figure}{figures}
	\Crefname{figure}{Figure}{Figures}
	\Crefname{assumption}{Assumption}{Assumptions}
	\crefname{enumi}{}{}
	\Crefname{enumi}{}{}
	\Crefname{ass}{Assumption \ref{ass:comon-scheme}}{Assumptions \ref{ass:comon-scheme}}
	\Crefname{ass-sub}{Assumption \ref{ass:comon-scheme} (i)}{Assumptions \ref{ass:comon-scheme} (i)}
	\Crefname{ass-exp-growth}{Assumption \ref{ass:exp-growth}}{Assumptions \ref{ass:exp-growth}}
	\Crefname{ass-pw-lip}{Assumption \ref{ass:pw-lip}}{Assumptions \ref{ass:pw-lip}}
	\Crefname{ass-zvonkin}{Assumption \ref{ass:zvonkin}}{Assumptions \ref{ass:zvonkin}}
	\Crefname{ass-growth-disc}{Assumption \ref{ass:growth-disc}}{Assumptions \ref{ass:growth-disc}}
	\Crefname{ass-growth-disc-A}{\Cref{it:A1}}{\Cref{it:A1}}

\usepackage{pgfplots}
\pgfplotsset{compat=newest,
    width=6cm,
    height=3cm,
    scale only axis=true,
    max space between ticks=25pt,
    try min ticks=5,
    every axis/.style={
        axis y line=left,
        axis x line=bottom,
        axis line style={line width = 0.3pt,-,>=latex, shorten >=-.4cm}
    },
    every axis plot/.append style={thick},
    tick style={black, thin}
}
\tikzset{
    semithick/.style={line width=0.5pt},
}

\newtheorem{theorem}{Theorem}[section]
\newtheorem*{mainresult*}{Main result}
\newtheorem*{auxresult*}{Auxiliary result}
\newtheorem{proposition}[theorem]{Proposition}
\newtheorem{lemma}[theorem]{Lemma}
\newtheorem{corollary}[theorem]{Corollary}
\theoremstyle{definition}
\newtheorem{definition}[theorem]{Definition}
\newtheorem{example}[theorem]{Example}
\newtheorem{remark}[theorem]{Remark}
\newtheorem{assumption}[theorem]{Assumption}

\newcommand{\F}{\mathcal F}

\newcommand{\X}{\mathcal X}
\newcommand{\Y}{\mathcal Y}

\newcommand{\B}{\mathcal B}
\newcommand{\Pc}{\mathcal P}

\newcommand{\N}{\mathbb N}
\newcommand{\R}{\mathbb R}
\newcommand{\E}{\mathbb E}

\renewcommand{\P}{\mathbb P} 
\newcommand{\Law}{\mathrm{Law}}
\newcommand{\id}{\mathrm{id}}
\newcommand{\sign}{\mathrm{sign}}
\newcommand{\corr}{\mathrm{corr}}
\newcommand{\loc}{\mathrm{loc}}

\newcommand{\D}{\,\mathrm{d}}
\newcommand{\ds}{\mathrm{d}}
\newcommand{\ind}[1]{\mathds{1}_{\!\left\{{#1}\right\}\!}}
\newcommand{\W}{\mathcal W}
\newcommand{\AW}{\mathcal A\mathcal W}

\newcommand{\cpl}{\mathrm{Cpl}}

\newcommand{\cplba}{\cpl_{\mathrm{bc}}}
\newcommand{\cplbc}{\cpl_{\mathrm{bc}}}
\newcommand{\kr}{\mathrm{KR}}
\newcommand{\sync}{\mathrm{sync}}

\newcommand{\antitone}{\mathrm{AT}}
\newcommand{\mono}{\mathrm{mono}}
\newcommand{\EM}{\mathrm{EM}}
\newcommand{\IEM}{\mathrm{IEM}}

\newcommand{\continuityset}{\R \setminus \bigcup_{k \in \{1, \dotsc, m\}}(\xi_k - c_0, \xi_k + c_0)}

\author{Michaela Hitz \and Benjamin A.\ Robinson}
	\address{Department of Statistics, University of Klagenfurt, Austria}
	\email{michaela.hitz@aau.at}
	\email{benjamin.robinson@aau.at}

\thanks{This research was funded in part by the Austrian Science Fund (FWF) [10.55776/Y782], [10.55776/P35519], [10.55776/P34743]. For open access purposes, the author has applied a CC BY public copyright license to any author accepted manuscript version arising from this submission.}

\title{Bicausal optimal transport for SDEs with irregular coefficients}
\date{\today}

\numberwithin{equation}{section}

\usepackage{subfiles}

\begin{document}

\begin{abstract}
	We solve constrained optimal transport problems in which the marginal laws are given by the laws of solutions of stochastic differential equations (SDEs). We consider SDEs with irregular coefficients, making only minimal regularity assumptions. We show that the so-called synchronous coupling is optimal among bicausal couplings, that is couplings that respect the flow of information encoded in the stochastic processes. Our results provide a method to numerically compute the adapted Wasserstein distance between laws of SDEs with irregular coefficients. We show that this can be applied to quantifying model uncertainty in stochastic optimisation problems. Moreover, we introduce a transformation-based semi-implicit numerical scheme and establish the first strong convergence result for SDEs with exponentially growing and discontinuous drift.
\end{abstract}
\maketitle

\section{Introduction}\label{sec:intro}
	A fundamental question in stochastic modelling is that of quantifying the effects of model uncertainty. In this context it is of interest to compute an appropriate distance between different stochastic models. In particular, we focus on models that are described by stochastic differential equations (SDEs), potentially having irregular coefficients.
	The main contribution of this paper is to provide a methodology to efficiently compute such a distance between the laws of SDEs, by solving an optimal transport problem. Our approach successfully brings together optimal transport and numerical analysis of SDEs. Moreover, for SDEs with exponentially growing and discontinuous drift, we prove strong well-posedness and strong convergence of a novel numerical scheme.

	In order to compare stochastic models, a reasonable choice of distance is a modification of the Wasserstein distance on the space of probability measures. The Wasserstein distance arises from an optimal transport problem and is defined as the infimum of an expected cost over the set of couplings $\cpl(\mu, \nu)$ between probability measures $\mu$ and $\nu$, that is the set of probability measures on the product space having marginals $\mu$ and $\nu$. This distance metrises the weak topology. However, in many situations the weak topology is unsuitable for comparing stochastic processes, since it does not take into account the information structure that is encoded in the filtrations generated by the processes. As a remedy, the adapted Wasserstein distance and associated bicausal optimal transport problem has been proposed. This constitutes the main object of study in this paper.
		
	We consider stochastic models described by scalar SDEs of the form
		\begin{equation}\label{eq:sde}
			\ds X_t = b_t(X_t) \D t + \sigma_t(X_t) \D W_t, \quad t \in [0, T],
		\end{equation}
	with $X_0 = x_0 \in \R$, where $T \in (0, \infty)$, $b \colon [0, T] \times \R \to \R$ and $\sigma \colon [0, T] \times \R \to [0, \infty)$ are measurable functions, and $W=(W_t)_{t\in[0,T]}$ is a standard scalar Brownian motion.
	Now consider two different models with data $(b,\sigma,x_0,W)$, respectively $(\bar b, \bar \sigma,x_0,\bar W)$.
	For $\Omega = C([0, T], \R)$, $p\ge 1$, and $\Pc_p(\Omega)$ being the set of all probability measures on $\Omega$ with finite $p$\textsuperscript{th} moment, denote the law of the solutions of these SDEs by $\mu \in \Pc_p(\Omega)$, respectively $\nu \in \Pc_p(\Omega)$. Let $(\omega, \bar \omega)$ denote the canonical process on $\Omega \times \Omega$.
	Then the adapted Wasserstein distance between $\mu$ and $\nu$ is defined by
	\begin{equation}
		\AW_p^p(\mu, \nu) \coloneqq \inf_{\pi \in \cplba(\mu, \nu)}\E^\pi\biggl[\int_0^T |\omega_t - \bar \omega_t |^p \D t\biggr],
	\end{equation}
	where we optimise over the set $\cplba(\mu, \nu)$ of bicausal couplings between $\mu$ and $\nu$, that is the set of couplings that satisfy an additional constraint that encodes the filtrations of the processes. Informally, this bicausality condition imposes that for any $t \in [0, T]$, the value $X_t$ is conditionally independent of the process $\bar X$ given $(\bar X_s)_{s \in [0, t]}$, with the symmetric condition when exchanging the roles of $X$ and $\bar X$. In the context of optimisation problems for stochastic processes, such as those that arise in mathematical finance and mathematical biology, the causality condition appears natural. One is typically only able to make decisions based on information contained in the filtration of the process at the present time.
	
	We aim to solve such a bicausal transport problem under minimal regularity assumptions on the coefficients of the SDE \eqref{eq:sde}.
	For one-dimensional SDEs \eqref{eq:sde} with smooth time-homogeneous coefficients, \citet[\S 2.1]{BiTa19} exploit PDE methods to show that the so-called synchronous coupling attains the adapted Wasserstein distance $\AW_2$. This coupling is the joint law of the SDEs when they are driven by a common Brownian motion. Following this work, \citet[Theorem 1.3]{BaKaRo22} show that the synchronous coupling attains $\AW_p$, $p \ge 1$, between the laws of SDEs with globally Lipschitz coefficients. The result of \citep{BaKaRo22} holds also for continuous linearly growing coefficients, provided that pathwise uniqueness holds. However, no explicit conditions for this are given in \cite{BaKaRo22}.
	
	We show that optimality of the synchronous coupling holds in much greater generality, considering also SDEs with irregular coefficients.
	We consider two classes of irregularity:
	\begin{enumerate}[label = (\arabic*)]
		\item \label{it:1} discontinuous drift that grows up to exponentially, degenerate diffusion;
		\item \label{it:2} bounded measurable drift, bounded diffusion that is $\alpha$-H\"older, with $\alpha \in [1/2,1]$.
	\end{enumerate}
	Our main result can be stated informally as follows:
	\begin{mainresult*}
		For models with data $(b,\sigma,x_0,W)$, respectively $(\bar b, \bar \sigma,x_0,\bar W)$, each potentially having any one of the above irregularities, the synchronous coupling is optimal. Furthermore, one can compute the adapted Wasserstein distance by approximating each scalar SDE with any numerical scheme for which strong convergent rates are known.
	\end{mainresult*}
		
	Given that the synchronous coupling attains the adapted Wasserstein distance, the problem of computing this distance reduces to that of simulating two one-dimensional SDEs, using the same noise for each, and applying (multilevel-)Monte Carlo methods. Such computational efficiency is not usually expected, even in the case of classical optimal transport, without some regularisation. Moreover, this distance provides a computationally efficient measure of model uncertainty for optimal stopping problems.
		
	The methods of proof in our paper are based on transformation and discretisation of SDEs. In case \Cref{it:1}, for SDEs with discontinuous and exponentially growing drift, additionally satisfying a piecewise one-sided Lipschitz condition, we introduce a transformation-based semi-implicit Euler--Maruyama scheme for which we prove strong convergence. We use the transformation method developed in \citep{LeSz17} to handle the discontinuities, and combine it with a semi-implicit scheme to handle the super-linear growth. We design this numerical scheme such that it is stochastically monotone. For measures on $\R^n$ satisfying such a monotonicity condition, it is known that the well-studied Knothe--Rosenblatt rearrangement is an optimiser for the bicausal transport problem between them. Applying the strong convergence of the scheme and a stability result for bicausal optimal transport, we prove the optimality of the synchronous coupling. In case \Cref{it:2}, with bounded and measurable drift, we apply the Zvonkin transformation to obtain SDEs with sufficiently regular coefficients, for which optimality results are already known. We show that such a transformation preserves the optimality property of the synchronous coupling. The novel strong convergence result in case \Cref{it:1} can be stated informally as follows:

	\begin{auxresult*}
		For the SDE \eqref{eq:sde} with drift $b$ that has exponential growth and satisfies mild piecewise continuity conditions, and with diffusion $\sigma$ that may be degenerate except at the points of discontinuity of $b$, there exists a unique strong solution with bounded moments of all orders, and a transformation-based drift-implicit Euler--Maruyama scheme converges strongly with known rates.
	\end{auxresult*}	
	
	We conclude this section by commenting on the classes of SDEs covered by our assumptions. To the best of our knowledge, the only existing optimality result for SDEs with discontinuous drift is \cite[Proposition 3.30]{BaKaRo22}, which considers bounded coefficients with additional regularity and integrability conditions. Our case \Cref{it:2} extends the bounded coefficients case to the full generality of the existence and uniqueness result of \citet{Zv74}. For case \Cref{it:1}, we are able to treat unbounded coefficients with discontinuous drift and even possibly degenerate diffusion coefficient. While a boundedness assumption is artificial for many applications, allowing for discontinuous drift coefficients that may grow super-linearly permits us to treat important examples, such as the value of a dividend-paying firm in a Black--Scholes market; see \Cref{ex:super-linear}. For our optimality proof in this setting, we require strong existence and uniqueness, as well as a strongly converging numerical scheme. As described above, we also go beyond existing results in this direction. In the setting of one-dimensional Markovian coefficients, it is unknown whether optimality of the synchronous coupling extends beyond the assumptions considered in this paper, although some of our assumptions are required for well-posedness; see, e.g.\ \Cref{rem:non-degeneracy}.
	One may naturally ask whether our results can be further extended to path-dependent coefficients or multi-dimensional equations. In both cases, however, counterexamples to optimality of the synchronous coupling are given in \cite[Section 5]{BaKaRo22}.
		
		\subsection{Structure of the article}
			Since this work combines optimal transport and numerical analysis of SDEs, we keep it self-contained such that it is accessible to readers from both fields. We provide the necessary background on optimal transport theory in \Cref{sec:OT-intro} and its applications in the theory of stochastic processes in \Cref{sec:OT-bc}. In \Cref{sec:num-sdes}, we discuss numerical approximation schemes for SDEs and their application to bicausal optimal transport. We present some preliminary stability and optimality results in \Cref{sec:preliminaries}.
			
			We present our first theorem in \Cref{sec:growth-disc}, where we consider the first class of irregularity of the coefficients stated above, that is discontinuous drift with exponential growth and degenerate diffusion. We define a novel numerical scheme, prove its strong convergence, and provide convergence rates in \Cref{sec:implicit-transformed-em}. We prove optimality of the synchronous coupling for this class of coefficients in \Cref{sec:bcot-growth-disc}.
			
			In \Cref{sec:zvonkin} we consider the second class of irregularity of the coefficients stated above, that is bounded measurable drift and H\"older continuous diffusion. Applying the Zvonkin transformation, we also prove optimality of the synchronous coupling for this class of coefficients.
			
			We combine the results of the preceding sections in \Cref{sec:mixing-assumptions} to prove the main theorem of this paper.
			In \Cref{sec:numerics} we apply our results to compute adapted Wasserstein distances, and finally we present an application to the quantification of model uncertainty for stochastic optimisation problems in \Cref{sec:finance}.
			
			The proofs of several results from \Cref{sec:preliminaries} and \Cref{sec:growth-disc} are collected in \Cref{app:prelim} and \Cref{app:proofs-transform}, respectively.
			
			The relevant literature can be found in the respective sections.

\section{Optimal transport and applications to stochastic processes}\label{sec:OT}

	Here we introduce the concepts from optimal transport that will be needed for our applications. The field has a long history, originating with \citet{Monge} in 1781 and taken up again by \citet{Kant42} in 1942. For the interested reader, modern accounts can be found, for example, in \citet{Vi03}, \citet{AGS}, and \citet{Sa15}, with the latter taking a more applied point of view. 
	
	\subsection{Introduction to optimal transport}\label{sec:OT-intro}
	
		Given separable metric spaces $(\X, d)$ and $(\Y, \rho)$, let $\B(\X)$, $\B(\Y)$ denote the respective Borel sigma-algebras, and $\Pc(\X)$, $\Pc(\Y)$ the spaces of probability measures on $(\X, \B(\X))$, $(\Y, \B(\Y))$, respectively. Further, let $\Pc(\X \times \Y)$ denote the set of probability measures on the product space equipped with the product Borel sigma-algebra. For $\mu \in \Pc(\X)$, $\nu \in \Pc(\Y)$, define the set of \emph{couplings} $\cpl(\mu, \nu)$ by
		\begin{equation}
			\cpl(\mu, \nu) \coloneqq \{\, \pi \in \Pc(\X \times \Y) : \text{$\pi(A, \Y) = \mu(A)$, $\pi(\X, B) = \nu(B)$, for all $A \in \B(\X), B \in \B(\Y)$} \,\}.
		\end{equation}
		Consider a lower semicontinuous \emph{cost function} $C \colon \X \times \Y \to \R \cup \{\infty\}$.
		
		The \emph{optimal transport problem} \citep{Kant42} is to find
		\begin{equation}\label{eq:kantorovich}\tag{OT}
			\inf_{\substack{\pi \in \cpl(\mu, \nu)\\\Law(X, Y) = \pi}} \E[C(X, Y)],
		\end{equation}
		supposing sufficient conditions on $(\X, \B(\X), \mu)$, $(\Y, \B(\Y), \nu)$ and $C$ such that the expectation in \eqref{eq:kantorovich} is well defined for all $\pi \in \cpl(\mu, \nu)$.
		
		\subsubsection{Brenier's theorem}\label{sec:brenier}
			 Central to the theory of optimal transport is the following key result of Brenier \citep{Br91}: \emph{an optimiser for the quadratic cost is given by the gradient of a convex function}.
			 More precisely, let $\X = \Y = \R^n$ be equipped with the Euclidean metric $|\cdot|$ and set $c(x, y) \coloneqq |x - y|^2$. Suppose that $\mu$ is absolutely continuous with respect to the Lebesgue measure.
			 Then \eqref{eq:kantorovich} admits a unique optimiser $\pi^\mono$ and there exists a convex function $v\colon \R^n \to \R$ such that $\pi^\mono = (\id, \nabla v)_\#\mu$, the pushforward of the measure $\mu$ by the map $(\id, \nabla v)\colon \R^n \to \R^n \times \R^n$.
			 In the one-dimensional case, $\nabla v = F^{-1}_\nu \circ F_\mu$, where $F_\eta, F^{-1}_\eta$ denote the cumulative distribution function and quantile function of a measure $\eta$, and the coupling $\pi^\mono$ is known as the \emph{monotone rearrangement} (or the \emph{Hoeffding--Fr\'echet coupling}).
			 
		\subsubsection{Wasserstein distance}\label{sec:wasserstein}
			Let $\X = \Y$ be equipped with some metric $d$ such that $(\X, d)$ is a Radon space. Then, following \citep[Section 7.1]{AGS}, a metric is induced on the space of probability measures $\Pc_p(\X)$, for any $p \geq 1$, by taking the cost function $C = d^p$ in \eqref{eq:kantorovich}. For measures $\mu, \nu \in \Pc_p(\X)$ the $p$-\emph{Wasserstein distance} $\W_p(\mu, \nu)$ is defined by
			\begin{equation}\label{eq:wass}
				\W_p^p(\mu, \nu) \coloneqq \inf_{\substack{\pi \in \cpl(\mu, \nu)\\\Law(X, Y) = \pi}}\E[d(X, Y)^p].
			\end{equation}
			The $p$-Wasserstein distance metrises the weak topology on $\Pc_p(\X)$ (see \cite[Proposition 7.1.5]{AGS}) and is thus a natural candidate to compare probability measures on $\X$.
	
	\subsection{Optimal transport applied to stochastic processes}\label{sec:OT-bc}
	
		We now consider optimal transport between the laws of real-valued stochastic processes.
		First we treat discrete-time processes.
		
		Set $\X = \Y = \R^n$, for some $n \in \N$. Equip this space with the $L^p$ norm, defined by $\|x\|_p^p \coloneqq \sum_{i = 1}^n|x_i|^p$ for $x \in \R^n$, for some $p \geq 1$, and with the corresponding Borel sigma-algebra $\B(\R^n)$. Let $x = (x_1, \dotsc, x_n)$ denote the canonical process on $\R^n$, and $(x, \bar x)$ the canonical process on the product space $\R^n \times \R^n$. Let $\F = (\F_k)_{k \in \{1, \dotsc, n\}}$ denote the canonical filtration on $\R^n$, and note that $\F_n = \B(\R^n)$. Given a random variable $X$ on $(\R^n, \F_n)$ with law $\eta$, let  $\F^X = (\F^X_k)_{k \in \{1, \dotsc, n\}}$ denote the completion of its natural filtration with respect to $\eta$.
		
		As stated in \Cref{sec:wasserstein}, the Wasserstein distance is a natural candidate to measure the distance between two probability measures on $(\R^n, \F)$. However, in the case that such probability measures are the laws of stochastic processes, the usual Wasserstein distance (and associated weak topology) lacks some desirable properties. For example, \citet[Section 11]{Al81} already noted that deterministic processes can converge weakly to martingales. We illustrate this with the following example from \citet{BaBaBeEd19a}; cf.~\citet[Example 11.4]{Al81}. A continuous-time analogue of this example also appears in \citet[Example 5.1]{BaKaRo22}.
				
		\begin{example}\label{ex:usual}
			Consider the two-step real-valued processes $X, X^\varepsilon$, with laws $\mu, \mu^\varepsilon$, for each $\varepsilon > 0$, defined by
			\begin{equation}
				\begin{split}
					\P[(X^\varepsilon_1, X^\varepsilon_2) = (\varepsilon, 1)] = \P[(X^\varepsilon_1, X^\varepsilon_2) = (-\varepsilon, -1)] & = \frac12,\\
					\P[(X_1, X_2) = (0, 1)] = \P[(X_1, X_2) = (0, -1)] & = \frac12.
				\end{split}
			\end{equation}
			The paths of the processes are shown in \Cref{fig:usual-example}.
			\begin{figure}[h]
				\centering
				\begin{tikzpicture}[baseline={(0,0)}]

\definecolor{darkgray176}{RGB}{176,176,176}
\definecolor{darkorange25512714}{RGB}{255,127,14}
\definecolor{steelblue31119180}{RGB}{31,119,180}

\begin{axis}[
tick align=outside,
tick pos=left,
x grid style={darkgray176},
xmin=0.0, xmax=2.01,
xtick style={color=black},
xtick = {1.0, 2.0},
xticklabels = {\small{$1$}, \small{$2$}},
y grid style={darkgray176},
ymin=-1.05, ymax=1.01,
ytick style={color=black},
ytick = {-1.0, -0.1, 0.1, 1.0},
yticklabels = {\small{$-1$}, \small{$-\varepsilon$}, \small{$\varepsilon$}, \small{$1$}}
]
\draw[-, steelblue31119180, thick](0.0, 0.0) -- (1.0, 0.1) -- (2.0, 1.0);
\draw[dashed, darkorange25512714, thick](0.0, 0.0) -- (1.0, -0.1) -- (2.0, -1.0);
\end{axis}

\coordinate [label = {$X^\varepsilon$}] (X) at (2.0, 2.8);

\end{tikzpicture} \quad \begin{tikzpicture}[baseline={(0,0)}]

\definecolor{darkgray176}{RGB}{176,176,176}
\definecolor{darkorange25512714}{RGB}{255,127,14}
\definecolor{steelblue31119180}{RGB}{31,119,180}

\begin{axis}[
tick align=outside,
tick pos=left,
x grid style={darkgray176},
xmin=0.0, xmax=2.01,
xtick style={color=black},
xtick = {1.0, 2.0},
xticklabels = {\small{$1$}, \small{$2$}},
y grid style={darkgray176},
ymin=-1.05, ymax=1.01,
ytick style={color=black},
ytick = {-1.0, 0, 1.0},
yticklabels = {\small{$-1$}, \small{$0$}, \small{$1$}}
]
\draw[-, steelblue31119180, thick](0.0, 0.0) -- (1.0, 0.0) -- (2.0, 1.0);
\draw[dashed, darkorange25512714, thick](0.0, -0.0) -- (1.0, -0.0) -- (2.0, -1.0);
\end{axis}

\coordinate [label = {$X$}] (X) at (2.0, 2.8);

\end{tikzpicture}
				\caption{A sequence of two-step deterministic processes $X^\varepsilon$ converges weakly as $\varepsilon \to 0$ to a martingale $X$ \citep[Figure 1]{BaBaBeEd19a}.}
				\label{fig:usual-example}
			\end{figure}
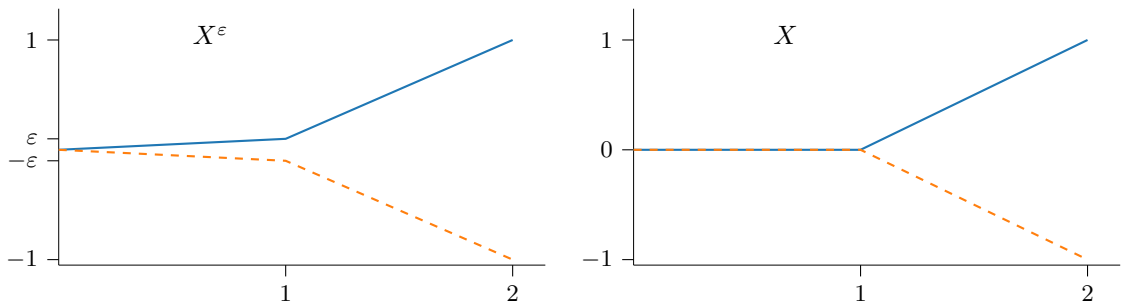
			Define $\pi^\ast \in \cpl(\mu^\varepsilon, \mu)$ such that positive (resp.~negative) values of $X^\varepsilon$ are coupled with  positive (resp.~negative) values of $X$; that is the solid blue paths in \Cref{fig:usual-example} are coupled with each other, and the dashed orange paths are coupled with each other. Then we find that, for $\Law(X^\varepsilon, X) = \pi^\ast$,
			\begin{equation}
				\W_p^p(\mu^\varepsilon, \mu) \leq \sum_{i = 1}^2\E[|X^\varepsilon_i - X_i|^p] = \varepsilon^p,
			\end{equation}
			implying the weak convergence $\mu^\varepsilon\rightharpoonup\mu$ as $\varepsilon \to 0$.
			
			On the other hand, consider the values of optimal stopping problems $V^\varepsilon \coloneqq \sup_\tau \E[X^\varepsilon_\tau]$ and $V \coloneqq \sup_\tau \E[X_\tau]$, where the suprema are taken over stopping times taking values in $\{1, 2\}$. Since $X$ is a martingale centred at $0$, we immediately see that $V = 0$. However, in the case of $X^\varepsilon$, the optimal stopping time is $2$ in the case that $X^\varepsilon_1 > 0$, and $1$ otherwise, and so we obtain
			\begin{equation}
				V^\varepsilon = \frac12 (1 - \varepsilon) \to \frac12.
			\end{equation}
			Thus we see that the processes $X^\varepsilon$ and $X$ have a very different information structure, but nevertheless their laws are close in the usual Wasserstein distance. This example motivates the introduction of the \emph{adapted Wasserstein distance} and the associated \emph{bicausal optimal transport} problem.
		\end{example}
		
		\subsubsection{Bicausal optimal transport}
			We have seen that the drawback of the Wasserstein distance for comparing stochastic processes is that this distance is not able to distinguish different information structures. To rectify this shortcoming, we consider Wasserstein distances with the additional constraint that the couplings must respect the flow of information. This constraint has been formalised in various ways, including
			the \emph{Markov constructions} of \citet{Ru85}, the \emph{nested distances} of \citet{PfPi12, PfPi14}, a stochastic control problem of \citet{BiTa19}, the \emph{causal transport plans} of \citet{La18}, and the \emph{adapted Wasserstein distance} introduced by \citet{BaBaBeEd19a}. Here we will use the terminology \emph{adapted Wasserstein distance}, which we define below. Although there are other candidates for distances between stochastic processes, many such distances have been shown to be topologically equivalent to the adapted Wasserstein distance; see, for example, \citep{BaBaBeEd19b, BoLiOb23}.
			
			Analogously to the classical Wasserstein distance, the adapted Wasserstein distance is induced by a \emph{bicausal optimal transport} problem.			
			\begin{definition}
			\label{def:bicausal}
				Let $\mu, \nu$ be probability measures on $(\R^n, \F)$ and let $\pi \in \cpl(\mu, \nu)$. The coupling $\pi$ is \emph{causal} if, for any random variables $X,Y$ with $\Law(X, Y) = \pi$, we have the conditional independence
				\begin{equation}\label{eq:causality}
					\text{$\F^Y_k$ is independent of $\F^X_n$ under $\pi$ conditional on $\F^X_k$},
				\end{equation}
				for all $k \in \{1, \dotsc, n\}$. If $\pi \in \cpl(\mu, \nu)$ is causal, and the coupling $\theta_\#\pi \in \cpl(\nu, \mu)$ is causal, where $\theta(x, y) = (y, x)$, then we say that $\pi$ is \emph{bicausal}. We write $\cplba(\mu, \nu)$ for the set of bicausal couplings.
			\end{definition}
			
			In words, $\pi = \Law(X, Y)$ is causal if the present value of $Y$ is independent of the future of the process $X$, conditional on the past of $X$.
			
			We assume that the cost function $C \colon \R^n \times \R^n \to \R$ in \eqref{eq:kantorovich} takes the following form. Suppose that there exist continuous functions $c_k \colon \R \times \R \to \R$, for $k \in \{1, \dotsc, n\}$ such that $C(x, y) = \sum_{k = 1}^n c_k(x_k, y_k)$, for any $x = (x_1, \dotsc, x_n), y = (y_1, \dotsc, y_n) \in \R^n$. For each $k \in \{1, \dotsc, n\}$, suppose that the function $c_k$ has \emph{polynomial growth}; i.e.~for some $p \geq 1$, there exists $K \geq 0$ such that, for all $x, y \in \R$, $k \in \{1, \dotsc, n\}$,
			\begin{equation}\label{eq:poly-growth}
				|c_k(x,y)|\leq K[1+|x|^p+|y|^p],
			\end{equation}
			and that $c_k$ is \emph{quasi-monotone}\footnote{also called \emph{L-superadditive} or \emph{supermodular}; see \citet[Remark 1.1]{BlGrSa89}} in the sense that
			\begin{equation}\label{eq:superadditivity}
				c_k(x, y) + c_k(x^\prime, y^\prime) - c_k(x, y^\prime) - c_k(x^\prime, y) \geq 0, \quad \text{for all} \; x \leq x^\prime, \, y \leq y^\prime.
			\end{equation}
			
			For $\mu, \nu \in \Pc_p(\R^n)$, the \emph{bicausal optimal transport problem} is then to find \eqref{eq:kantorovich} with the usual set of couplings $\cpl(\mu, \nu)$ replaced by the set of bicausal couplings $\cplba(\mu, \nu)$, i.e.
				\begin{equation}\label{eq:bcot}
					\inf_{\pi \in \cplba(\mu, \nu)}\E^\pi\!\left[\sum_{k = 1}^n c_k(x_k, \bar x_k)\right]\!.
				\end{equation}
			Taking each function $c_k$ in \eqref{eq:bcot} to be the $p$\textsuperscript{th} power of the Euclidean distance, we define an adapted analogue of the Wasserstein distance \eqref{eq:wass} as follows.
			
			\begin{definition}[adapted Wasserstein distance -- discrete time]\label{def:disc-aw}
				Let $p \geq 1$ and $\mu, \nu \in \Pc_p(\R^n)$. Then the \emph{adapted Wasserstein distance} $\AW_p(\mu, \nu)$ is defined by
				\begin{equation}
					\AW_p^p(\mu, \nu) \coloneqq \inf_{\pi \in \cplba(\mu, \nu)}\E^\pi\!\left[\sum_{k = 1}^n|x_k - \bar x_k|^p\right]\!.
				\end{equation}
			\end{definition}
			
			\begin{remark}
			The term \emph{adapted} appears natural here, since \Cref{def:bicausal} constrains the usual set of couplings to those that satisfy a condition analogous to adaptedness of stochastic processes. The term \emph{causal}, introduced in \citep{La18}, does not necessarily correspond to the notion of causality in statistics; for example, the product coupling is bicausal. On the other hand, \citet{ChEc23} relate causal optimal transport to the causal structure of graphical models.
			\end{remark}
			
			\begin{example}[\Cref{ex:usual} revisited]\label{ex:revisited}
				For the laws given in \Cref{ex:usual}, the adapted Wasserstein distance $\AW_p(\mu^\varepsilon, \mu)$ is bounded away from $0$, uniformly in $\varepsilon > 0$, for any $p \geq 1$. Indeed, under any $\pi \in \cplba(\mu, \mu^\varepsilon)$, $\F^{X^\varepsilon}_1$ is independent of $\F^{X}_2$ conditional on $\F^{X}_1$. Since $\F^{X}_1$ is trivial, $\F^{X^\varepsilon}_1$ and $\F^{X}_2$ are independent under $\pi$, without any conditioning. Moreover, $X^\varepsilon_2$ is fully determined by $X^\varepsilon_1$, so $\F^{X^\varepsilon}_2 = \F^{X^\varepsilon}_1$. Thus $\{\F^{X^\varepsilon}_k\}_{k \in \{1, 2\}}$ is independent of $\F^X_2$ under $\pi$ and, using again that $\F^X_1$ is trivial, is independent of $\{\F^{X}_k\}_{k \in \{1, 2\}}$. Hence the product coupling is the only element of $\cplba(\mu, \mu^\varepsilon)$, and so $\AW_p^p(\mu^\varepsilon, \mu)  = \varepsilon^p + 2^{p-1} > 2^{p-1}$.
				Recalling that, for the optimal stopping problems in \Cref{ex:usual}, $V^\varepsilon \not \to V$ as $\varepsilon \to 0$, the adapted Wasserstein distance appears to be better suited to comparing stochastic processes than the usual Wasserstein distance.
			\end{example}
			
			Stability of various optimisation problems with respect to the adapted Wasserstein distance is shown in \citet{BaBaBeEd19a}. Moreover, optimal stopping problems are in fact continuous with respect to the adapted Wasserstein distance by \citep[Lemma 7.1]{BaBaBeEd19b}; cf.~\citep[Theorem 17.2]{Al81}. This confirms our intuition from \Cref{ex:usual,ex:revisited}.
		
		\subsubsection{Optimality of bicausal couplings}
		
			Under the following condition of stochastic monotonicity, \citet{Ru85} identified an optimiser for a bicausal transport problem between the laws of discrete-time Markov processes. For a measure $\eta$ on $\R^n$, denote by $\eta_1$ the first one-dimensional marginal of $\eta$ and, for each $k \in \{2, \dotsc, n\}$, denote by $\eta_{x_1, \dotsc, x_{k-1}}$ the $k$\textsuperscript{th} conditional marginal of $\eta$, given that the first $k-1$ coordinates are equal to $(x_1, \dotsc, x_{k -1})$.
			
			\begin{definition}[stochastic co-monotonicity \citep{Da68}]\label{def:stoch-dom}
				Let $(X_n)_{n \in \{1, \dotsc, N\}}$ be a real-valued Markov process with law $\mu \in \Pc(\R^n)$. The process $X$, or the measure $\mu$, is said to be \emph{stochastically increasing} (resp.\ \emph{decreasing}) if $x_k \mapsto \mu_{x_1, \dotsc, x_{k - 1}, x_k}$ is increasing (resp.\ decreasing) in first order stochastic dominance, for each $k \in \{1, \dotsc, N - 1\}$.
				Two such processes or measures are \emph{stochastically co-monotone} if they are both stochastically increasing, both stochastically decreasing, or one is stochastically increasing and decreasing while the other is arbitrary.
			\end{definition}
			
			As we shall see below, an optimal bicausal coupling under this monotonicity condition is given by the Knothe--Rosenblatt rearrangement, introduced independently by \citet{Ro52} and \citet{Kn57}. This coupling is a multi-dimensional extension of the monotone rearrangement described in \Cref{sec:brenier}. Under the Knothe--Rosenblatt rearrangement, the first marginals are coupled monotonically and then, inductively, the conditional marginals are coupled monotonically; see \citep[Figure 1]{BaKaRo22} for an illustration.
			
			\begin{definition}[Knothe--Rosenblatt rearrangement]\label{def:kr}
					Given probability measures $\mu, \nu$ on $\R^n$, let $U_1$, $\dotsc\,$, $U_n$ be independent uniform random variables on $[0, 1]$, define $X_1 = F_{\mu_1}^{-1}(U_1)$, $Y_1 = F_{\nu_1}^{-1}(U_1)$ and, for $k \in \{2, \dotsc, n\}$, define inductively the random variables
				\begin{equation}
					X_k = F^{-1}_{\mu_{X_1, \dotsc, X_{k - 1}}}(U_k), \quad Y_k = F^{-1}_{\nu_{Y_1, \dotsc, Y_{k - 1}}}(U_k).
				\end{equation}
				The Knothe--Rosenblatt rearrangement between the marginals $\mu$ and $\nu$ is then given by 
				\[
				\pi^\kr_{\mu, \nu} = \Law(X_1, \dotsc, X_n, Y_1, \dotsc, Y_n).
				\]
			\end{definition}
			We will make use of the following equivalent definition, which is proved in \citet[Proposition 5.8]{BaBeLiZa17}. Similarly to \Cref{def:stoch-dom} of stochastic co-monotonicity, we say that two functions on $\R$ are \emph{co-monotone} if they are both increasing or both decreasing.
			\begin{proposition}\label{rem:kr-equivalent-def}
				A coupling $\pi \in \cpl(\mu, \nu)$ is equal to $\pi^\kr_{\mu, \nu}$ if and only if $\pi = \Law(X, Y)$, where
				\begin{align}
						X & = \bigl(T_1(U_1), T_2(U_2; X_1), \dotsc, T_n(U_n; X_1, \dotsc, X_{n - 1})\bigr),\\
						Y & = \bigl(S_1(U_1), S_2(U_2; Y_1), \dotsc, S_n(U_n; Y_1, \dotsc, Y_{n - 1})\bigr),
				\end{align}
				for $U_1$, $\dotsc\,$, $U_n$ independent uniform random variables on $[0, 1]$, and functions $T_i, S_i \colon \R^i \to \R$ that are co-monotone in $U_i$, for each $i \in \{1, \dotsc, n\}$.
			\end{proposition}
			
			The following optimality result is a straightforward modification of \citep[Proposition 3.5]{BaKaRo22}, which generalises earlier results from \citet{BaBeLiZa17, Ru85}. We note that, although \citep[Proposition 3.5]{BaKaRo22} considers a single function $c \equiv c_k$ for all $k \in \{1, \dotsc, n\}$, only minor changes to the proof are required for our setting and we do not repeat the proof here.
			
			\begin{proposition}\label{prop:kr-optimal}
				For $p \geq 1$, $n \in \N$, let $\mu, \nu \in \Pc_p(\R^n)$ be the laws of Markov processes. Consider the bicausal optimal transport problem \eqref{eq:bcot} with continuous functions $c_k \colon \R \times \R \to \R$, $k \in \{1, \dotsc, n\}$, satisfying \eqref{eq:poly-growth} and \eqref{eq:superadditivity}. If $\mu$ and $\nu$ are stochastically co-monotone, then the Knothe--Rosenblatt rearrangement attains the infimum in \eqref{eq:bcot}.
			\end{proposition}
			
			\begin{remark}
				We show in \Cref{ex:second-order-mono} that, when first order stochastic dominance is replaced by second order stochastic dominance for one of the marginals in the definition of stochastic co-monotonicity, optimality of the Knothe--Rosenblatt rearrangement may fail.
			\end{remark}
		
		\subsubsection{Continuous-time bicausal optimal transport}
			The notions of bicausality and adapted Wasserstein distance extend to continuous-time processes.
			
			Fix a finite time horizon $T > 0$ and let $\X = \Y = \Omega \coloneqq C([0, T], \R)$ be the space of real-valued continuous paths on the time interval $[0, T]$. Analogously to the discrete-time setting, equip the space $\Omega$ with the $L^p$ norm, defined by $\|\gamma\|_p^p \coloneqq \int_0^T|\gamma_t|^p \D t$ for $\gamma\in \Omega$, and with the corresponding Borel sigma-algebra $\B(\Omega)$. Let $\omega = (\omega_t)_{t \in [0, T]}$ denote the canonical process on $\Omega$ and $\F = (\F_t)_{t \in [0, T]}$ the canonical filtration, noting that $\F_T = \B(\Omega)$. Let $(\omega, \bar \omega) = ((\omega_t, \bar \omega_t))_{t \in [0, T]}$ denote the canonical process on the product space $\Omega \times \Omega$, and equip this space with the product filtration. As in the discrete-time case, for a random variable $X$ on $(\Omega, \F_T)$ with law $\eta$, the completion of its natural filtration under $\eta$ is denoted by $\F^X = (\F^X_t)_{t \in [0, T]}$.
			
			Let $\mu, \nu \in \Pc_p(\Omega)$, for some $p \geq 1$. We say that a coupling $\pi \in \cpl(\mu, \nu)$ is bicausal if it satisfies the continuous-time analogue of \Cref{def:bicausal}; i.e.~$\pi$ is causal if \eqref{eq:causality} holds for all $t \in [0, T]$, for any random variables $X, Y$ with $\pi = \Law(X, Y)$, and $\pi$ is bicausal if \eqref{eq:causality} also holds with the roles of $X$ and $Y$ reversed. Again we write $\cplba(\mu, \nu)$ for the set of bicausal couplings. Note that, by \citep[Example 1(v), Lemma 4, and Proposition 2]{La18}, our definition is equivalent to \citet[Definition 1.2]{BaBaBeEd19a} and \citet[Definition 1]{La18}.
			
			In the continuous-time setting, we consider cost functions $C \colon \Omega \times \Omega \to \R$ in \eqref{eq:kantorovich} that take the form $C(\omega, \bar \omega) = \int_0^T c_t(\omega_t, \bar \omega_t) \D t$, for any $\omega, \bar \omega \in \Omega$, where $c \colon [0, T] \times \R \times \R \to \R$ is a measurable function. Additionally, we make the following assumption.
			
			\begin{assumption}\label{ass:cost}
				Suppose that $c \colon [0, T] \times \R \times \R \to \R$ is continuous on $[0, T] \times M$ for any compact set $M \subseteq \R \times \R$, and suppose that, for some $p \geq 1$, there exists $K \geq 0$ such that, for each fixed $t \in [0, T]$, $c_t(\cdot, \cdot) \colon \R \times \R \to \R$ satisfies \eqref{eq:poly-growth} and \eqref{eq:superadditivity}.
			\end{assumption}
			For $\mu, \nu \in \Pc_p(\Omega)$, the continuous-time \emph{bicausal optimal transport problem} between $\mu$ and $\nu$ is to find the value
			\begin{equation}\tag{BCOT}\label{eq:bcot-cont}
				V_c(\mu, \nu) \coloneqq \inf_{\pi \in \cplba(\mu, \nu)}\E^\pi\biggl[\int_0^T c_t(\omega_t, \bar \omega_t)\D t\biggr].
			\end{equation}
			We define the adapted Wasserstein distance in the continuous-time setting as follows.
			\begin{definition}[adapted Wasserstein distance -- continuous time]\label{def:aw}
			For $p \geq 1$ and $\mu, \nu \in \Pc_p(\Omega)$, the \emph{adapted Wasserstein distance} $\AW_p(\mu, \nu)$ is defined by
				\begin{equation}\label{eq:adapted-wasserstein}
					\AW_p^p(\mu, \nu) \coloneqq \inf_{\pi \in \cplba(\mu, \nu)}\E^\pi\biggl[\int_0^T |\omega_t - \bar \omega_t |^p \D t\biggr].
				\end{equation}
			\end{definition}
			The adapted Wasserstein distance $\AW_2$, in essentially the form of \Cref{def:aw}, was first studied by \citet{BiTa19} via a stochastic control formulation and associated Hamilton--Jacobi--Bellman equation. \citet{BaKaRo22} studied $\AW_p$ as defined in \Cref{def:aw}, for general $p \geq 1$. We will see in \Cref{sec:numerics} that this distance is amenable to efficient computation, and in \Cref{sec:finance} we show that this distance can be applied to robust stochastic optimisation. For semimartingales, \citet{BaBaBeEd19a} consider an alternative definition of adapted Wasserstein distance based on the Doob decomposition. While this definition also yields stability of various stochastic optimisation problems, we are not aware of efficient methods for the numerical computation of such a distance.
			
			We consider the problem \eqref{eq:bcot-cont} between the laws of stochastic processes that are described by SDEs. To the best of our knowledge, for such continuous-time problems, \citet{BiTa19} and \citet{BaBaBeEd19a} were the first to find an explicit optimiser, the so-called \emph{synchronous coupling}, which we define after the following preparations.
			
			\begin{definition}
				For an adapted process $\rho = (\rho_t)_{t \in [0, T]}$ taking values in $[-1, 1]$, we call $(W, \bar W)$ a \emph{$\rho$-correlated Brownian motion} if $W$, $\bar W$ are each one-dimensional Brownian motions and, for each $t \in [0, T]$, their correlation is given by $\corr(W_t, \bar W_t) = \rho_t$.
			\end{definition}
			
			Let $b, \bar b \colon [0, T] \times \R \to \R$ and $\sigma, \bar \sigma \colon [0, T] \times \R \to [0, \infty)$ be measurable functions, and let $x_0 \in \R$. Suppose that there exists a unique strong solution $(X, \bar X)$ of the system of SDEs
			\begin{equation}\label{eq:sde-correlated}
				\begin{cases}
					\ds X_t = b_t(X_t) \D t + \sigma_t(X_t) \D W_t, & X_0 = x_0, \\
					\ds \bar X_t = \bar b_t(\bar X_t) \D t + \bar \sigma_t(\bar X_t) \D \bar W_t, & \bar X_0 = x_0,
				\end{cases}
			\end{equation}
			on the time interval $[0, T]$, where $(W, \bar W)$ is a correlated Brownian motion, and write $\mu, \nu$ for the laws of $X, \bar X$, respectively. It is shown in \citep[Proposition 2.2]{BaKaRo22} that the set of bicausal couplings $\cplba (\mu, \nu)$ is equal to the set of all couplings of the form $\pi = \Law(X, \bar X)$, where \eqref{eq:sde-correlated} is driven by some correlated Brownian motion. In particular, the synchronous coupling is defined as follows.
			
			\begin{definition}\label{def:synchronous}
				 Let $(X, \bar X)$ be the unique strong solution of \eqref{eq:sde-correlated} with $W = \bar W$. Then the \emph{synchronous coupling} is defined as $\pi^\sync_{\mu, \nu} \coloneqq \Law(X, \bar X)$, the joint law of the solutions when the SDEs are driven by a common Brownian motion.
			\end{definition}
			
		\subsection{Numerical methods for SDEs}\label{sec:num-sdes}
			
			In order to solve SDEs of the form \eqref{eq:sde}, one usually has to resort to numerical methods. The simplest of these is the \emph{Euler--Maruyama scheme}. Let $N \in \N$ and $h = T/N$. We suppose here and throughout the paper that $N > T$ so that $h < 1$. The Euler--Maruyama scheme $\bigl(X^{h, \EM}_{kh}\bigr)_{k \in \{0, \dotsc, N\}}$ is defined by $X^{h, \EM}_0 = x_0$ and, for all $k \in \{0, \dotsc, N - 1\}$
		   \begin{equation}\label{eq:em}
		           X^{h,\EM}_{(k+1)h} = X^{h,\EM}_{kh} + b_{kh}\bigl(X^{h,\EM}_{kh}\bigr) h + \sigma_{kh}\bigl(X^{h,\EM}_{kh}\bigr)\Delta W_{k + 1},
		   \end{equation}
		   where $\Delta W_{k + 1} \coloneqq W_{(k + 1)h} - W_{kh} \sim \mathcal{N}(0,h)$. One can also define a \emph{semi-implicit} Euler--Maruyama scheme $\bigl(X^{h, \IEM}_{kh}\bigr)_{k \in \{0, \dotsc, N\}}$ by $X^{h, \EM}_0 = x_0$ and, for all $k \in \{0, \dotsc, N - 1\}$,
		  \begin{equation}\label{eq:iem}
		           X^{h,\IEM}_{(k+1)h} = X^{h,\IEM}_{kh} + b_{kh}\bigl(X^{h,\IEM}_{(k+1)h}\bigr) h + \sigma_{kh}\bigl(X^{h,\IEM}_{kh}\bigr)\Delta W_{k + 1}.
		   \end{equation}
		   
		   Under global Lipschitz conditions on the coefficients, both of the above schemes converge in the strong $L^p$-sense to the true solution of the SDE; see, for example, \citet{KlPl92}.
		   In recent years, numerical methods for SDEs with irregular (non-globally Lipschitz continuous) coefficients have been studied intensively. We refer to \Cref{rem:growth-disc,rem:zvonkin} and \Cref{ex:num-disc-drift,ex:cir} for references to relevant results.
		
			The \emph{monotone Euler--Maruyama scheme} is introduced in \citep[Definition 3.13]{BaKaRo22}. This scheme is defined analogously to \eqref{eq:em} with the Brownian increments $\Delta W_{k + 1}$ replaced by \emph{truncated Brownian increments} $\Delta W^h_{k + 1}$, defined in \Cref{def:trunc-bm} below. Similarly, we define the \emph{monotone semi-implicit Euler--Maruyama scheme} by replacing the Brownian increments in \eqref{eq:iem} by truncated Brownian increments.
		
			\begin{definition}[{truncated Brownian motion; cf.~\cite[\S 2.1]{MiReTr02}}]\label{def:trunc-bm}
				Let $(W_t)_{t \in [0, T]}$ be a standard Brownian motion, let $N \in \N$ and set $h = T/N$. Define $A_h = 2 \sqrt{- h \log h}$ and, for each $k \in \{0, \dotsc, N - 1\}$, define the stopping time $\tau^h_k \coloneqq \inf\{\, t \geq kh :  |W_t - W_{kh}| > A_h \,\}$. The \emph{truncated Brownian motion} $(W^h_t)_{t \in [0, T]}$ is defined by $W^h_0 = 0$ and $W^h_t \coloneqq W^h_{kh} + W_{t \wedge \tau^h_k} - W_{kh}$, for $t \in (kh, (k+1)h]$, $k \in \{0, \dotsc, N - 1\}$. The \emph{truncated Brownian increments} are defined by $\Delta W^h_{k + 1} \coloneqq W^h_{(k + 1)h} - W^h_{kh} = W_{(k+1)h \wedge \tau^h_k} - W_{kh}$, $k \in \{0, \dotsc, N - 1\}$.
			\end{definition}
			
			If the coefficients $b$ and $\sigma$ of the SDE \eqref{eq:sde} are time-homogeneous and globally Lipschitz continuous, \cite[Lemma 3.14]{BaKaRo22} shows that, in contrast to \eqref{eq:em}, the monotone Euler--Maruyama scheme is stochastically increasing for sufficiently small step size. The scheme also converges strongly to the solution of \eqref{eq:sde}; in particular, \cite[Proposition 3.16]{BaKaRo22} shows that the linear interpolation of the monotone Euler--Maruyama scheme converges in $L^p$ uniformly in time. Combined with a stability argument, these properties are key to the proof of optimality of the synchronous coupling in \cite[Proposition 3.19]{BaKaRo22}. Similar numerical schemes based on truncated Brownian increments are also studied in \citet{LiPa22} and \citet{JoPa23}, with applications to (monotone) convex ordering of SDEs. 
			
			A different perspective on the interplay between optimal transport and numerical methods for SDEs is given by \citet{Da22}.
	
\section{Preliminary results}\label{sec:preliminaries}
	
	We present two general stability results for bicausal optimal transport, which we will use in our optimality proofs. The first allows us to pass from discrete to continuous time, and the second allows us to approximate the coefficients of SDEs. As an immediate consequence, we obtain a first generalisation of \cite[Theorem 1.3]{BaKaRo22}, which we will later apply in combination with a transformation of SDEs in order to prove \Cref{thm:zvonkin-optimality}. The proofs of the results in this preliminary section are presented in \Cref{app:prelim}.
	
	\begin{assumption}\label{ass:comon-scheme}
		Suppose that $(X, \bar X)$ and $(X^h, \bar X^h)$ satisfy the following assumptions:
		\begin{enumerate}[label = (\roman*)]
			\item $(X, \bar X)$ are the unique strong solutions of the SDEs \eqref{eq:sde-correlated} driven by some correlated Brownian motion $(W, \bar W)$, where $X = (X_t)_{t \in [0, T]}$, $\bar X = (\bar X_t)_{t \in [0, T]}$, with $X_0 = \bar X_0 = x_0$;
			\item for any $N \in \N$ with $N > T$ and $h = T/N$, $X^h = (X_{kh})_{k \in \{0, \dotsc, N\}}$ and $\bar X^h = (\bar X_{kh})_{k \in \{0, \dotsc, N\}}$, with $X^h_0 = \bar X^h_0 = x_0$;
			\item \label[ass]{ass:it-2} there exist measurable maps
				\begin{equation}
					\phi, \bar \phi \colon (0, 1) \times \N \times \R \times \R \to \R, \quad \text{and} \quad f \colon (0, 1) \times \N \times C([0, T], \R) \to \R,
				\end{equation}
				such that $\phi$ and $\bar \phi$ are co-monotone in their final argument and, for any correlated Brownian motion $(W, \bar W)$, any $N \in \N$ with $N > T$ and $h = T/N$, and any $k \in \{0, \dotsc, N - 1\}$,
				\begin{equation}
					X^h_{(k + 1)h} = \phi\bigl(h, k, X^h_{kh}, f(h, k, W)\bigr), \quad \bar X^h_{(k + 1)h} = \bar \phi\bigl(h, k, \bar X^h_{kh}, f(h, k, \bar W)\bigr);
				\end{equation}
			\item \label[ass]{ass:it-1} there exists $N_0 \in \N$ with $N_0 > T$ such that, for any $N \in \N$ with $N \geq N_0$ and $h = T/N$, $X^h$ and $\bar X^h$ are stochastically co-monotone (see \Cref{def:stoch-dom});
			\item \label[ass]{ass:it-3} in the case that $(X, \bar X)$ and $(X^h, \bar X^h)$ are driven by the same correlated Brownian motion $(W, \bar W)$, for some $p \geq 1$, the linear interpolation of $X^h$ (resp.~$\bar X^h$) converges in $L^p$ to $X$ (resp.~$\bar X$); i.e.~
				\begin{align}
					\lim_{h \to 0}\E\!\left[\sum_{k = 1}^N\int_{(k-1)h}^{kh}|X^h_{kh} - X_s|^p\D s\right]\! =  \lim_{h \to 0}\E\!\left[\sum_{k = 1}^N\int_{(k-1)h}^{kh}|\bar X^h_{kh} - \bar X_s|^p\D s\right]\! = 0.
				\end{align}
		\end{enumerate}
	\end{assumption}
	
	If $(X, \bar X)$ and $(X^h, \bar X^h)$ satisfy \Cref{ass:comon-scheme}, then we say that $X^h$ (resp.~$\bar X^h$) is a numerical scheme for $X$ (resp.~$\bar X$).
	
	\begin{example}\label{ex:comon-scheme}
		Suppose that $(b, \sigma)$ are time-homogeneous and Lipschitz.
		\begin{enumerate}[label = (\roman*)]
			\item The monotone Euler--Maruyama scheme $X^h$ for $X$ corresponds to taking $\phi$, $f$ in \Cref{ass:comon-scheme} such that, for any Brownian motion $B$, any $N \in \N$ sufficiently large, $h = T/N$ and $k \in \{0, \dotsc, N - 1\}$, $f(h, k, B) = \Delta B^h_{k + 1}$, the truncated Brownian increment defined in \Cref{def:trunc-bm}, and $$\phi\bigl(h, k, x, f(h, k, B)\bigr) = x + b(x) h + \sigma(x) f(h, k, B).$$
			\item The semi-implicit monotone Euler--Maruyama scheme $X^h$ for $X$ corresponds to taking $\phi$, $f$ in \Cref{ass:comon-scheme} such that, for any Brownian motion $B$, any $N \in \N$ sufficiently large, $h = T/N$ and $k \in \{0, \dotsc, N - 1\}$, $f(h, k, B) = \Delta B^h_{k + 1}$ and $$\phi\bigl(h, k, x, f(h, k, B)\bigr) = (\id - b h)^{-1}\bigl(x + \sigma(x) f(h, k, B)\bigr).$$
		\end{enumerate} 
	\end{example}
	
	\begin{proposition}\label{prop:general-optimality}
		Suppose that $(X, \bar X)$ and $(X^h, \bar X^h)$ satisfy \Cref{ass:comon-scheme} for some $N_0 \in \N$ and $p \geq 1$, and write $\mu = \Law(X)$, $\nu = \Law(\bar X)$, $\mu^h = \Law(X^h)$, $\nu^h = \Law(\bar X^h)$. Let $c \colon [0, T] \times \R \times \R \to \R$ satisfy \Cref{ass:cost} with the same value of $p$. Then
		\begin{equation}\label{eq:limit-bcot}
			\lim_{h \to 0}\inf_{\pi^h \in \cplba(\mu^h, \nu^h)} \E^{\pi^h} \biggl[h\sum_{k = 1}^Nc_{kh}(x_k, y_k)\biggr] = \inf_{\pi \in \cplba(\mu, \nu)}\E^\pi\!\left[\int_0^Tc_t(\omega_t, \bar \omega_t) \D t\right]\!,
		\end{equation}
		and the synchronous coupling $\pi^\sync_{\mu, \nu}$ attains the infimum on the right hand side.
	\end{proposition}
	
	\begin{lemma}\label{lem:stability}
		Let $b, \bar b, b^N, \bar b^N \colon [0, T] \times \R \to \R$ and $\sigma, \bar \sigma, \sigma^N, \bar \sigma^N \colon [0, T] \times \R \to [0, \infty)$ be measurable functions with uniform linear growth in space. For each $t \in [0, T]$, suppose that  $b_t, \bar b_t, \sigma_t, \bar \sigma_t$ are continuous in space and that the following convergence holds uniformly on compact sets as $N \to \infty$:
		\begin{equation}
			b^N_t \to b_t, \; \bar b^N_t \to \bar b_t, \; \sigma^N_t \to \sigma_t, \; \bar \sigma^N_t \to \bar \sigma_t.
		\end{equation}
		Suppose moreover that, for coefficients $(b, \sigma)$, $(\bar b, \bar \sigma)$, there exist unique strong solutions $(X, \bar X)$ of the SDEs \eqref{eq:sde-correlated}, and that for coefficients $(b^N, \sigma^N)$, $(\bar b^N, \bar \sigma^N)$, $N \in \N$, there exist unique strong solutions $(X^N, \bar X^N)$ of the SDEs \eqref{eq:sde-correlated}, for each $N \in \N$. Write $\mu = \Law(X), \nu = \Law(\bar X)$ and $\mu^N = \Law(X^N), \nu^N = \Law(\bar X^N)$.
		
		Let $c \colon [0, T] \times \R \times \R \to \R$ satisfy \Cref{ass:cost} and suppose that, for each $N \in \N$, $\pi^\sync_{\mu^N, \nu^N}$ attains the infimum $V_c(\mu^N, \nu^N)$ in \eqref{eq:bcot-cont}.
		Then
		\begin{equation}
			\lim_{N \to \infty} V_c(\mu^N, \nu^N) = V_c(\mu, \nu),
		\end{equation}
		and $\pi^\sync_{\mu, \nu}$ attains the infimum $V_c(\mu, \nu)$.
	\end{lemma}
	
	Under the following assumption on both $(b, \sigma)$ and $(\bar b, \bar \sigma)$, the unique strong solutions of \eqref{eq:sde-correlated} and their corresponding monotone Euler--Maruyama schemes satisfy \Cref{ass:comon-scheme} (cf.~\Cref{ex:comon-scheme}) and we can apply \Cref{prop:general-optimality}.
	
	\begin{assumption}\label{ass:lipschitz}
		Suppose that $(b, \sigma)$ satisfy the following:
		\begin{enumerate}[label = (\roman*)]
			\item $t \mapsto b_t(x), t \mapsto \sigma_t(x)$ are Lipschitz uniformly in $x \in \R$;
			\item $x \mapsto b_t(x), x \mapsto \sigma_t(x)$ are Lipschitz uniformly in $t \in [0, T]$.
		\end{enumerate}
	\end{assumption}
	
	\begin{proposition}\label{prop:time-dependent-lipschitz}
		Suppose that the coefficients $(b, \sigma)$ and $(\bar b, \bar \sigma)$ of \eqref{eq:sde-correlated} satisfy \Cref{ass:lipschitz} and write $\mu, \nu$ for the laws of the respective strong solutions $X$, $\bar X$. Let $c \colon [0, T] \times \R \times \R \to \R$ satisfy \Cref{ass:cost}. Then the synchronous coupling attains the value $V_c(\mu, \nu)$ of the bicausal optimal transport problem \eqref{eq:bcot-cont}.
	\end{proposition}
	
	Combining \Cref{prop:time-dependent-lipschitz} with the stability result \Cref{lem:stability} leads to the following generalisation of \cite[Theorem 1.3, Remark 3.27]{BaKaRo22} to the case of both time-dependent coefficients and cost functional.

	\begin{assumption}\label{ass:regular}
		Suppose that $(b, \sigma)$ satisfy the following:
		\begin{enumerate}[label = (\roman*)]
			\item $b$ and $\sigma$ are continuous;
			\item $b$ and $\sigma$ have linear growth in $x$;
			\item pathwise uniqueness holds for the SDE \eqref{eq:sde} with coefficients $(b, \sigma)$.
		\end{enumerate}
	\end{assumption}
	
	\begin{corollary}\label{cor:regular}
		Suppose that the coefficients $(b, \sigma)$ and $(\bar b, \bar \sigma)$ of \eqref{eq:sde-correlated} satisfy \Cref{ass:regular}. Then there exist unique strong solutions $X$, $\bar X$ of the SDEs \eqref{eq:sde-correlated}, and we let $\mu, \nu$ denote their laws.
		Further, let $c \colon [0, T] \times \R \times \R \to \R$ satisfy \Cref{ass:cost}. Then the synchronous coupling attains the value $V_c(\mu, \nu)$ of the bicausal optimal transport problem \eqref{eq:bcot-cont}.
	\end{corollary}

\section{Discontinuous drift with exponential growth}\label{sec:growth-disc}

	Consider the SDE \eqref{eq:sde} with time-homogeneous coefficients. We allow the drift coefficient $b$ to exhibit both \emph{discontinuities} and \emph{exponential growth}. We will impose only mild piecewise continuity assumptions on $b$, and we will assume that the diffusion coefficient $\sigma$ is globally Lipschitz continuous and that $\sigma$ is non-zero at the points of discontinuity of $b$. In this setting we will prove:
	\begin{itemize}
		\item existence and uniqueness of strong solutions of \eqref{eq:sde};
		\item strong $L^p$-convergence of a transformed semi-implicit Euler--Maruyama scheme;
		\item optimality of the synchronous coupling between the laws of two such SDEs for the problem \eqref{eq:bcot-cont}.
	\end{itemize}
	To the best of our knowledge, this is the first work to treat SDEs with both discontinuous and exponentially growing drift. Precisely, we make the following assumptions on the coefficients $(b, \sigma)$.
	
	\begin{assumption}\label{ass:growth-disc}
		Let $(b, \sigma)$ be time-homogeneous; i.e.~$b_t \equiv b \colon \R \to \R$ and $\sigma_t \equiv \sigma \colon \R \to [0, \infty)$, for all $t \in [0, T]$. Suppose that there exists $m \in \N \cup \{0\}$ and $\xi_1$, $\dotsc\,$,~$\xi_m \in \R$, with $- \infty = \xi_0 < \xi_1 < \xi_1 < \dotsc, \xi_m < \xi_{m +1} = \infty$, and there exist constants $L_b$, $K_b$, $L_\sigma$, $\gamma$, $\eta \in (0, \infty)$ such that $(b, \sigma)$ satisfy the following:
		\begin{enumerate}[label = (A\arabic*)]
			\item \label[ass-growth-disc]{it:A1} on each interval $(\xi_k, \xi_{k + 1})$, $k \in \{0, \dotsc, m\}$,
			\begin{enumerate}[label = (\roman*)]
				\item \label[ass-growth-disc-A]{it:A1ii} $b$ is one-sided Lipschitz; i.e.~$(x - y)(b(x) - b(y)) \leq L_b|x - y|^2$, for all $x, y \in (\xi_k, \xi_{k + 1})$,
				\item \label[ass-growth-disc-A]{it:A1iii} $b$ satisfies the following exponential growth and local Lipschitz condition: for all $x, y \in (\xi_k, \xi_{k + 1})$,
					\begin{equation}|b(x) - b(y)| \leq K_b \bigl(\exp\{\gamma |x|^\eta\}+\exp\{\gamma |y|^\eta\}\bigr)|x - y|;
					\end{equation}
			\end{enumerate}
			\item \label[ass-growth-disc]{it:A2} $\sigma$ is globally Lipschitz; i.e.~$|\sigma(x) - \sigma(y)| \leq L_\sigma |x - y|$ for all $x, y \in \R$;
			\item \label[ass-growth-disc]{it:A3} $\sigma(\xi_k) \neq 0$, for all $k \in \{1, \dotsc, m\}$.
		\end{enumerate}
	\end{assumption}
		
	\begin{remark}\label{rem:non-degeneracy}
		We note that we do not impose any uniform non-degeneracy condition on the diffusion coefficient $\sigma$. This is in contrast to the setting of \citet{DaGeLe23}, for example. However, in the case of a discontinuous drift, \cite[Example 4.2]{LeSzTh15} shows that well-posedness may fail without \Cref{it:A3}.
	\end{remark}
	
	\begin{example}\label{ex:super-linear}
		Consider the introductory example from \citep{HuJeKl11}:
		\begin{equation}\label{eq:ex-super-linear}
			\ds X_t = - X_t^3 \D t + \ds W_t, \quad X_0 = x_0 \in \R,
		\end{equation}
		and a modification of the example in \citep[Section 2]{Sz21}:
		\begin{equation}\label{eq:ex-discontinuous}
			\ds X_t = \!\left(\frac12 - 2\sign(X_t - 1)\right)\! \D t + |X_t| \D W_t, \quad X_0 = x_0 \in \R.
		\end{equation}
		A more practical example is the value of a company in a Black--Scholes market, with parameters $\mu, \rho \in (0, \infty)$, that pays out dividends following a threshold strategy; i.e.\ dividends are paid out at a rate $u \in (0, \infty)$ whenever the value of the company exceeds a certain threshold level $\ell \in (0, \infty)$:
		\begin{equation}\label{eq:black-scholes}
			\ds X_t = \bigl(\mu-u \ind{X_t\ge \ell}\bigr) X_t \D t + \rho X_t \D W_t, \quad X_0 = x_0 \in (0,\infty).
		\end{equation}
		The coefficients of the SDEs \eqref{eq:ex-super-linear}, \eqref{eq:ex-discontinuous}, and \eqref{eq:black-scholes} each satisfy \Cref{ass:growth-disc}.
	\end{example}
	\begin{remark}\label{rem:growth-disc}~
		\begin{enumerate}[label = (\roman*)]
			\item In the case that $m = 0$, the drift coefficient $b$ is continuous, and there is no non-degeneracy condition on $\sigma$. In this setting, convergence in probability and almost sure convergence of the explicit Euler--Maruyama scheme were shown by \citet{GyKr96} and \citet{Gy98}, respectively. However, \citet[Theorem 1]{HuJeKl11} shows that, for the SDE \eqref{eq:ex-super-linear}, this scheme has unbounded moments and diverges in the strong $L^p$ sense. Two approaches to recover moment bounds and strong convergence have been taken in the literature.
			\begin{enumerate}
				\item Tamed numerical schemes, in which polynomially growing coefficients are appropriately rescaled, have been studied, for example, by \citet{HuJeKl12, Sa13}. These are explicit schemes, which are computationally efficient. However, the technique of taming does not readily yield a scheme satisfying the monotonicity property \Cref{ass:it-1} and to our knowledge has not been applied to exponentially growing coefficients.
				\item For a semi-implicit Euler--Maruyama scheme, \citet{Hu96} proves $L^2$-convergence of order $1/2$ at the terminal time. When $b$ has at most polynomial growth, \citet{HiMaSt02} prove uniform-in-time $L^2$-convergence at the same rate. Under additional assumptions, \citet{MaSz13a,MaSz13b,HuJe15} also study strong convergence of this scheme. However, under \Cref{ass:growth-disc} with $m = 0$, we are not aware of any existing results on uniform-in-time $L^p$-convergence.
			\end{enumerate}
			\item Strengthening \Cref{it:A1} to piecewise Lipschitz continuity, we are in the setting of \citet{LeSz16,LeSz17}, who were the first to consider transformation-based schemes for SDEs with piecewise-Lipschitz drift. Under these assumptions, the authors proved strong well-posedness of the SDEs and strong $L^p$-convergence of a transformation-based Euler--Maruyama scheme. \citet{LeSz18,MuYa20} also prove strong convergence rates for the classical Euler--Maruyama scheme. Higher-order transformation-based schemes and adaptive schemes have been introduced in \citet{MuYa22,PrScSz22,Ya21,NeSzSz19}. \citet{PrSz21,PrSzXu21} additionally allow for processes with jumps. Overviews of this line of research can be found in \citet{LeSz17b,Sz21}.
			\item Recent results of \citet{MuSaYa22, SpSz23}, as well as \citet{HuGa22}, prove existence and uniqueness of strong solutions and convergence of tamed Euler--Maruyama schemes for SDEs with drift coefficients that are both discontinuous and superlinearly growing. However, these results impose a polynomial growth condition on the drift coefficient, whereas we allow for exponential growth. We note, however, that \cite{MuSaYa22} allows for both the drift and diffusion coefficient to have polynomial growth.
		\end{enumerate}
	\end{remark}
	
	\subsection{A transformed SDE}\label{sec:transformation}
		In order to remove the discontinuities form the drift coefficient of the SDE \eqref{eq:sde}, we define a transformation of space $G \colon \R \to \R$ as in \citet{LeSz17}.
		
		First we define $\phi\colon \R \to \R$ by
		\begin{equation}
			\phi(u) =
			\begin{cases}
				(1 + u)^3(1 - u)^3 & |u| \leq 1,\\
				0 & |u|>1.
			\end{cases}
		\end{equation}
		Then, for constants $c_0 \in (0, \infty)$ and $\alpha_1$, $\alpha_2$, $\dotsc\,$, $\alpha_m \in \R \setminus \{0\}$, define $\bar \phi_k \colon \R \to \R$, for each $k \in \{1, \dotsc, m\}$, by
		\begin{equation}\label{eq:phi-bar}
			\bar \phi_k (x) \coloneqq \phi \!\left(\frac{x - \xi_k}{c_0}\right)\!(x - \xi_k)|x - \xi_k|, \quad \text{for} \; x \in \R,
		\end{equation}
		and define $G \colon \R \to \R$ by
		\begin{equation}\label{eq:G-def}
			G(x) \coloneqq x + \sum_{k = 1}^m \alpha_k \bar \phi_k(x), \quad \text{for} \; x \in \R.
		\end{equation}
		Before fixing the constants in the definition of $G$, we make the following remark.
		
		\begin{remark}\label{rem:lip-compacts}
			On the finite intervals $(\xi_k, \xi_{k+1})$, for $k \in \{1, \dotsc, m - 1\}$, the local Lipschitz constant in \Cref{it:A1ii} is bounded by $2 K_b \exp\{\gamma (|\xi_k| \vee |\xi_{k+1}|)^\eta\} < \infty$, and thus $b$ extends to a Lipschitz function on the compact interval $[\xi_k, \xi_{k + 1}]$. In particular, there exist finite one-sided limits $b(\xi_k+) \coloneqq \lim_{x \searrow \xi_k}b(x)$ and $b(\xi_{k+1}-) \coloneqq \lim_{x \nearrow \xi_{k + 1}}b(x)$. Similarly, $b$ extends to a Lipschitz function on $[\xi_1 - c, \xi_1]$ and $[\xi_m, \xi_m + c]$, for any $c \in \R$, and the one-sided limits $b(\xi_1-), b(\xi_m+)$ exist and are finite.
		\end{remark}
		
		By \Cref{rem:lip-compacts} and \Cref{it:A3}, for each $k \in \{1, \dotsc, m\}$ we can now define
		\begin{equation}\label{eq:alpha}
			\alpha_k \coloneqq \frac12 \sigma(\xi_k)^{-2}\bigl(b(\xi_k -) - b(\xi_k +)\bigr),
		\end{equation}
		and choose
		\begin{equation}\label{eq:c_0}
			c_0 < \min_{k \in \{1, \dotsc, m\}}\bigg\{\frac{1}{6|\alpha_k|}\bigg\} \wedge \min_{k \in \{1, \dotsc, m - 1\}}\bigg\{\frac 12 (\xi_{k + 1} - \xi_{k})\bigg\}.
		\end{equation}
		We will see that this choice of constants ensures that the transformation $G$ acts only locally at each discontinuity, is strictly increasing, and the drift of the transformed SDE is continuous. We prove the following properties of the transformation $G$ in \Cref{app:proofs-transform}.
		
		\begin{lemma}\label{lem:phi-bar}
			For each $k \in \{1, \dotsc, m\}$, $\bar \phi_k$ defined in \eqref{eq:phi-bar} with $c_0$ chosen according to \eqref{eq:c_0} satisfies the following:
			\begin{itemize}
				\item $\bar \phi_k(x) = \bar \phi^\prime_k(x) = \bar \phi^{\prime \prime}_k(x) = 0$ for $|x - \xi_k| \geq c_0$,
				\item $\bar \phi_k(\xi_k) = \bar \phi_k^\prime(\xi_k) = 0$, $\bar \phi_k^{\prime \prime}(\xi_k -) = -2$, $\bar \phi_k^{\prime \prime}(\xi_k +) = 2$,
				\item $\bar \phi_k^\prime(x) \in [-6 c_0, 6c_0]$, for all $x \in \R$,
				\item $\bar \phi_k \colon \R \to \R$ is Lipschitz, has Lipschitz first derivative, and has piecewise Lipschitz almost-everywhere second derivative with discontinuity point $\xi_k$.
			\end{itemize}
		\end{lemma}
		
		\begin{lemma}[Cf.~{\citep[Lemma 2.2]{LeSz17}}]\label{lem:G-properties-growth-disc}
			The transformation $G$ defined by \eqref{eq:G-def}, \eqref{eq:alpha}, and \eqref{eq:c_0} satisfies the following:
			\begin{itemize}
				\item $G = \id$ on the set $\continuityset$,
				\item $G$ is strictly increasing with strictly increasing inverse $G^{-1}\colon \R \to \R$,
				\item $G, G^{-1}$ are bounded and Lipschitz with Lipschitz constants $L_G, L_{G^{-1}} \in (0, \infty)$,
				\item $G, G^{-1}$ have bounded Lipschitz first derivatives $G^\prime, (G^{-1})^\prime$, and have bounded piecewise Lipschitz almost-everywhere second derivatives $G^{\prime \prime}, (G^{-1})^{\prime \prime}$ with discontinuity points $\xi_1$, $\dotsc\,$,~$\xi_m$.
			\end{itemize}
		\end{lemma}
		
		\begin{example}\label{ex:single-discontinuity}
			Suppose that $m = 1$; i.e.~there is only a single exceptional point $\xi \in \R$ at which $b$ is not Lipschitz.
			\Cref{fig:G} illustrates that $G$ is close to linear, increasing, and has Lipschitz first derivative and Lipschitz-almost-everywhere second derivative.
			\begin{figure}[h]
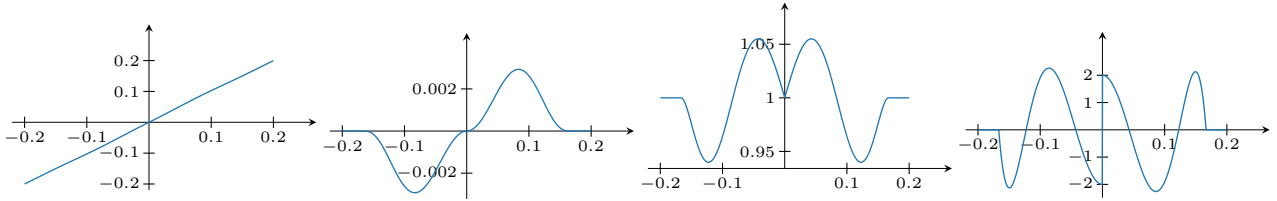

				\centering
				\input{G.pgf}
				\input{Gminus.pgf}
				\input{Gprime.pgf}
				\input{Gsecond.pgf}				\caption{From left to right: the functions $G$, $G-\operatorname{id}$, $G'$, and $G''$, for $\xi = 0$, $\alpha = 1$, $c = 1/6$.}
				\label{fig:G}
			\end{figure}
		\end{example}
		
		Supposing that there exists a strong solution $X$ of \eqref{eq:sde}, we define $Z_t \coloneqq G(X_t)$, for all $t \in [0, T]$. Formally, applying It\^o's formula, we find that $Z$ solves the SDE
		\begin{equation}\label{eq:sde-z}
			\D Z_t = \tilde b(Z_t) \D t + \tilde \sigma(Z_t) \D W_t, \quad Z_0 = G(x_0),
		\end{equation}
		where $\tilde b \colon \R \to \R$, $\tilde \sigma \colon \R \to [0, \infty)$ are defined by
		\begin{equation}\label{eq:transformed-coefficients}
			\begin{split}
				\tilde b & \coloneqq (b G^\prime) \circ G^{-1} + \frac12 (\sigma^2 G^{\prime \prime})\circ G^{-1},\\
				\tilde \sigma & \coloneqq (\sigma G^{\prime}) \circ G^{-1}.
			\end{split}
		\end{equation}
		Note that, by \Cref{lem:G-properties-growth-disc}, $\tilde b = b$ and $\tilde \sigma = \sigma$ on the set $\continuityset$. That is, the transformation only affects the coefficients on a neighbourhood of each discontinuity point.
		
		To prove regularity of the transformed coefficients, we will make use of the following lemma, which is analogous to \cite[Lemma 2.6]{LeSz16}. Again we postpone the proof to \Cref{app:proofs-transform}.
		
		\begin{lemma}\label{lem:piecewise-plus-continuity}
			Suppose that $f \colon \R \to \R$ satisfies \Cref{it:A1} and is continuous. Then $f$ satisfies \Cref{it:A1ii,it:A1iii} globally on $\R$.
		\end{lemma}
		
		We say that coefficients $(b, \sigma)$ satisfy the \emph{monotonicity condition} if there exists $\alpha, \beta \in (0, \infty)$ such that, for all $x \in \R$,
		\begin{equation}\label{eq:monotonicity}
			xb(x) \vee |\sigma(x)|^2 \leq \alpha + \beta |x|^2.
		\end{equation}
		We let $L^1_\loc(\R)$ denote the space of locally integrable functions on $\R$.
		The proof of the following properties of the transformed coefficients is once again found in \Cref{app:proofs-transform}.
		
		\begin{lemma}\label{lem:transformed-coefficients}
			Suppose that $(b, \sigma)$ satisfy \Cref{ass:growth-disc}. Then there exist constants $L_{\tilde b}$, $K_{\tilde b}$, $L_{\tilde \sigma}$, $\tilde \gamma$, $\tilde \eta$, $\alpha$, $\beta \in (0,\infty)$ such that the transformed coefficients $(\tilde b, \tilde \sigma)$ defined in \eqref{eq:transformed-coefficients} satisfy the following:
			\begin{enumerate}[label = (\roman*)]
				\item $\tilde b$ is one-sided Lipschitz; i.e.~$(x - y)(\tilde b(x) - \tilde b(y)) \leq L_{\tilde b}|x - y|^2$, for all $x, y \in \R$;
				\item $\tilde b$ satisfies the following exponential growth and local Lipschitz condition: for all $x,y\in\R$,
				\begin{equation}\label{eq:exp-growth-transformed}
					|\tilde b(x)- \tilde b(y)|\le K_{\tilde b} \bigl(\exp\{\tilde \gamma |x|^{\tilde \eta}\} + \exp\bigl\{\tilde \gamma|y|^{\tilde \eta}\bigr\}\bigr) |x - y|; 
				\end{equation}
				\item $\tilde b$ is locally absolutely continuous; i.e.\ there exists $\tilde b^\prime \in L^1_\loc(\R)$ such that, for any $x, y \in \R$,
					\begin{equation}
						\tilde b(y) - \tilde b(x) = \int_x^y \tilde b^\prime(r) \D r;
					\end{equation}
				\item $\tilde \sigma$ is globally Lipschitz; i.e.~$|\tilde \sigma (x) - \tilde \sigma(y)| \leq L_{\tilde \sigma} |x - y|$, for all $x, y \in \R$;
				\item $(\tilde b, \tilde \sigma)$ satisfy the monotonicity condition \eqref{eq:monotonicity} with constants $\alpha, \beta$.
			\end{enumerate}
		\end{lemma}
		
		The following result gives well-posedness of the transformed SDE, as well as moment bounds and one-step error bounds that will be used in the convergence analysis of the numerical scheme.
		
		\begin{proposition}\label{lem:transformed-existence-uniqueness}
			Suppose that $(b, \sigma)$ satisfy \Cref{ass:growth-disc}. Then the SDE \eqref{eq:sde-z} admits a unique strong solution $(Z_t)_{t \in [0, T]}$. Moreover, for any $p \geq 1$, there exist constants $C_0(p, T)$, $C_1(p, T)$, $C_2(p, T) > 0$ such that
			\begin{equation}\label{eq:moment-bound-monotonicity}
				\E \!\left[\sup_{t \in [0, T]}|Z_t|^p\right]\! \leq C_0(p, T)
			\end{equation}
			and, for all $s, t \in [0, T]$ with $s < t$,
			\begin{equation}\label{eq:one-step-error-exp}
				\E\!\left[\sup_{r \in [s, t]}|Z_r - Z_s|^p\right]\! \leq C_1(p, T)(t - s)^{\frac p2}
			\end{equation}
			and
			\begin{equation}\label{eq:exp-b-bound}
				\E\!\left[\sup_{r \in [s, t]}|\tilde b(Z_r) - \tilde b(Z_s)|^p\right]\! \leq C_2(p, T)(t - s)^{\frac p2}.
			\end{equation}
		\end{proposition}
				
		\begin{proof}
			By \Cref{lem:transformed-coefficients}, $(\tilde b, \tilde \sigma)$ satisfy the monotonicity condition \eqref{eq:monotonicity}, $\tilde b$ is continuous and one-sided Lipschitz, and $\tilde \sigma$ is Lipschitz. Under these conditions, the classical result of Krylov \citep[Theorem 1]{Kr91} shows existence and uniqueness of a strong solution $Z$ of \eqref{eq:sde-z}.
			The monotonicity condition \eqref{eq:monotonicity} also implies moment bounds \eqref{eq:moment-bound-monotonicity} on the solution, as shown, for example, in \citep[Lemma 3.2]{HiMaSt02}.
			
			Given that $\tilde b$ also satisfies the exponential growth condition \eqref{eq:exp-growth-transformed}, we can follow the proofs of \citet[Lemma 3.4, Theorem 3.5]{Hu96} to prove the bounds \eqref{eq:one-step-error-exp} and \eqref{eq:exp-b-bound}. We note that \citet{Hu96} assumes that the coefficients $(\tilde b, \tilde \sigma)$ are continuously differentiable, but an inspection of the proofs reveals that no such assumptions are necessary; in particular, \Cref{lem:transformed-coefficients} shows that the fundamental theorem of calculus holds.
		\end{proof}
		
		We now prove our existence and uniqueness result, following the proof of \citep[Theorem 2.6]{LeSz17}.
		
		\begin{theorem}\label{thm:existence-uniqueness}
			Suppose that the coefficients $(b, \sigma)$ satisfy \Cref{ass:growth-disc}. Then there exists a unique strong solution $X$ of \eqref{eq:sde} with coefficients $(b, \sigma)$ and initial condition $X_0 = x_0 \in \R$. Moreover, for any $p \geq 1$, there exists a constant $C(p, T) > 0$ such that
			\begin{equation}\label{eq:moment-bounds}
				\E\!\left[\sup_{t \in [0, T]}|X_t|^p\right]\! \leq C(p, T).
			\end{equation}
		\end{theorem}
		
		\begin{proof}
			By \Cref{lem:transformed-existence-uniqueness}, there exists a unique strong solution $Z$ of the SDE \eqref{eq:sde-z}. Now define $X_t = G^{-1}(Z_t)$, for all $t \in [0, T]$.
			By \Cref{lem:G-properties-growth-disc}, $G^{-1}$ is Lipschitz with Lipschitz constant $L_{G^{-1}}$, and by \Cref{lem:transformed-existence-uniqueness}, $Z$ satisfies the one-step error bound \eqref{eq:one-step-error-exp}. Hence, for $p \geq 1$,
			\begin{equation}
				\begin{split}
					\E\!\left[\sup_{t \in [0, T]}|X_t|^p\right]\!
					& \leq 2^{p - 1}L_{G^{-1}}^p\E\!\left[\sup_{t \in [0, T]}|Z_t - Z_0|^p\right]\! + 2^{p - 1}|x_0|^p \leq C(p,T),
				\end{split}
			\end{equation}
			where $C(p, T) \coloneqq 2^{p - 1}L_{G^{-1}}^pC_1(p, T)T^{\frac p 2} + 2^{p - 1}|x_0|^p$.
			
			By \Cref{lem:G-properties-growth-disc}, $G^{-1}\colon \R \to \R$ is continuously differentiable and has almost everywhere second derivative that is locally integrable. Thus we may apply It\^o's formula (see, e.g.~\cite[\S VI~Theorem 1.5 and Remarks]{ReYo99}) to see that $X$ satisfies the SDE \eqref{eq:sde} with coefficients $(b, \sigma)$, and $X_0 = G^{-1}(Z_0) = x_0$. Since $Z$ is a strong solution, and $X$ is a deterministic transformation of $Z$ at each time, $X$ is also a strong solution. Conversely, for any strong solution $X^\prime$ of \eqref{eq:sde} with coefficients $(b, \sigma)$ and with $X^\prime_0 = x_0$, the process $Z^\prime$ defined by taking $Z^\prime_t = G(X^\prime_t)$ for $t \in [0, T]$ is a strong solution of \eqref{eq:sde-z}. By pathwise uniqueness, $Z^\prime_t = Z_t$ for all $t \in [0, T]$ almost surely. Thus, $X^\prime_t = G^{-1}(Z^\prime_t) = G^{-1}(Z_t) = X_t$ for all $t \in [0, T]$ almost surely. We conclude that strong existence and pathwise uniqueness hold for \eqref{eq:sde} with coefficients $(b, \sigma)$.
		\end{proof}
		
	\subsection{Transformed semi-implicit Euler--Maruyama scheme}\label{sec:implicit-transformed-em}
		We now define a numerical scheme via the transformation introduced in the previous section.
		
		\begin{definition}[transformed (monotone) semi-implicit Euler--Maruyama scheme]\label{def:mono-si-em}
			Fix $N \in \N$, with $N > T$, and set $h = T / N$. For the SDE \eqref{eq:sde} with coefficients $(b, \sigma)$ satisfying \Cref{ass:growth-disc}, let $G$ be defined as in \eqref{eq:G-def} and $(\tilde b, \tilde \sigma)$ as in \eqref{eq:transformed-coefficients}.
		  
			Define the \emph{transformed monotone semi-implicit Euler--Maruyama scheme} $X^h$ by $X^h_0 = x_0$ and, for $k \in \{0, \dotsc, N - 1\}$,
			\begin{equation}\label{eq:implicit-scheme}
				 X^h_{(k + 1)h} \coloneqq G^{-1}\Bigl(G(X^h_{kh}) + (\tilde b \circ G)(X^h_{(k+1)h}) \cdot h + (\tilde \sigma \circ G)(X^h_{kh})\Delta W^h_{k + 1}\Bigr),
			\end{equation}
			where $\Delta W^h_{k + 1}$ is a truncated Brownian increment, as in \Cref{def:trunc-bm}.
			
			We may define the \emph{transformed semi-implicit Euler--Maruyama scheme} $\tilde X^h$ analogously, replacing the truncated increments with the usual Brownian increments $\Delta W_{k}$.
		\end{definition}
		
		\begin{remark}\label{rem:terminology-monotone-implicit}~
			\begin{enumerate}[label=(\roman*)]
				\item Similarly to the monotone Euler--Maruyama scheme discussed in \Cref{sec:num-sdes}, the term \emph{monotone} in \Cref{def:mono-si-em} is motivated by the fact that the scheme will be shown to be stochastically monotone when $\tilde b$ is one-sided Lipschitz and $\tilde \sigma$ is Lipschitz; see \Cref{lem:stochdom}.
				\item If $b$ is continuous, then the transformation $G$ is the identity and $\tilde X^h$ is the classical semi-implicit (or backward) Euler--Maruyama scheme, as studied for example in \cite{Hu96,HiMaSt02,MaSz13a,MaSz13b}.
			\end{enumerate}
		\end{remark}
		
		For the scheme defined in \Cref{def:mono-si-em}, we will prove moment bounds and strong $L^p$-convergence for $p \ge 1$. In particular, for the semi-implicit Euler--Maruyama scheme, we generalise the moment bounds from \citet{HiMaSt02} and the $L^2$ convergence from \citet{Hu96}. We note that \citep{Hu96,HiMaSt02} also assume continuous differentiability of $b$ and $\sigma$, but \cite{Hu96} uses this condition only for existence and uniqueness of solutions, which we have established under \Cref{ass:growth-disc} in \Cref{thm:existence-uniqueness}. Similarly, the result of \cite{HiMaSt02} can be established without the continuous differentiability assumption in our setting.
		
		We will repeatedly make use of the following lemma.
		
		\begin{lemma}\label{lem:invertible-implicit}
			Suppose that $f \colon \R \to \R$ is one-sided Lipschitz with constant $L_f$ and let $h \in (0, 1/L_f)$. Then the map $(\id - h f)$ is strictly increasing and invertible with strictly increasing inverse.
		\end{lemma}
		
		\begin{proof}
			Suppose that $x, y \in \R$, $x < y$. By the one-sided Lipschitz condition,
			\begin{equation}
				\begin{split}
					h (x - y)(f(x) - f(y)) \leq h L_f |x - y|^2 < |x - y|^2.
				\end{split}
			\end{equation}
			Then, since $x - y < 0$,
			\begin{equation}
				\begin{split}
					h (f(x) - f(y)) > (x - y)^{-1}|x - y|^2 = x - y,
				\end{split}
			\end{equation}
			or equivalently
			\begin{equation}
				(\id - h f)(x) < (\id - h f)(y).
			\end{equation}
			Hence $(\id - h f)$ is strictly increasing, and thus invertible with strictly increasing inverse.
		\end{proof}
		
		\begin{lemma}\label{well-defined-transf-implicit}
			Suppose that the coefficients of \eqref{eq:sde} satisfy \Cref{ass:growth-disc}, let $\tilde b$ be the transformed drift in \eqref{eq:transformed-coefficients}, and let $L_{\tilde b} \in (0, \infty)$ be as in \Cref{lem:transformed-coefficients}. Then, for all $h < 1/L_{\tilde b}$, the transformed (monotone) semi-implicit Euler--Maruyama scheme \eqref{eq:implicit-scheme} is well defined.
		\end{lemma}
		
		\begin{proof}
			By \Cref{lem:transformed-coefficients}, $\tilde b$ is one-sided Lipschitz with constant $L_{\tilde b}$. Then, for $h < 1/L_{\tilde b}$, the map $(\id - h \tilde b)$ is invertible by \Cref{lem:invertible-implicit}. We can rewrite \eqref{eq:implicit-scheme} as
			\begin{equation}
				(\id - h \tilde b)\bigl(G(X^h_{(k + 1)h})\bigr) = G(X^h_{kh}) + (\tilde \sigma \circ G)(X^h_{kh})\Delta W^h_{k + 1},
			\end{equation}
			and similarly in the case of non-truncated Brownian increments. Since $G$ is also invertible by \Cref{lem:G-properties-growth-disc}, we may conclude.
		\end{proof}
		
		\begin{proposition}\label{prop:implicit-euler-moment-bounds}
			Suppose that the coefficients of \eqref{eq:sde} satisfy \Cref{ass:growth-disc}, let $(\tilde b, \tilde \sigma)$ be the transformed coefficients in \eqref{eq:transformed-coefficients}, and let $\beta, L_{\tilde b} \in (0, \infty)$ be as in \Cref{lem:transformed-coefficients}. Let $N \in \N$ with $N > (1 \vee 4\beta \vee 2 L_{\tilde b})T$, let $h = T/N$, and let $X^h$ denote the transformed monotone semi-implicit Euler--Maruyama scheme defined by \eqref{eq:implicit-scheme}.
			
			Then, for any $p \geq 1$, there exists a constant $C(p, T) > 0$, not depending on $h$, such that
			\begin{equation}\label{eq:moment-bound-implicit-scheme}
				\E\!\left[\sup_{k \in \{0, \dotsc, N\}}|X^h_{kh}|^p\right]\! \leq C(p, T).
			\end{equation}
			Analogous bounds hold for the scheme $\tilde X^h$ defined with respect to non-truncated Brownian increments.
		\end{proposition}
		
		\begin{proof}
			Since $h < 1/L_{\tilde b}$, \Cref{well-defined-transf-implicit} implies that $X^h$ and $\tilde X^h$ are well defined.
			
			First suppose that $p > 4$.
			For $k \in \{0, \dotsc, N\}$, define $Z^h_{kh} \coloneqq G(X^h_{kh})$. By \Cref{lem:G-properties-growth-disc}, $G^{-1}$ is Lipschitz with constant $L_{G^{-1}}$, and so
			\begin{equation}\label{eq:lipschitz-moment}
				\E\!\left[\sup_{k \in \{0, \dotsc, N\}}|X^h_{kh}|^p\right]\! \leq 3^{p - 1}L_{G^{-1}}^p\E\!\left[\sup_{k \in \{0, \dotsc, N\}}|Z^h_{kh}|^p\right]\! + 3^{p - 1}L_{G^{-1}}^p|G(x_0)|^p + 3^{p - 1}|x_0|^p.
			\end{equation}
			We aim to bound the moments of $\sup_{k \in \{0, \dotsc, N\}}|Z^h_{kh}|$.
			
			Define $K_h \coloneqq (1 - 2 \beta h)^{-1} < 2$ and, for $k \in \{0, \dotsc, N\}$, define $Y^h_{kh} \coloneqq (\id - h \tilde b)(Z^h_{kh})$. By \Cref{lem:transformed-coefficients}, $(\tilde b, \tilde \sigma)$ satisfy the monotonicity condition \eqref{eq:monotonicity}. As in \citep[Lemma 3.2]{MaSz13b}, this implies that, for any $k \in \{0, \dotsc, N\}$,
			\begin{equation}\label{eq:SSBE-bounds-BE}
				|Z^h_{kh}|^2 \leq K_h\bigl(|Y^h_{kh}|^2 + 2 \alpha h\bigr).
			\end{equation}
			Therefore it suffices to bound the moments of $\sup_{k \in \{0, \dotsc, N\}}|Y^h_{kh}|$. Observe that, for any $k \in \{0, \dotsc, N - 1\}$,
			\begin{equation}\label{eq:SSBE}
				Y^h_{(k + 1)h} = Z^h_{kh} + \tilde \sigma(Z^h_{kh})\Delta W^h_{k+1},
			\end{equation}
			and
			\begin{equation}
				Z^h_{kh} = Y^h_{kh} + h \tilde b(Z^h_{kh}).
			\end{equation}
			Thus $(Y^h_{kh})_{k \in \{0, \dotsc, N \}}$ coincides with the so-called \emph{split-step backward Euler method} described in \citep[Section 3.3]{HiMaSt02}, started from $Y^h_0 = x_0 - h \tilde b(x_0)$ and with the increments $(\Delta W_{k + 1})_{k \in \{0, \dotsc, N - 1\}}$ replaced by $(\Delta W^h_{k + 1})_{k \in \{0, \dotsc, N - 1\}}$. We now closely follow the proof of \citep[Lemma 3.7]{HiMaSt02} to obtain moment bounds.
			
			To ensure finite moments, for $R > 0$, define the stopping time
			\begin{equation}
				\theta_R \coloneqq \inf\{\, k \in \{0, \dotsc N\} : |Y^h_{kh}| > R \,\},
			\end{equation}
			with the convention that $\inf \emptyset \coloneqq \infty$. For $k \in \{0, \dotsc, \theta_R - 1\}$, we have the bound $|Y^h_{kh}| \leq R$. To bound the moments of $|Y^h_{\theta_R h}|$ we first note that, by \eqref{eq:SSBE-bounds-BE},
			\begin{equation}
				|Z^h_{(\theta_R - 1)h}|^p \leq 2^{\frac p2 - 1}K_h^{\frac p2} \bigl(|Y^h_{(\theta_R - 1)h}|^p + (2 \alpha h)^{\frac p2}\bigr) \leq 2^{p - 1} \bigl(R^p + (2 \alpha)^{\frac p2}\bigr) \eqqcolon C(p, R).
			\end{equation}
			Note that, in \Cref{def:trunc-bm} of the truncation Brownian motion, we have $A_h < 2$. Then, combining \eqref{eq:SSBE} with the monotonicity condition \eqref{eq:monotonicity} proved in \Cref{lem:transformed-coefficients}, we have
			\begin{equation}
				\begin{split}
					\E[|Y^h_{\theta_R h}|^p] & \leq 2^{p - 1}\E[|Z^h_{(\theta_R - 1)h}|^p] + 2^{p -1}A_h^p \E\bigl[|\tilde \sigma(Z_{(\theta_R - 1)h})|^p\bigr]\\
					& \leq 2^{p - 1}\E[|Z^h_{(\theta_R - 1)h}|^p] + 2^{p - 1}2^{\frac p 2 - 1}A_h^p \bigl(\alpha^{\frac p2} + \beta^{\frac p2}\E[|Z_{(\theta_R - 1)h}|^p]\bigr)\\
					& \leq 2^{p - 1}C(p, R)^p(1 + 2^{\frac{3p}{2} - 1}\beta^{\frac p2}) + 2^{\frac{5p}{2} - 2}\alpha^{\frac p2} \eqqcolon \tilde C(p, R).
				\end{split}
			\end{equation}
			Hence
			\begin{equation}\label{eq:moments-finite}
				\sup_{k \in \{0, \dotsc, N\}} \E\bigl[|Y^h_{(k \wedge \theta_R)h}|^p\bigr] \leq \max \bigl\{R^p, \tilde C(p, R)\bigr\} < \infty.
			\end{equation}
			
			Combining \eqref{eq:SSBE-bounds-BE} with \eqref{eq:SSBE} gives
			\begin{equation}
				\begin{split}
					|Y^h_{(k + 1)h}|^2 & = |Z^h_{kh}|^2 + 2 Z^h_{kh} \tilde \sigma(Z^h_{kh}) \Delta W^h_{k + 1} + |\tilde \sigma(Z^h_{kh}) \Delta W^h_{k+1}|^2\\
					& \leq K_h |Y^h_{kh}|^2 + 2 \alpha h K_h + 2 Z^h_{kh} \tilde \sigma(Z^h_{kh}) \Delta W^h_{k + 1} + |\tilde \sigma(Z^h_{kh}) \Delta W^h_{k+1}|^2,
				\end{split}
			\end{equation}
			for any $k \in \{0, \dotsc, N - 1\}$.
			For all $\ell \in \{1, \dotsc, N\}$, noting that $K_h = 1 + 2\beta h K_h$, and iterating the above inequality, we have
			\begin{equation}
				\begin{split}
					|Y^h_{(\ell \wedge \theta_R)h}|^2 & \leq |Y^h_0|^2 + 2 \beta h K_h \sum_{k = 0}^{(\ell \wedge \theta_R) - 1}|Y^h_{kh}|^2 + 2 \alpha TK_h + 2 \sum_{k = 0}^{(\ell \wedge \theta_R) - 1}Z^h_{kh} \tilde \sigma(Z^h_{kh}) \Delta W^h_{k + 1}\\
					& \quad + \sum_{k = 0}^{(\ell \wedge \theta_R) - 1}\bigl|\tilde \sigma(Z^h_{jh})\Delta W^h_{k + 1}\bigr|^2.
				\end{split}
			\end{equation}
			We now raise each side to the power $p/2$ and use the inequality $(a_1 + a_2 + a_3 + a_4)^{\frac p2} \leq 4^{\frac p2 - 1}(|a_1|^{\frac p2} + |a_2|^{\frac p2} + |a_3|^{\frac p2} + |a_4|^{\frac p2})$, for any $a_1, a_2, a_3, a_4 \in \R$. Setting $c_0(p, T) \coloneqq 2^{\frac {3p}{2} - 2}(|G(x_0)|^2 + |\tilde b(G(x_0))|^2 + 2 \alpha T)^{\frac p2}$, we have
			\begin{equation}\label{eq:BE-bound-sum}
				\begin{split}
					|Y^h_{(\ell \wedge \theta_R)h}|^p & \leq c_0(p, T) + 2^{\frac{3p}{2} - 2}\bigg(\beta h K_h \sum_{k = 0}^{(\ell \wedge \theta_R) - 1}|Y^h_{kh}|^2\bigg)^{\frac p2} + 2^{\frac{3p}{2} - 2}\bigg | \sum_{k = 0}^{(\ell \wedge \theta_R) - 1} Z^h_{kh} \tilde \sigma(Z^h_{kh}) \Delta W^h_{k + 1}\bigg |^{\frac p2}\\
					& \quad + 2^{p - 2}\bigg(\sum_{k = 0}^{(\ell \wedge \theta_R) - 1}|\tilde \sigma(Z^h_{kh}) \Delta W^h_{k + 1}|^2\bigg)^{\frac p2}.
				\end{split}
			\end{equation}
			
			For all $M \in \{1, \dotsc, N\}$, we seek a bound for $\E[\sup_{\ell \in \{0, \dotsc, M\}} |Y^h_{(\ell \wedge \theta_R)h}|^p]$. We will treat each term in the sum \eqref{eq:BE-bound-sum} in turn. First, since $p > 2$, an application of H\"older's inequality gives
			\begin{equation}\label{eq:BE-bound-1}
				2^{\frac{3p}{2} - 2}\bigg(\beta h K_h \sum_{k = 0}^{(\ell \wedge \theta_R) - 1}|Y^h_{kh}|^2\bigg)^{\frac p2} \leq c_1(p, T) h \sum_{k = 0}^{\ell - 1} |Y^h_{(k \wedge \theta_R)h}|^p,
			\end{equation}
			where $c_1(p, T) \coloneqq 2^{2p - 2}\beta^{\frac p2} T^{\frac p2 -1}$.
			
			To bound the second term in \eqref{eq:BE-bound-sum}, we apply the Burkholder--Davis--Gundy inequality, noting that the truncated Brownian motion $W^h$ is an $\F^W$-martingale with quadratic variation $\langle W^h \rangle$ satisfying the bound $\langle W^h \rangle_t \leq \langle W \rangle_t = t$, for all $t \geq 0$. Thus there exists a constant $c_2(p) > 0$ such that
			\begin{equation}
				\begin{split}
					\E \!\left[\sup_{\ell \in \{0, \dotsc, M\}}\bigg | \sum_{k = 0}^{(\ell \wedge \theta_R) - 1} Z^h_{kh} \tilde \sigma(Z^h_{kh}) \Delta W^h_{k + 1}\bigg |^{\frac p2}\right]\! & = \E \!\left[\sup_{\ell  \in \{0, \dotsc, M\}}\bigg | \sum_{k = 0}^{(\ell \wedge \theta_R) - 1} \int_{kh}^{(k + 1)h}Z^h_{kh} \tilde \sigma(Z^h_{kh}) \D W^h_s\bigg |^{\frac p2}\right]\!\\
					& \leq c_2(p) \E \!\left[\bigg ( \sum_{k = 0}^{(M \wedge \theta_R) - 1}h\bigl|Z^h_{kh} \tilde \sigma(Z^h_{kh})\bigr|^2 \bigg )^{\frac p4}\right]\!.
				\end{split}
			\end{equation}
			Now, noting that $p > 4$, H\"older's inequality gives
			\begin{equation}
				\begin{split}
					\E \!\left[\biggl( \sum_{k = 0}^{(M \wedge \theta_R) - 1}h\bigl|Z^h_{kh} \tilde \sigma(Z^h_{kh})\bigr|^2 \biggr)^{\frac p4}\right]\! & \leq T^{{\frac p4} - 1} h \sum_{k = 0}^{M - 1}\E\Bigl[\bigl|Z^h_{(k \wedge \theta_R)h} \tilde \sigma(Z^h_{(k \wedge \theta_R)h})\bigr|^\frac{p}{2}\Bigr].
				\end{split}
			\end{equation}
			Using the monotonicity condition \eqref{eq:monotonicity} and the inequality $a \leq \frac12(1 + a^2)$ for any $a \in \R$, we find a constant $c_3(p) > 0$ such that
			\begin{equation}
				\begin{split}
					\bigl|Z^h_{(k \wedge \theta_R)h} \tilde \sigma(Z^h_{(k \wedge \theta_R)h})\bigr|^{\frac p2} & \leq |Z^h_{(k \wedge \theta_R)h}|^{\frac p2}\bigl(\alpha + \beta |Z^h_{(k \wedge \theta_R)h}|^2\bigr)^{\frac p 4}\\
					& \leq c_3(p) |Z^h_{(k \wedge \theta_R)h}|^{\frac p2}\bigl(\alpha^{\frac p4} + \beta^{\frac p4} |Z^h_{(k \wedge \theta_R)h}|^{\frac p2}\bigr)\\
					& \leq c_3(p) \!\left[\frac 12 + \Big(\frac{\alpha^{\frac{p}{2}}}{2} + \beta^{\frac p4}\Big)|Z^h_{(k \wedge \theta_R)h}|^p\right]\!,
				\end{split}
			\end{equation}
			for all $k \in \{0, \dotsc, N\}$.
			Combining the previous estimates and setting $c_4(p, T) \coloneqq 2^{-1} T^{\frac p4} c_2(p)c_3(p)$ and $c_5(p, T) \coloneqq (2^{-1}\alpha^{\frac p2} + \beta^{\frac p4})T^{\frac p4 - 1}c_2(p)c_3(p)$, we arrive at
			\begin{equation}\label{eq:BE-bound-2}
				\E \!\left[\sup_{\ell \in \{0, \dotsc, M\}}\bigg | \sum_{k = 0}^{(\ell \wedge \theta_R) - 1} Z^h_{kh} \tilde \sigma(Z^h_{kh}) \Delta W^h_{k + 1}\bigg |^{\frac p2}\right]\! \leq c_4(p, T) + c_5(p, T) h \sum_{k = 0}^{M - 1}\E[|Z^h_{(k \wedge \theta_R)h}|^p].
			\end{equation}
			
			For the final term in \eqref{eq:BE-bound-sum}, note that
			 \begin{equation}\label{eq:BE-bound-sum-final}
			 	\begin{split}
			 		\E\!\left[\sup_{\ell \in \{0, \dotsc, M\}}\bigg(\sum_{k = 0}^{(\ell \wedge \theta_R) - 1}|\tilde \sigma(Z^h_{kh}) \Delta W^h_{k + 1}|^2\bigg)^{\frac p2}\right]\! & \leq N^{\frac p2 - 1}\E\!\left[\sup_{\ell \in \{0, \dotsc, M\}}\sum_{k = 0}^{(\ell \wedge \theta_R) - 1}\bigl|\tilde \sigma(Z^h_{kh}) \Delta W^h_{k + 1}\bigr|^p\right]\!\\
			 		& \leq N^{\frac p2 - 1}\sum_{k = 0}^{M - 1}\E\bigl[|\tilde \sigma(Z^h_{kh})|^p|\Delta W^h_{k + 1}|^p\bigr].
			 	\end{split}
			 \end{equation}
			 For each $k \in \{0, \dotsc, N - 1\}$, we have $\E[|\tilde \sigma(Z^h_{kh})|^p|\Delta W^h_{k + 1}|^p] = \E[|\tilde \sigma(Z^h_{kh})|^p]\E[|\Delta W^h_{k + 1}|^p]$, by independence of increments. The monotonicity condition \eqref{eq:monotonicity} implies that $\E[|\tilde \sigma(Z^h_{kh})|^p] \leq c_6(p)(\alpha^{\frac p2} + \beta^{\frac p2}\E[|Z^h_{kh}|^p])$, for some constant $c_6(p) > 0$. For the truncated Brownian motion $W^h$, there exists a constant $c_7(p) > 0$ such that $\E[|\Delta W^h_{k + 1}|^p] \leq \E[|\Delta W_{k + 1}|^p] \leq c_7(p) h^{\frac p2}$, where the first inequality can be shown by conditioning the Brownian increment and applying Jensen's inequality. Hence, by \eqref{eq:BE-bound-sum-final}, we get
			 \begin{equation}\label{eq:BE-bound-3}
			 	\begin{split}
			 		\E\!\left[\sup_{\ell \in \{0, \dotsc, M\}}\bigg(\sum_{k = 0}^{(\ell \wedge \theta_R) - 1}|\tilde \sigma(Z^h_{kh}) \Delta W^h_{k + 1}|^2\bigg)^{\frac p2}\right]\! & \leq c_8(p, T) + c_9(p, T)h\sum_{k = 0}^{M - 1}\E[|Z^h_{(k \wedge \theta_R)h}|^p],
			 	\end{split}
			 \end{equation}
			 where $c_8(p, T) \coloneqq c_6(p) c_7(p) \alpha^{\frac p2} T^{\frac p2}$ and $c_9(p, T) \coloneqq c_6(p) c_7(p) \beta^{\frac p2}T^{\frac p2 - 1}$.
			 		 
			 Combining \eqref{eq:SSBE-bounds-BE} with \eqref{eq:BE-bound-2} and \eqref{eq:BE-bound-3} gives
			 \begin{equation}
			 	\begin{split}
			 		\E \!\left[\sup_{\ell \in \{0, \dotsc, M\}}\bigg | \sum_{k = 0}^{(\ell \wedge \theta_R) - 1} Z^h_{kh} \tilde \sigma(Z^h_{kh}) \Delta W^h_{k + 1}\bigg |^{\frac p2}\right]\! & \leq \tilde c_4(p, T) + \tilde c_5(p, T) h \sum_{k = 0}^{M - 1}\E[|Y^h_{(k \wedge \theta_R)h}|^p],\\
			 		\E\!\left[\sup_{\ell \in \{0, \dotsc, M\}}\bigg(\sum_{k = 0}^{(\ell \wedge \theta_R) - 1}|\tilde \sigma(Z^h_{kh}) \Delta W^h_{k + 1}|^2\bigg)^{\frac p2}\right]\! & \leq \tilde c_8(p, T) + \tilde c_9(p, T)h\sum_{k = 0}^{M - 1}\E[|Y^h_{(k \wedge \theta_R)h}|^p],
			 	\end{split}
			 \end{equation}
			 for constants $\tilde c_4(p, T) = c_4(p,T) + 2^{\frac{3p}{2} - 1}\alpha^{\frac p2} T c_5(p, T)$, $\tilde c_5(p, T) = 2^{p - 1}c_5(p, T)$, $\tilde c_8(p, T) = c_8(p,T) + 2^{\frac{3p}{2} - 1} \alpha^{\frac p2} T c_9(p, T)$, $\tilde c_5(p, T) = 2^{p - 1}c_9(p, T)$. Further combining this with \eqref{eq:BE-bound-sum} and \eqref{eq:BE-bound-1}, we find that
			 \begin{equation}\label{eq:discrete-gronwall}
			 	\E \!\left[\sup_{\ell \in \{0, \dotsc, M\}} |Y^h_{(\ell \wedge \theta_R) h}|^p\right]\! \leq \tilde c_0(p, T) + \tilde c_1(p, T)h \sum_{k = 0}^{M - 1}\E\!\left[\sup_{j \in \{0, \dotsc, k \}}|Y^h_{(j \wedge \theta_R)h}|^p\right]\!,
			 \end{equation}
			 where $\tilde c_0(p, T) = c_0(p, T) + 2^{\frac{3p}{2} - 2}\tilde c_4(p, T) + 2^{p - 2}\tilde c_8(p, T)$ and $\tilde c_1(p, T) = c_1(p, T) + 2^{\frac{3p}{2} - 2}\tilde c_5(p, T) + 2^{p - 2}\tilde c_9(p, T)$.
			
			In light of \eqref{eq:moments-finite}, we can apply the discrete Gronwall inequality (see, e.g.~\citep[Lemma 3.4]{MaSz13a}) to conclude that
			 \begin{equation}
			 	\E \!\left[\sup_{\ell \in \{0, \dotsc, N\}} |Y^h_{(\ell \wedge \theta_R)h}|^p\right]\! \leq \tilde c_0(p, T)e^{\tilde c_1(p, T)T}.
			 \end{equation}
			 We can now take the limit $R \to \infty$ and apply Fatou's lemma to find
			  \begin{equation}
			 	\E \!\left[\sup_{\ell \in \{0, \dotsc, N\}} |Y^h_{\ell h}|^p\right]\! \leq \tilde c_0(p, T)e^{\tilde c_1(p, T)T}.
			 \end{equation}
			 A further application of \eqref{eq:SSBE-bounds-BE} yields
			 \begin{equation}
			 	\E \!\left[\sup_{\ell \in \{0, \dotsc, N\}} |Z^h_{\ell h}|^p\right]\! \leq 2^{\frac p2 - 1}K_h^{\frac p2}\Big(\tilde c_0(p, T)e^{\tilde c_1(p, T)T} + (2\alpha h)^{\frac p2}\Big).
			 \end{equation}
			 Set $C_0(p, T) \coloneqq 2^{p - 1}(\tilde c_0(p, T)e^{\tilde c_1(p, T)T} + (\tfrac{\alpha}{\beta})^{\frac p2})$ and $C(p, T) \coloneqq 2^{2(p - 1)}L_{G^{-1}}^p (C_0(p, T) + |G(x_0)|^p) + 2^{p - 1}|x_0|^p$. Applying the bound \eqref{eq:lipschitz-moment}, we conclude that \eqref{eq:moment-bound-implicit-scheme} holds.
			 
			 Given the bound \eqref{eq:moment-bound-implicit-scheme} for some $p = \bar p > 4$, we deduce that \eqref{eq:moment-bound-implicit-scheme} also holds for any $p \in [1, 4]$ with $C(p, T) = C(\bar p, T)^{\frac{p}{\bar p}}$, by applying Jensen's inequality.
			 
			 The result for $\tilde X^h$ is proved analogously.
		\end{proof}
		
		We prove the following strong convergence rates.
	
		\begin{theorem}\label{thm:implicit-transformed-convergence}
			Suppose that the coefficients of \eqref{eq:sde} satisfy \Cref{ass:growth-disc}, let $X$ be the unique strong solution of \eqref{eq:sde}, let $(\tilde b, \tilde \sigma)$ be the transformed coefficients in \eqref{eq:transformed-coefficients}, and let $\beta, L_{\tilde b} \in (0, \infty)$ be as in \Cref{lem:transformed-coefficients}. Let $N \in \N$ with $N > (1 \vee 4\beta \vee 2 L_{\tilde b})T$, let $h = T/N$, and let $X^h$ denote the transformed monotone semi-implicit Euler--Maruyama scheme defined by \eqref{eq:implicit-scheme}.
			
			Then, for any $p \geq 1$, there exists $C \in(0,\infty)$ such that
			\begin{equation}\label{eq:convergence-rate}
				\sup_{k \in \{1, \dotsc, N\}}\mathbb{E}[|X_{kh}- X^h_{kh}|^{p}]^{\frac1p} \vee \E\Biggl[\sum_{k = 0}^{N - 1} \int_{kh}^{(k+1)h}|X_t - X^h_{kh}|^p\D t\Biggr]^{\frac 1p} \le
				\begin{cases}
					C h^{\frac 12}, & p \in [1, 2],\\
					C h^{\frac{1}{p(p - 1)}}, & p \in [2, \infty).
				\end{cases}
			\end{equation}
			Analogous estimates hold for the scheme $\tilde X^h$ defined with respect to non-truncated Brownian increments.
		\end{theorem}
		
		\begin{remark}\label{rem:rates-transformed-implicit}~
			\begin{enumerate}[label = (\roman*)]
				\item To our knowledge, this is the first result on strong convergence of an implicit scheme in the presence of \emph{both} discontinuities \emph{and} exponential growth in the drift coefficient.
				\item For $p \in [1, 2]$, we obtain the strong $L^p$ convergence rate of $1/2$ that we expect. For $p > 2$, we leave the question of improving the strong $L^p$ convergence rate for future research.
				\item In the case that $b$ is continuous, our result still strengthens \citet[Theorem 2.4]{Hu96}, since we relax the continuous differentiability assumption on the coefficients and consider general $p \in [1, \infty)$, not only $p = 2$.
			\end{enumerate}
		\end{remark}
	
		\begin{proof}[Proof of \Cref{thm:implicit-transformed-convergence}]
		By assumption, $h$ is chosen sufficiently small that $X^h$ and $\tilde X^h$ are well defined by \Cref{well-defined-transf-implicit}, and that the moment bounds shown in \Cref{prop:implicit-euler-moment-bounds} hold.
		
		We first consider $p = 2$. 
		As in the proof of \Cref{thm:existence-uniqueness}, $X_t = G^{-1}(Z_t)$, for all $t \in [0, T]$, where $Z$ is the unique strong solution of \eqref{eq:sde} with coefficients $(\tilde b, \tilde \sigma)$, as defined in \eqref{eq:transformed-coefficients}, and $Z_0 = G(x_0)$. Now, for $k \in \{0, \dotsc, N\}$, define $Z^h_{kh} \coloneqq G(X^h_{kh})$, so that  $Z^h_0 = Z_0$ and, for $k \in \{0, \dotsc, N - 1\}$,
			\begin{equation}\label{eq:semi-implicit-z}
				Z^h_{(k + 1)h} = Z^h_{kh} + h \tilde b(Z^h_{(k + 1)h}) + \tilde \sigma(Z^h_{kh}) \Delta W^h_{k + 1}.
			\end{equation}
			Then, using the Lipschitz property of $G^{-1}$ proved in \Cref{lem:G-properties-growth-disc}, we have
			\begin{equation}\label{eq:lipschitz-estimate}
				\begin{split}
					\E [|X_{kh} - X^h_{kh}|^2] & \leq L_{G^{-1}}^2 \E[|Z_{kh} - Z^h_{kh}|^2],\\
					\E\Biggl[\sum_{k = 0}^{N - 1} \int_{kh}^{(k+1)h}|X_t - X^h_{kh}|^2\D t\Biggr] & \leq L_{G^{-1}}^2\E\Biggl[\sum_{k = 0}^{N - 1} \int_{kh}^{(k+1)h}|Z_t - Z^h_{kh}|^2\D t\Biggr].
				\end{split}
			\end{equation}
			Therefore it suffices to analyse the $L^2$ error between $Z^h$ and $Z$.
			
			For any $k \in \{0, \dotsc, N - 1\}$, we can write
			\begin{equation}
				 Z_{(k+1)h} =  Z_{kh} + \tilde b(Z_{(k+1)h})h+ \tilde \sigma(Z^{h}_{kh})\Delta W^h_{k + 1} +B_{k + 1} +\Sigma_{k + 1} + \Sigma^h_{k + 1} + \mathcal E_{k + 1},
			\end{equation}
			where we define the error terms
			\begin{equation}
				\begin{aligned}
					B_{k + 1} \coloneqq \int_{kh}^{(k+1)h} (\tilde b(Z_s)-\tilde b(Z_{(k+1)h}))\D s, \qquad 
					\Sigma_{k + 1} & \coloneqq \int_{kh}^{(k+1)h} (\tilde \sigma(Z_s)-\tilde \sigma(Z_{kh}))\D W_s,\\
					\Sigma^h_{k + 1} \coloneqq  (\tilde \sigma(Z_{kh})-\tilde \sigma(Z^h_{kh}))\Delta W^h_{k + 1}, \qquad
					\mathcal E_{k + 1} & \coloneqq \tilde \sigma(Z_{kh})[\Delta W_{k + 1} - \Delta W^h_{k + 1}].
				\end{aligned}
			\end{equation}
			Note that, if we define  $\tilde Z^h_{kh} \coloneqq G(\tilde X^h_{kh})$, for $k \in \{0, \dotsc, N\}$, then $\tilde Z^h$ satisfies analogous estimates to \eqref{eq:lipschitz-estimate} and also satisfies \eqref{eq:semi-implicit-z} with the truncated increment $\Delta W^h_{k + 1}$ replaced by $\Delta W_{k + 1}$. Thus we can write
			\begin{equation}
				Z_{(k+1)h} =  Z_{kh} + \tilde b(Z_{(k+1)h})h+ \tilde \sigma(\tilde Z^{h}_{kh})\Delta W_{k + 1} +B_{k + 1} +\Sigma_{k + 1} + \tilde{\Sigma}_{k + 1},
			\end{equation}
			with $\tilde{\Sigma}_{k + 1} \coloneqq \big(\tilde \sigma (Z_{kh}) - \tilde \sigma(\tilde Z^h_{kh})\big)\Delta W_{k + 1}$. We can bound the second moment of $\tilde \Sigma_{k + 1}$ in exactly the same way as we do for $\Sigma^h_{k + 1}$ below. Thus the estimates that we prove for $Z^h$ also hold for $\tilde Z^h$ with only a minor modification in the proof.
			
			Now we consider
			\begin{equation}
				\begin{aligned}
					Z_{(k+1)h}-Z^h_{(k+1)h}& = Z_{kh} -Z^h_{kh} + \bigl(\tilde b(Z_{(k+1)h})-\tilde b(Z^h_{(k+1)h})\bigr)h
					+ B_{k + 1} + \Sigma_{k + 1} + \Sigma^h_{k + 1} + \mathcal E_{k + 1}.
				 \end{aligned}
			\end{equation}
			By \Cref{lem:transformed-coefficients}, there exists $\tilde b^\prime \in L^1_\loc(\R)$ such that, for all $x,y\in\mathbb{R}$,
			\begin{equation}
				\tilde b(y)- \tilde b(x) = \int_x^y \tilde b'(s) \D s=\!\left(\int_0^1 \tilde b'(uy+(1-u)x) \D u\right)\!(y-x).
			\end{equation}
			Hence
			\begin{equation}\label{eq:conv1}
				\begin{aligned}
					&\Big(1 - h\int_0^1 \tilde b^\prime (uZ_{(k+1)h} -(1-u)Z^h_{(k+1)h}) \D u\Big)(Z_{(k+1)h} -Z^h_{(k+1)h})
					\\&\quad= (Z_{kh} -Z^h_{kh} )  
					+B_{k + 1} + \Sigma_{k + 1} + \Sigma^h_{k + 1} + \mathcal E_{k + 1}.
				\end{aligned}
			\end{equation}
			
			By the one-sided Lipschitz property shown in \Cref{lem:transformed-coefficients}, we have that $\tilde b^\prime \leq L_{\tilde b}$ Lebesgue-almost everywhere.
			Then, similarly to \citep[Proof of Theorem 2.4]{Hu96}, since $h<1/L_{\tilde b}$, the inverse 
			\begin{equation}
				\mathcal I_{k + 1} \coloneqq \biggl(1 - h\int_0^1 \tilde b'(uZ_{(k+1)h} -(1-u)Z^h_{(k+1)h}) \D u\biggr)^{-1}
			\end{equation}
			exists and $|\mathcal I_{k + 1}| \leq (1-L_{\tilde b} h)^{-1}$.
			Hence, by \eqref{eq:conv1},
			\begin{equation}
				Z_{(k+1)h} -Z^h_{(k+1)h}=\mathcal{I}_{k + 1} (Z_{kh} -Z^h_{kh} ) 
				+\mathcal{I}_{k + 1} B_{k + 1} + \mathcal{I}_{k + 1}\Sigma_{k + 1} + \mathcal{I}_{k + 1}\Sigma^h_{k + 1} + \mathcal{I}_{k + 1}  \mathcal E_{k + 1},
			\end{equation}
			and iterating gives
			\begin{equation}\label{eq:conv1a}
				\begin{aligned}
					Z_{(k+1)h} -Z^h_{(k+1)h} & =
					\sum_{j=1}^{k+1} \mathcal{I}_{k+1} \dots \mathcal{I}_{j} B_j +
					\sum_{j=1}^{k+1} \mathcal{I}_{k+1} \dots \mathcal{I}_{j} \Sigma_j\\
					& \quad +
					\sum_{j=1}^{k+1} \mathcal{I}_{k+1} \dots \mathcal{I}_{j} \Sigma^h_j +
					\sum_{j=1}^{k+1} \mathcal{I}_{k+1} \dots \mathcal{I}_{j} \mathcal E_j.
				\end{aligned}
			\end{equation}
			We now bound the second moment of the first term. By the Cauchy--Schwarz inequality,
			\begin{equation}\label{eq:conv2}
				\Big|\sum_{j=1}^{k+1} \mathcal{I}_{k+1} \dots \mathcal{I}_{j} B_j\Big|^2 \leq \sum_{j=1}^{k+1} | \mathcal{I}_{k+1} \dots \mathcal{I}_{j}|^2 \cdot \sum_{j=1}^{k+1} |B_j|^{2} \leq (k + 1)(1-L_{\tilde b}h)^{-2(k+1)}\sum_{j=1}^{k+1} |B_j|^{2}.
			\end{equation}
			Since $L_{\tilde b} h \leq \frac12$, setting $c_1 \coloneqq 4L_{\tilde b}$ gives
			\begin{equation}\label{eq:conv3}
				\begin{aligned}
					(1-L_{\tilde b}h)^{-2(k+1)} \leq \exp\!\left\{2(k+1)\frac{L_{\tilde b}h}{1 - L_{\tilde b}h}\right\}\! \leq \exp\!\left\{c_1 T\right\}\!.
				\end{aligned}
			\end{equation}
			For each $j \in \{1, \dotsc, N\}$, applying the Cauchy--Schwarz inequality and the estimate \eqref{eq:exp-b-bound} from \Cref{lem:transformed-existence-uniqueness} gives us
			\begin{align}
				\E\bigl[|B_j|^2\bigr] & \leq h \int_{(j - 1)h}^{jh}\E\Bigl[|\tilde b(Z_s) - \tilde b(Z_{jh})|^2\Bigr] \D s \leq C_2(2, T) h^3. 
			\end{align}
			Therefore, setting $c_2 = C_2(2, T)$, we have
			\begin{equation}\label{eq:conv5}
				\begin{aligned}
					\E\Biggl[\Bigl|\sum_{j=1}^{k+1} \mathcal{I}_{k+1} \dots \mathcal{I}_{j} B_j\Bigr|^2\Biggr]
					\le
					c_2 T^2 \exp(c_1T) h.
				\end{aligned}
			\end{equation}
			To estimate the remaining terms in \eqref{eq:conv1a}, we set $\mathcal S_0=0$ and, for all $k\in\{1,\dots,N\}$, denote
			\begin{equation}\label{eq:conv6}
				\begin{aligned}
					\mathcal S_k &\coloneqq
					\sum_{j=1}^{k} \mathcal{I}_{k} \dots \mathcal{I}_{j} \Sigma_j
					+
					\sum_{j=1}^{k} \mathcal{I}_{k} \dots \mathcal{I}_{j} \Sigma^h_j
					+
					\sum_{j=1}^{k} \mathcal{I}_{k} \dots \mathcal{I}_{j} \mathcal E_j
					.
				\end{aligned}
			\end{equation}
			Observe that, for all $k\in\{0,\dots,N-1\}$,
			\begin{equation}\label{eq:conv7}
				\mathcal S_{k+1} = \mathcal{I}_{k+1} (\mathcal S_k + \Sigma_{k+1} + \Sigma^h_{k+1} + \mathcal E_{k+1}).
			\end{equation}
			By independence of Brownian increments, we have that $\mathcal S_k$ is independent of $\Sigma_{k+1}$, $\Sigma^h_{k+1}$, and $ \mathcal E_{k+1}$. Using Pythagoras' theorem and then the Cauchy--Schwarz inequality, we see that the second moment is bounded by
			\begin{equation}\label{eq:conv8}
				\mathbb{E} [|\mathcal S_{k+1}|^{2}]
				\le	
				(1-L_{\tilde b}h)^{-{2} }\bigl(\mathbb{E}[| \mathcal S_k|^{2}] + 3\mathbb{E}[|\Sigma_{k+1}|^{2}] + 3\mathbb{E}[|\Sigma^h_{k+1}|^{2}]+ 3\mathbb{E}[|\mathcal E_{k+1}|^{2}]\bigr).
			\end{equation}
			We note that, in order to iterate this bound, it is essential that no constant appears in front of the term $\mathbb{E}[| \mathcal S_k|^{2}]$. Thus the above argument using independence does not directly generalise to $p$\textsuperscript{th} moments for $p \neq 2$.
			
			We now bound the right-hand side of \eqref{eq:conv8} term by term. By \Cref{lem:transformed-coefficients}, $\tilde \sigma$ is Lipschitz with constant $L_{\tilde \sigma}$. Therefore, by the It\^o isometry,
			\begin{equation}
				\begin{aligned}
					\mathbb{E}[|\Sigma_{k + 1}|^{2}]
					& =
					\E \biggl[\int_{kh}^{(k+1)h} |\tilde\sigma(Z_s)-\tilde\sigma(Z_{kh})|^2 \D s\biggr]
					\le
					L_{\tilde \sigma}^2 \E \Biggl[\int_{kh}^{(k+1)h} |Z_s - Z_{kh}|^2 \D s\Biggr].
				\end{aligned}
			\end{equation}
			Then the bound
			\eqref{eq:one-step-error-exp} gives a constant $c_3\in(0,\infty)$ such that
			\begin{equation}\label{eq:conv9}
				\begin{aligned}
					\mathbb{E}[|\Sigma_{k + 1}|^{2}]
					&\le
					c_3 L_{\tilde \sigma}^2 \int_{kh}^{(k+1)h} |s-kh| \D s
					\le
					c_3 L_{\tilde \sigma}^2 h^2.
				\end{aligned}
			\end{equation}
			
			By the independence of $Z_{kh}$ from the increments $\Delta W_{k + 1}$ and $\Delta W^h_{k + 1}$, together with the linear growth of $\tilde \sigma$ and the second moment bound from \Cref{prop:implicit-euler-moment-bounds}, there exists $c_4\in(0,\infty)$ such that
			\begin{equation}
				\begin{aligned}
					\mathbb{E}[|\mathcal E_{k + 1}|^{2}]
					& =
					\mathbb{E}[|\tilde \sigma(Z_{kh})|^{2}] \cdot \mathbb{E}[|\Delta W_{k + 1} - \Delta W^h_{k + 1}|^{2}]
					\le
					c_4  \mathbb{E}[|\Delta W_{k  + 1} - \Delta W^h_{k + 1}|^{2}].
				\end{aligned}
			\end{equation}
			By \citep[Lemma A.3]{BaKaRo22}, we have the bound $\mathbb{E}[|\Delta W_{k + 1} - \Delta W^h_{k + 1}|^{2}] \leq 2 h^3$, and so there exists $c_5\in(0,\infty)$ such that
			\begin{equation}\label{eq:conv11}
				\mathbb{E}[|\mathcal E_{k + 1}|^{2}]\le c_5  h^3.
			\end{equation}
			
			An application of the triangle inequality then yields the existence of a constant $c_6 \in (0, \infty)$ such that $\mathbb{E}[|\Delta W^h_{k + 1}|^{2}] \leq 2(\mathbb{E}[|\mathcal E_{k + 1}|^{2}] + \mathbb{E}[|\Delta W_{k + 1}|^{2}])  \leq c_6 h$.
			This, the independence of $Z_{kh}$ and  $Z^h_{kh}$ from the increment $\Delta W^h_{k + 1}$, and the Lipschitz continuity of $\tilde\sigma$ yield the bound
			\begin{equation}\label{eq:conv10}
				\begin{aligned}
					\mathbb{E}[|\Sigma^h_{k + 1}|^2]
					& =
					\mathbb{E}[|\tilde\sigma(Z_{kh})-\tilde\sigma(Z^h_{kh})|^{2}] \cdot \mathbb{E}[|\Delta W^h_{k + 1}|^{2}] 
					\\&\le
					c_6 L_{\tilde \sigma}^{2}  \mathbb{E}[|Z_{kh}-Z^h_{kh}|^{2}]  h.
				\end{aligned}
			\end{equation}
			
			Combining \eqref{eq:conv1a} with the definition of $\mathcal S_k$ and the bound \eqref{eq:conv5} gives
			\begin{equation}\label{eq:conv12}
				\begin{aligned}
					\E[|Z_{kh} -Z^h_{kh}|^{2}]\le 2 c_2T^2\exp(c_1T) h + 2 \mathbb{E} [|\mathcal S_{k}|^2].
				\end{aligned}
			\end{equation}
			We now substitute the bounds \eqref{eq:conv9}, \eqref{eq:conv11}, and \eqref{eq:conv10} into \eqref{eq:conv8}, to get
			\begin{equation}
				\begin{aligned}
					&\mathbb{E} [|\mathcal S_{k+1}|^2] \le
					(1-L_{\tilde b}h)^{-{2}}\bigl(\mathbb{E}[| \mathcal S_k|^2] + 3c_3L_{\tilde \sigma}^{2} h^2 + 3c_6 L_{\tilde \sigma}^{2}  \mathbb{E}[|Z_{kh}-Z^h_{kh}|^{2}]  h +  3c_5  h^3\bigr).
				\end{aligned}
			\end{equation}
			Then \eqref{eq:conv12} gives
			\begin{equation}\label{eq:conv13}
				\begin{aligned}
					\mathbb{E} [|\mathcal S_{k+1}|^{2}] & \le
					(1-L_{\tilde b}h)^{-{2}} (1+6c_6 L_{\tilde \sigma}^{2}  h)\mathbb{E}[| \mathcal S_k|^2] + c_7 h^2 (1-L_{\tilde b}h)^{-{2}},
				\end{aligned}
			\end{equation}
			where $c_7 \coloneqq 3(c_3 L_{\tilde \sigma}^2 + 2c_2c_6 L_{\tilde \sigma}^{2}  T^2\exp(c_1T) + c_5)$.
			Since $L_{\tilde b} h \leq \frac12$, we can check that $1 \leq (1-L_{\tilde b}h)^{-{2}} (1+6c_6 L_{\tilde \sigma}^{2}  h) \leq 1 + c_8 h$, where $c_8 \coloneqq 24 c_6L_{\tilde \sigma}^2$, and we have
			\begin{equation}\label{eq:conv14}
				\begin{aligned}
					&\mathbb{E} [|\mathcal S_{k+1}|^{2}] \le
					(1 + c_8 h)\mathbb{E}[| \mathcal S_k|^2] + 4 c_7 h^2.
				\end{aligned}
			\end{equation}
			Iterating, we obtain
			\begin{equation}
				\begin{aligned}
					\mathbb{E} [|\mathcal S_{k + 1}|^2]
					&\le
					4 c_7 h^2 \sum_{j=0}^{k}(1+c_8 h)^j
					\le
					4 c_7 h^2 \sum_{j=0}^{k}\exp(c_8 h N)
					\le	4 c_7 T \exp(c_8 T) h.
				\end{aligned}
			\end{equation}
			Combining this with \eqref{eq:conv12} gives a constant $c_{9}\in(0,\infty)$ such that
			\begin{equation}\label{eq:conv15}
				\mathbb{E} [|Z_{kh}-Z^h_{kh}|^2]
				\le
				c_{9} h.
			\end{equation}
			Now by \eqref{eq:one-step-error-exp}, for $k \in \{0, \dotsc, N - 1\}$ and $t \in [kh, (k+1)h)$, there exists a constant $c_{10} \in (0, \infty)$ such that $\E[|Z_{t} - Z_{kh}|]^2 \leq c_{10}h$. Combining this with \eqref{eq:conv15}, we have
			\begin{equation}\label{eq:conv16}
				\begin{split}
					\E\!\left[\sum_{k = 0}^{N - 1} \int_{kh}^{(k+1)h}|Z_t - Z^h_{kh}|^2\D t\right]\! & \leq 2\sum_{k = 0}^{N - 1} \int_{kh}^{(k+1)h}\!\left(\E[|Z_t - Z_{kh}|^2] + \E[|Z^h_{kh} - Z_{kh}|^2]\right)\!\D t\\
					& \leq 2 N h (c_{9} h + c_{10} h) = \tilde C h,
				\end{split}
			\end{equation}
			for $\tilde C \coloneqq 2T(c_{9} + c_{10})$.
			
			The estimates \eqref{eq:lipschitz-estimate} applied to \eqref{eq:conv15} and \eqref{eq:conv16} yield
			\begin{equation}\label{eq:conv17}
				\sup_{k \in \{1, \dotsc, N\}}\mathbb{E}\bigl[|X_{kh}- X^h_{kh}|^{2}\bigr]^{1/2} \leq L_{G^{-1}} c_9^{1/2}h^{1/2},
			\end{equation}
			and
			\begin{equation}\label{eq:conv18}
				\E\Biggl[\sum_{k = 0}^{N - 1} \int_{kh}^{(k+1)h}|X_t - X^h_{kh}|^2\D t\Biggr]^{1/2} \leq L_{G^{-1}}\tilde C^{1/2} h^{1/2}.
			\end{equation}
			
			For $p \in [1, 2)$, we obtain the same rate of convergence by applying Jensen's inequality.
			For $p \in (2, \infty)$, we apply H\"older's inequality, followed by the $L^2$ bound \eqref{eq:conv18} and the $(p+1)$-moment bounds for the scheme and the solution of the SDE that are shown in \Cref{prop:implicit-euler-moment-bounds} and \Cref{lem:transformed-existence-uniqueness}, respectively. Hence, there exists a constant $C \in (0, \infty)$ such that
			\begin{equation}
				\begin{split}
					& \E\Biggl[\sum_{k = 0}^{N - 1} \int_{kh}^{(k+1)h}|X_t - X^h_{kh}|^p\D t\Biggr]^{\frac 1p} = \E\Biggl[\sum_{k = 0}^{N - 1} \int_{kh}^{(k+1)h}|X_t - X^h_{kh}|^{\frac {2}{p - 1}} \cdot |X_t - X^h_{kh}|^{\frac{(p - 2)(p + 1)}{p - 1}}\D t\Biggr]^{\frac 1p}\\
					& \qquad \leq \E\Biggl[\sum_{k = 0}^{N - 1} \int_{kh}^{(k+1)h}|X_t - X^h_{kh}|^2\D t\Biggr]^{\frac{1}{p(p - 1)}} \E\Biggl[\sum_{k = 0}^{N - 1} \int_{kh}^{(k+1)h}|X_t - X^h_{kh}|^{p + 1}\D t\Biggr]^{\frac{p - 2}{p(p - 1)}}\\
					& \qquad \leq C h^{\frac{1}{p(p - 1)}}.
				\end{split}
			\end{equation}
			In the same way, we show that
			\begin{equation}
				\sup_{k \in \{1, \dotsc, N\}}\mathbb{E}\bigl[|X_{kh}- X^h_{kh}|^{p}\bigr]^{\frac1p} \leq C h^{\frac{1}{p(p - 1)}},
			\end{equation}
			which concludes the proof.
		\end{proof}
		
		\begin{remark}[additive noise]\label{rem:additive-noise}
			Under the additional assumption that $\sigma \equiv 1$, we only require estimates for the error terms $B_{k}$, but not for $\Sigma_k, \tilde \Sigma^h_k$, in the proof of convergence of the scheme $\tilde X^h$. Repeated applications of H\"older's inequality yield $p$\textsuperscript{th} moment estimates analogous to \eqref{eq:conv5} with order $h^{\frac p2}$. Proceeding as in the proof of \Cref{thm:implicit-transformed-convergence}, we find the strong $L^p$ convergence rate $1/2$, for all $p \geq 1$, for the scheme $\tilde X^h$ in the case of additive noise.
		\end{remark}
		
	\subsection{Application to optimal transport}\label{sec:bcot-growth-disc}
		We next apply the transformed monotone semi-implicit Euler--Maruyama scheme from \Cref{def:mono-si-em} to finding the optimiser of \eqref{eq:bcot-cont} between the laws of SDEs satisfying \Cref{ass:growth-disc}. To this end, we will verify the following properties of the scheme.
		
		\begin{lemma}\label{lem:stochdom}
			Suppose that the coefficients of \eqref{eq:sde} satisfy \Cref{ass:growth-disc}. Then, for all $h$ sufficiently small, the transformed monotone semi-implicit Euler--Maruyama scheme \eqref{eq:implicit-scheme} is stochastically increasing and, for $k \in \{0, \dotsc, N - 1\}$, the map $\Delta W^h_{k+1} \mapsto X^h_{(k+1)h}$ is increasing.
		\end{lemma}
		
		\begin{proof}
			By \Cref{lem:transformed-coefficients}, $\tilde \sigma$ is Lipschitz with constant $L_{\tilde \sigma} \in (0, \infty)$. Fix $k \in \{0, \dotsc, N - 1\}$ and $x, y \in \R$ with $x < y$. Then, using \Cref{def:trunc-bm} of $\Delta W^h$,
			\begin{equation}
				\begin{split}
					x + \tilde \sigma(x) \Delta W^h_{k + 1} - (y + \tilde \sigma(y) \Delta W^h_{k + 1}) & = x - y + (\tilde \sigma(x) - \tilde \sigma(y))\Delta W^h_{k + 1}\\
					& \leq (1 + L_{\tilde \sigma} \Delta W^h_{k + 1})(x - y)\\
					& \leq (1 - L_{\tilde \sigma} A_h)(x - y).
				\end{split}
			\end{equation}
			For $h > 0$ sufficiently small that $1 - L_{\tilde \sigma}A_h > 0$, we have that the map $(\id + \Delta W^h_{k + 1} \tilde \sigma)$ is strictly increasing with strictly increasing inverse. Let us also take $h$ sufficiently small that $(\id - h \tilde b)^{-1}$ is well defined and strictly increasing by \Cref{lem:invertible-implicit}. Then we can rewrite \eqref{eq:implicit-scheme} as
			\begin{equation}\label{eq:concat-increasing}
				X^h_{(k + 1)h} = G^{-1} \circ (\id - h \tilde b)^{-1} \circ (\id + \Delta W^h_{k + 1} \tilde \sigma) \circ G (X^h_{kh}).
			\end{equation}
			Recall that $G$ and $G^{-1}$ are strictly increasing by \Cref{lem:G-properties-growth-disc}. Thus $X^h_{(k + 1)h}$ is obtained from $X^h_{kh}$ by the concatenation of four strictly increasing maps. This shows that $(X^h_{kh})_{k \in \{0, \dotsc,  N\}}$ is stochastically increasing. Noting that $\tilde \sigma$ is non-negative, the second claim also follows from \eqref{eq:concat-increasing}.
		\end{proof}

		We can now state and prove our optimality result.
		
		\begin{theorem}\label{thm:optimality-growth-disc}
			Suppose that the coefficients $(b, \sigma)$ and $(\bar b, \bar \sigma)$ of \eqref{eq:sde-correlated} satisfy \Cref{ass:growth-disc} and write $\mu, \nu$ for the laws of the respective strong solutions $X$, $\bar X$. Let $c \colon [0, T] \times \R \times \R \to \R$ satisfy \Cref{ass:cost}. Then the synchronous coupling attains the value $V_c(\mu, \nu)$ of the bicausal optimal transport problem \eqref{eq:bcot-cont}.
		\end{theorem}
		
		\begin{proof}
			We proved strong existence and pathwise uniqueness in \Cref{thm:existence-uniqueness}. For $N \in \N$, $h = T/N$, denote by $(X^h_{kh})_{k \in \{1, \dotsc, N\}}$ and $(\bar X^h_{kh})_{k \in \{1, \dotsc, N\}}$ the transformed monotone semi-implicit Euler--Maruyama scheme for the solutions $X$ and $\bar X$ of \eqref{eq:sde}, respectively, as defined in \Cref{def:mono-si-em}. By \Cref{lem:stochdom}, $(X, \bar X)$ and  $(X^h, \bar X^h)$ satisfy \Cref{ass:it-1,ass:it-2}. By \Cref{thm:implicit-transformed-convergence}, \Cref{ass:it-3} is also satisfied. We conclude by applying \Cref{prop:general-optimality}.
		\end{proof}

\section{SDEs with bounded measurable drift}\label{sec:zvonkin}
	We now consider a second class of irregular coefficients. We only require the drift coefficient to be bounded and measurable, and we allow the diffusion coefficient to be H\"older continuous in space rather than Lipschitz, but we do require a uniform non-degeneracy condition and boundedness of the diffusion, which was not needed in \Cref{sec:growth-disc}. The coefficients may be time-dependent. In this setting, we again obtain optimality of the synchronous coupling for \eqref{eq:bcot-cont} between the laws of SDEs.
	Similarly to \Cref{sec:growth-disc}, the proof relies on a transformation of the SDEs. Here we use the drift-removing transformation that is defined by \citet[Theorem 1]{Zv74}. We solve a transformed bicausal optimal transport problem, which we show to be equivalent to \eqref{eq:bcot-cont}. This is a different approach from \Cref{sec:growth-disc}; in \Cref{sec:counterexamples} we show that the key monotonicity result required for the former approach is not satisfied in the present setting, and we provide a counterexample to show that a weaker monotonicity condition is not sufficient.
	
	We make the following assumptions on the coefficients $(b, \sigma)$ of the SDE \eqref{eq:sde}, following Zvonkin \citep{Zv74}.
	
	\begin{assumption}\label{ass:zvonkin}Suppose that the coefficients $(b, \sigma)$ satisfy the following:
		\begin{enumerate}[label = (\roman*)]
			\item $b$ is bounded and measurable;
			\item $\sigma$ is bounded and continuous and, for some $\alpha \in [1/2, 1]$, $x \mapsto \sigma_t(x)$ is $\alpha$-H\"older continuous, uniformly in $t \in [0, T]$;
			\item $\sigma$ is uniformly non-degenerate; i.e.\ there exists $C > 0$ such that $\sigma^2_t(x) \geq C$, for all $(t, x) \in [0, T] \times \R$.
		\end{enumerate}
	\end{assumption}
	
	\begin{remark}\label{rem:zvonkin}~
		\begin{enumerate}[label = (\roman*)]
			\item Under \Cref{ass:zvonkin}, \citet[Theorem 4]{Zv74} proves existence and uniqueness of strong solutions of the SDE \eqref{eq:sde} via a transformation that removes the drift.
			\item A similar transformation to that of \citep{Zv74} was used by \citet{Ta83} to analyse the numerical approximation of SDEs. More recently, \citet{NgTa17,NgTa17b,NgTa19,NeSz21,BaKaRo22,GeLaLi23} use a similar drift-removing transformation in the numerical approximation of SDEs with irregular coefficients, under stronger conditions than \Cref{ass:zvonkin}.
			\item For the SDE \eqref{eq:sde} with coefficients satisfying \Cref{ass:zvonkin}, strong convergence rates for the Euler--Maruyama scheme have been obtained by \citet{GyRa11} under additional conditions on the drift, and by \citet{DaGe20,DaGeLe23} for time-homogeneous coefficients under additional conditions on the diffusion.
		\end{enumerate}
		
	\end{remark}
	
	\begin{example}
		Consider the SDE
		\begin{equation}
			\ds X_t = \sign(X_t)\D t + \bigl(1 + \sqrt{|X_t|}\ind{|X_t| \leq 4} + 2\ind{|X_t|>4}\bigr) \D W_t; \quad X_0 = x_0 \in \R.
		\end{equation}
		The coefficients of this SDE satisfy \Cref{ass:zvonkin}.
	\end{example}
	
	Suppose that $(b, \sigma)$ satisfy \Cref{ass:zvonkin}, and let $X$ be the unique strong solution of \eqref{eq:sde} with coefficients $(b, \sigma)$. \citet[Theorem 1]{Zv74} gives the existence of a function $u\colon[0, T] \times \R \to \R$ that is in the Sobolev space $W^{1, 2}_p([0, T] \times D)$, for any bounded domain $D \subset \R$ and $p > \frac32$, and solves the PDE problem
	\begin{equation}\label{eq:zvonkin-pde}
		\begin{split}
			\partial_t u_t(x) + \partial_x u_t(x)b_t(x) + \frac12\partial_{xx}u_t(x)\sigma_t^2(x) = 0, & \quad (t, x) \in [0, T] \times \R,\\
				u_T(x) = x, & \quad x \in \R.
		\end{split}
	\end{equation}
	By \cite[Theorem 2]{Zv74}, there exists an inverse function $v \colon [0, T] \times \R \to \R$ satisfying $(v_t \circ u_t)(x) = x$ for all $t \in [0, T]$, $x \in \R$.
	
	Define $Z_t \coloneqq u_t(X_t)$ for all $t \in [0, T]$. Then \cite[Theorem 3]{Zv74} shows that we can apply It\^o's formula to obtain
	\begin{equation}\label{eq:transf-sde-zvonkin}
			\ds Z_t = \tilde \sigma_t(Z_t)\D W_t, \quad Z_0 = u_0(x_0),
	\end{equation}
	where $\tilde \sigma_t  \coloneqq (\partial_x u_t \cdot \sigma_t) \circ v_t$, for all $t \in [0, T]$.
	
	We recall the following properties of $u$ and $v$ from \cite[Theorem 2]{Zv74} without proof.
	\begin{lemma}\label{lem:G-properties}
		Suppose that $(b, \sigma)$ satisfy \Cref{ass:zvonkin}. Then, uniformly in $t \in [0, T]$,
		\begin{enumerate}[label = (\roman*)]
			\item \label{it:G1} $x \mapsto u_t(x)$ and $x \mapsto v_t(x)$ are strictly increasing and Lipschitz continuous;
			\item \label{it:G2} $x \mapsto \partial_x u_t(x)$ is bounded from above, bounded away from zero, and $\beta$-H\"older continuous, for any $\beta \in (0, 1)$;
		\end{enumerate}
		and, uniformly in $x \in D$, for any bounded domain $D \subset \R$,
		\begin{enumerate}[label = (\roman*), start = 3]
			\item \label{it:G3} $t \mapsto u_t(x)$, $t \mapsto v_t(x)$, and $t \mapsto \partial_x u_t(x)$ are $\beta$-H\"older continuous for any $\beta \in (0, 1)$.
		\end{enumerate}
	\end{lemma}
	
	The following is an immediate consequence of \Cref{lem:G-properties}.
		
	\begin{corollary}\label{cor:lipschitz-coeffs}
		Suppose that $(b, \sigma)$ satisfies \Cref{ass:zvonkin}. Then the diffusion coefficient $\tilde \sigma$ of the transformed SDE \eqref{eq:transf-sde-zvonkin} is continuous, bounded from above, and bounded away from zero. Moreover, $x \mapsto \tilde \sigma_t(x)$ is $\alpha$-H\"older continuous, uniformly in $t \in [0, T]$.
	\end{corollary}
	
	Let $c \colon [0, T] \times \R \times \R \to \R$ satisfy \Cref{ass:cost}. Suppose that $(b, \sigma)$ and $(\bar b, \bar \sigma)$ satisfy \Cref{ass:zvonkin}, and write $\mu, \nu$ for the laws of the solutions $X, \bar X$ of $\eqref{eq:sde-correlated}$. For $(\bar b, \bar \sigma)$, define functions $\bar u, \bar v$ analogously to $u, v$ and define $\bar Z_t = \bar u_t(\bar X_t)$ for all $t \in [0, T]$.

	In order to find the value $V_c(\mu, \nu)$ of \eqref{eq:bcot-cont}, we consider the following auxiliary problem. Write $\tilde \mu, \tilde \nu$ for the laws of $Z, \bar Z$, respectively, and define a cost function $\tilde c\colon [0, T] \times \R \times \R \to \R$ by 
	\begin{equation}\label{eq:transf-cost}
		\tilde c_t(z, \bar z) \coloneqq c_t(v_t(z), \bar v_t(\bar z)).
	\end{equation}
	Consider the problem
	\begin{equation}\label{eq:transformed-valued-z}
		V_{\tilde c}(\tilde \mu, \tilde \nu) \coloneqq \inf_{\pi \in \cplba(\tilde \mu, \tilde \nu)}\E^\pi\!\left[\int_0^T\tilde c_t(\omega_t, \bar \omega_t) \D t\right]\!.
	\end{equation}
	We now verify that $\tilde c$ inherits the continuity, polynomial growth and quasi-monotonicity properties from $c$.
		
	\begin{lemma}\label{lem:transformed-cost}
		Suppose that $(b, \sigma)$ and $(\bar b, \bar \sigma)$ satisfy \Cref{ass:zvonkin} and that $c \colon [0, T] \times \R \times \R \to \R$ satisfies \Cref{ass:cost} for some $p \geq 1$.
		Then the function $\tilde c_t \colon \R \times \R$ defined by \eqref{eq:transf-cost} also satisfies \Cref{ass:cost} for the same value of $p$.
	\end{lemma}
	
	\begin{proof}
		By \Cref{lem:G-properties}, $v, \tilde v \colon [0, T] \times \R \to \R$ are continuous. Thus, since $c$ is continuous on $[0, T] \times M$ for any compact $M \subseteq \R \times \R$, the same property holds for $\tilde c$.
		By assumption, there exist $p \geq 1$ and $K \geq 0$ such that, for all $t \in [0, T]$, $c_t$ satisfies the growth bound \eqref{eq:poly-growth}. Now fix $t \in [0, T]$. Then
		\begin{equation}
			|\tilde c_t(z, \bar z)| = |c_t(v_t(z), \bar v_t(\bar z))| \leq K(1 + |v_t(z)|^p + |\bar v_t(\bar z)|^p).
		\end{equation}
		Writing $L$ for the Lipschitz constant of $v_t$, we have $|v_t(z)| \leq L|z| + L|u_t(x_0)| + |x_0|$, and similarly for $\bar v_t(\bar z)$. Thus the growth bound \eqref{eq:poly-growth} holds for $\tilde c_t$ with a constant $\tilde K$ independent of $t$, and with power $p$. It remains to prove that $\tilde c_t$ satisfies \eqref{eq:superadditivity}. Since $v_t, \bar v_t$ are increasing by \Cref{lem:G-properties}, this follows immediately from the fact that $c_t$ satisfies \eqref{eq:superadditivity}.
	\end{proof}
	
	Now we show that \eqref{eq:transformed-valued-z} is equivalent to \eqref{eq:bcot-cont}.
	
	\begin{lemma}\label{lem:z-transformed-value}
		Suppose that $(b, \sigma)$ and $(\bar b, \bar \sigma)$ satisfy \Cref{ass:zvonkin}. Let $c \colon [0, T] \times \R \times \R \to \R$ be a measurable function. Then $V_c(\mu, \nu) = V_{\tilde c}(\tilde\mu, \tilde\nu)$.
	\end{lemma}
	
	\begin{proof}
		Let $\pi \in \cplba(\mu, \nu)$. By \citep[Proposition 2.2]{BaKaRo22}, $\pi = \Law(X, \bar X)$, where $(X, \bar X)$ is the unique strong solution of the system of SDEs \eqref{eq:sde-correlated} driven by some correlated Brownian motion $(W, \bar W)$. Now define $(Z, \bar Z)$ by $Z_t = u_t(X_t)$, $\bar Z_t = \bar u_t(\bar X_t)$, for all $t \in [0, T]$. Then $Z$ is the unique strong solution of the SDE \eqref{eq:transf-sde-zvonkin} driven by $W$, and $\bar Z$ of the analogous transformed SDE for $(\bar b, \bar \sigma)$ driven by $\bar W$. Hence, by \citep[Proposition 2.2]{BaKaRo22} again, we can define a bicausal coupling by $\tilde \pi \coloneqq \Law(Z, \bar Z) \in \cplba(\tilde \mu, \tilde \nu)$. By the definition of $\tilde c$, we can write
		\begin{equation}\label{eq:cost-equality}
			 \E^{\tilde\pi}\!\left[\int_0^T \tilde c_t(\omega_t, \bar \omega_t) \D t\right]\! = \E\!\left[\int_0^T c_t(v_t(Z_t), \bar v_t(\bar Z_t)) \D t\right]\! = \E^\pi\!\left[\int_0^T c_t(\omega_t, \bar \omega_t) \D t\right]\!.
		\end{equation}
		Starting from an arbitrary $\tilde \pi \in \cplba(\tilde \mu, \tilde \nu)$, we can construct $\pi \in \cplba(\mu, \nu)$ in a similar manner, via the inverse transform $v$. Thus, taking the infimum on both sides of \eqref{eq:cost-equality},we obtain $V_c(\mu, \nu) = V_{\tilde c}(\tilde \mu, \tilde \nu)$.
	\end{proof}
	
	Combining the preceding results, we obtain optimality of the synchronous coupling for \eqref{eq:bcot-cont}.
	
	\begin{theorem}\label{thm:zvonkin-optimality}
		Suppose that $(b, \sigma)$, $(\bar b, \bar \sigma)$ satisfy \Cref{ass:zvonkin}, and write $\mu, \nu$ for the laws of $X, \bar X$, respectively. 
		Let $c \colon [0, T] \times \R \times \R \to \R$ satisfy \Cref{ass:cost}. Then the synchronous coupling attains the value $V_c(\mu, \nu)$ of the bicausal optimal transport problem \eqref{eq:bcot-cont}.
	\end{theorem}
	 	
	\begin{proof}
		Thanks to \Cref{lem:z-transformed-value}, it suffices to prove that the synchronous coupling $\pi^\sync_{\tilde \mu, \tilde \nu}$ attains the value $V_{\tilde c}(\tilde \mu, \tilde \nu)$, where $\tilde \mu, \tilde \nu$ are the laws of solutions of transformed SDEs of the form \eqref{eq:transf-sde-zvonkin}. By \Cref{cor:lipschitz-coeffs}, the coefficients of these SDEs are bounded and continuous, and pathwise uniqueness holds by \citet[Theorem 2]{YW2}. Thus the coefficients of each of the transformed SDEs satisfy \Cref{ass:regular}. By \Cref{lem:transformed-cost}, $\tilde c$ satisfies the conditions of \Cref{cor:regular}, and applying this result completes the proof.
	\end{proof}
		
	\subsection{Lack of monotonicity and counterexample}\label{sec:counterexamples}
		
		In the proof of \Cref{thm:zvonkin-optimality}, we consider transformed SDEs with H\"older continuous diffusion coefficients that may not be Lipschitz. Recall the monotone Euler--Maruyama scheme from \Cref{sec:num-sdes} that was introduced in \cite[Definition 3.13]{BaKaRo22}. We show that this scheme is not stochastically monotone for SDEs whose diffusion coefficient is not Lipschitz. Thus we could not apply \Cref{prop:general-optimality} with the monotone Euler--Maruyama scheme to prove optimality in this case.
		
		\begin{proposition}\label{prop:non-monotone}
			Let $\sigma\colon \R \to (0, \infty)$ be $\alpha$-H\"older continuous for some $\alpha \in [1/2, 1)$, bounded from above, and bounded away from zero. Consider the SDE $\D Z_t = \sigma(Z_t) \D W_t$, with $Z_0 = z_0 \in \R$. If $\sigma$ is not Lipschitz continuous, then the monotone Euler--Maruyama scheme for $Z$ is not stochastically monotone for any step size.
		\end{proposition}
		
		\begin{remark}
			In the setting of \Cref{prop:non-monotone}, the monotone Euler--Maruyama scheme coincides with the transformed monotone semi-implicit Euler--Maruyama scheme, since there is no drift term.
		\end{remark}
		
		\begin{proof}[Proof of \Cref{prop:non-monotone}]
			The conditions of the proposition are sufficient to guarantee the existence of a unique strong solution $Z$ of the SDE by, e.g.~\citep[Theorem 4]{Zv74}. For any $N \in \N$ and $h = {\frac TN}$, let $Z^h$ denote the monotone Euler--Maruyama scheme for $Z$ with step size $h$, and let $A_h$ be the truncation level given in \Cref{def:trunc-bm}. Note that, since $\sigma$ is not Lipschitz continuous, there exist $z, \bar z \in \R$ with $z < \bar z$ such that
			\begin{equation}\label{eq:non-lip}
				|\sigma(\bar z) - \sigma(z)| > A_h^{-1}(\bar z - z) > 0.
			\end{equation}
			Now consider the random variables $Y \coloneqq z + \sigma(z)\xi^h$, $\bar Y \coloneqq \bar z + \sigma(\bar z)\xi^h$, where $\xi^h$ is a random variable with the same law as an increment of the truncated Brownian motion $W^h$. For all $a \in \R$, we have
			\begin{equation}
				\P[Y \leq a] = \P\!\left[\xi^h \leq \frac{a - z}{\sigma(z)}\right]\! = \begin{cases}
					0, & a \leq z - A_h \sigma(z),\\
					2\Phi\!\left(h^{-\frac12}\frac{a - z}{\sigma(z)}\right)\!, & a \in (z - A_h\sigma(z), z + A_h\sigma(z)),\\
					1, & a \geq z + A_h \sigma(z),
				\end{cases}
			\end{equation}
			where $\Phi\colon \R \to [0, 1]$ is the distribution function of a standard Gaussian, which is strictly increasing. The same holds for $\bar Y$ with $z$ replaced by $\bar z$. We aim to show that there exist $a_\ast, a^\ast \in \R$ such that $\P[\bar Y \leq a_\ast] < \P[Y \leq a_\ast]$ and $\P[Y \leq a^\ast] < \P[\bar Y \leq a^\ast]$. This would imply that $Z^h$ is neither stochastically decreasing nor increasing. Suppose first that $\sigma(\bar z) \ge \sigma(z)$. Then rearranging \eqref{eq:non-lip} gives
			\begin{equation}
				\bar z - A_h\sigma(\bar z) < z - A_h \sigma(z).
			\end{equation}
			So there exists $a^\ast \in (\bar z - A_h \sigma(\bar z), z - A_h \sigma(z))$, and we see that
			\begin{equation}
				\P[Y \leq a^\ast] = 0, \; \text{while} \quad \P[\bar Y \leq a^\ast] > 0.
			\end{equation}
			Set $a_\ast = z + A_h \sigma( z)$ and note that $a_\ast <  \bar z + A_h \sigma(\bar z)$. Therefore $\P[Y \leq a_\ast]= 1$ and $\P[\bar Y \leq a_\ast] < 1$.
			
			On the other hand, if $\sigma(\bar z) < \sigma(z)$, then $\bar z + A_h\sigma(\bar z) < z + A_h\sigma( z)$, and so there exists $a^\ast$ such that $\P[\bar Y \leq a^\ast] = 1$ and $\P[Y \leq a^\ast] < 1$. Now setting $a_\ast = \bar z - A_h \sigma(\bar z) > z - A_h\sigma(z)$ gives $\P[\bar Y \leq a_\ast] = 0$ and $\P[Y \le a_\ast] > 0$. Hence $Z^h$ is not stochastically monotone.
		\end{proof}
		
		A natural generalisation of stochastic monotonicity is monotonicity with respect to second order stochastic dominance. We now provide an example to show that, for two marginals that are increasing in second order stochastic dominance, the Knothe--Rosenblatt rearrangement may fail to be optimal for the bicausal transport problem, if one of the marginals is not stochastically increasing.
		
		\begin{example}[second order stochastic dominance is insufficient for optimality]\label{ex:second-order-mono}
			Consider the discrete-time Markov processes $X = (X_1, X_2)$, $\bar X = (\bar X_1, \bar X_2)$, with first marginals given by $\P[X_1 = -1/2] = \P[X_1 = 1/2] = \P[\bar X_1 = -1/2] = \P[\bar X_1 = 1/2] = 1/2$, and conditional second marginals given by
			$\P[X_2 = 2 \mid X_1 = -1/2] = \P[X_2 = -2 \mid X_1 = -1/2] = 1/2$, $\P[X_2 = 0 \mid X_1 = 1/2] = 1$, and
			$\P[\bar X_2 = -2 \mid \bar X_1 = - 1/2] = \P[\bar X_2 = 0 \mid \bar X_1 = -1/2] = \P[\bar X_2 = -2 \mid \bar X_1 = 1/2] = \P[\bar X_2 = 2 \mid \bar X_1 = 1/2] = 1/2$.
			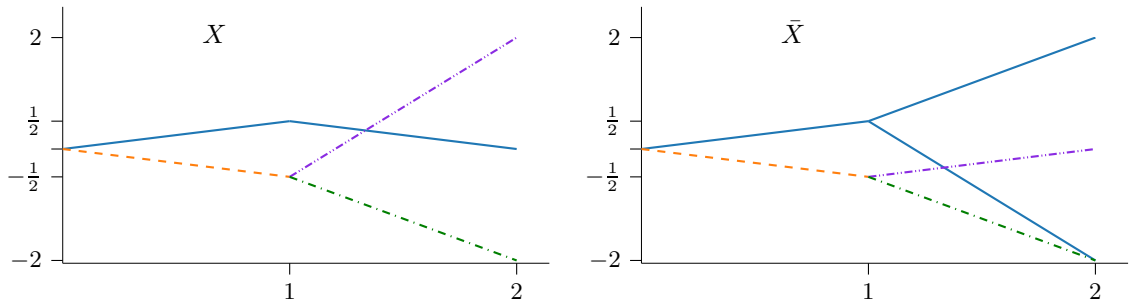
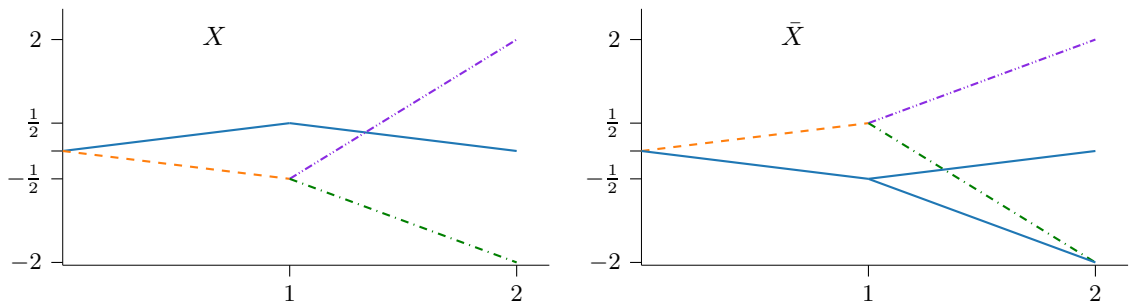
\begin{figure}[h]
				\begin{subfigure}{\textwidth}
					\centering
					\begin{tikzpicture}[baseline={(0,0)}]

\definecolor{darkgray176}{RGB}{176,176,176}
\definecolor{darkorange25512714}{RGB}{255,127,14}
\definecolor{steelblue31119180}{RGB}{31,119,180}
\definecolor{darkgreen}{RGB}{0, 127, 0}
\definecolor{blueviolet}{RGB}{138,43,226}

\begin{axis}[
tick align=outside,
tick pos=left,
x grid style={darkgray176},
xmin=0.0, xmax=2.01,
xtick style={color=black},
xtick = {1.0, 2.0},
xticklabels = {\small{$1$}, \small{$2$}},
y grid style={darkgray176},
ymin=-2.05, ymax=2.01,
ytick style={color=black},
ytick = {-2.0, -0.5, 0.0, 0.5, 2.0},
yticklabels = {\small{$-2$}, \small{$-\frac12$}, , \small{$\frac12$}, \small{$2$}}
]
\draw[-, steelblue31119180, thick](0.0, 0.0) -- (1.0, 0.5);
\draw[dashed, darkorange25512714, thick](0.0, 0.0) -- (1.0, -0.5);
\draw[-, steelblue31119180, thick](1.0, 0.5) -- (2.0, 0.0);
\draw[densely dashdotdotted, blueviolet, thick](1.0, -0.5) -- (2.0, 2.0);
\draw[dashdotted, darkgreen, thick](1.0, -0.5) -- (2.0, -2.0);
\end{axis}

\coordinate [label = {$X$}] (X) at (2.0, 2.8);

\end{tikzpicture}
					\quad
					\begin{tikzpicture}[baseline={(0,0)}]

\definecolor{darkgray176}{RGB}{176,176,176}
\definecolor{darkorange25512714}{RGB}{255,127,14}
\definecolor{steelblue31119180}{RGB}{31,119,180}
\definecolor{darkgreen}{RGB}{0, 127, 0}
\definecolor{blueviolet}{RGB}{138,43,226}

\begin{axis}[
tick align=outside,
tick pos=left,
x grid style={darkgray176},
xmin=0.0, xmax=2.01,
xtick style={color=black},
xtick = {1.0, 2.0},
xticklabels = {\small{$1$}, \small{$2$}},
y grid style={darkgray176},
ymin=-2.05, ymax=2.01,
ytick style={color=black},
ytick = {-2.0, -0.5, 0.0, 0.5, 2.0},
yticklabels = {\small{$-2$}, \small{$-\frac12$}, , \small{$\frac12$}, \small{$2$}}
]
\draw[-, steelblue31119180, thick](0.0, 0.0) -- (1.0, 0.5);
\draw[dashed, darkorange25512714, thick](0.0, 0.0) -- (1.0, -0.5);
\draw[-, steelblue31119180, thick](1.0, 0.5) -- (2.0, 2.0);
\draw[-, steelblue31119180, thick](1.0, 0.5) -- (2.0, -2.0);
\draw[densely dashdotdotted, blueviolet, thick](1.0, -0.5) -- (2.0, 0.0);
\draw[dashdotted, darkgreen, thick](1.0, -0.5) -- (2.0, -2.0);
\end{axis}

\coordinate [label = {$\bar X$}] (Y) at (2.0, 2.8);

\end{tikzpicture}
					\caption{The laws of $X$ and $\bar X$ coupled with the Knothe--Rosenblatt rearrangement $\pi^\kr_{\mu, \nu}$.}
					\label{subfig:kr-example}
				\end{subfigure}
				\vskip 3ex
				\begin{subfigure}{\textwidth}
					\centering
					\begin{tikzpicture}[baseline={(0,0)}]

\definecolor{darkgray176}{RGB}{176,176,176}
\definecolor{darkorange25512714}{RGB}{255,127,14}
\definecolor{steelblue31119180}{RGB}{31,119,180}
\definecolor{darkgreen}{RGB}{0,127,0}
\definecolor{blueviolet}{RGB}{138,43,226}

\begin{axis}[
tick align=outside,
tick pos=left,
x grid style={darkgray176},
xmin=0.0, xmax=2.01,
xtick style={color=black},
xtick = {1.0, 2.0},
xticklabels = {\small{$1$}, \small{$2$}},
y grid style={darkgray176},
ymin=-2.05, ymax=2.01,
ytick style={color=black},
ytick = {-2.0, -0.5, 0.0, 0.5, 2.0},
yticklabels = {\small{$-2$}, \small{$-\frac12$}, , \small{$\frac12$}, \small{$2$}}
]
\draw[-, steelblue31119180, thick](0.0, 0.0) -- (1.0, 0.5);
\draw[dashed, darkorange25512714, thick](0.0, 0.0) -- (1.0, -0.5);
\draw[-, steelblue31119180, thick](1.0, 0.5) -- (2.0, 0.0);
\draw[densely dashdotdotted, blueviolet, thick](1.0, -0.5) -- (2.0, 2.0);
\draw[dashdotted, darkgreen, thick](1.0, -0.5) -- (2.0, -2.0);
\end{axis}

\coordinate [label = {$X$}] (X) at (2.0, 2.8);

\end{tikzpicture}
					\quad
					\begin{tikzpicture}[baseline={(0,0)}]

\definecolor{darkgray176}{RGB}{176,176,176}
\definecolor{darkorange25512714}{RGB}{255,127,14}
\definecolor{steelblue31119180}{RGB}{31,119,180}
\definecolor{darkgreen}{RGB}{0,127,0}
\definecolor{blueviolet}{RGB}{138,43,226}

\begin{axis}[
tick align=outside,
tick pos=left,
x grid style={darkgray176},
xmin=0.0, xmax=2.01,
xtick style={color=black},
xtick = {1.0, 2.0},
xticklabels = {\small{$1$}, \small{$2$}},
y grid style={darkgray176},
ymin=-2.05, ymax=2.01,
ytick style={color=black},
ytick = {-2.0, -0.5, 0.0, 0.5, 2.0},
yticklabels = {\small{$-2$}, \small{$-\frac12$}, , \small{$\frac12$}, \small{$2$}}
]
\draw[dashed, darkorange25512714, thick](0.0, 0.0) -- (1.0, 0.5);
\draw[-, steelblue31119180, thick](0.0, 0.0) -- (1.0, -0.5);
\draw[densely dashdotdotted, blueviolet, thick](1.0, 0.5) -- (2.0, 2.0);
\draw[dashdotted, darkgreen, thick](1.0, 0.5) -- (2.0, -2.0);
\draw[-, steelblue31119180, thick](1.0, -0.5) -- (2.0, 0.0);
\draw[-, steelblue31119180, thick](1.0, -0.5) -- (2.0, -2.0);
\end{axis}

\coordinate [label = {$\bar X$}] (Y) at (2.0, 2.8);

\end{tikzpicture}
					\caption{The laws of $X$ and $\bar X$ coupled with an alternative bicausal coupling $\pi^\antitone_{\mu, \nu}$.}
					\label{subfig:at-example}
				\end{subfigure}
				\caption{Two couplings of the laws of $X$ and $\bar X$. In each case, paths with the same colour and line style are coupled with each other.}
			\end{figure}
			
			We claim that $X$ is increasing with respect to second order stochastic dominance, but not with respect to first order stochastic dominance. Denote by $F_-, F_+\colon \R \to [0, 1]$ the distribution functions of the second marginal of $X$, conditional on $X_1 = -\frac12$ and $X_1 = \frac12$, respectively; that is
			\begin{equation}
				F_-(x) =
				\begin{cases}
					0, & x < -2,\\
					\frac12, & x \in [-2, 2),\\
					1, & x \geq 2,
				\end{cases}
				\quad
				F_+(x) =
				\begin{cases}
					0, & x < 0,\\
					1, & x \ge 0.
				\end{cases}
			\end{equation}
			We see that, for any $x \in [0, 2)$, $F_-(x) < F_+(x)$, and so $X$ is not stochastically increasing. However, integrating, we find that 
			\begin{equation}
				\int_{-\infty}^x(F_-(u) - F_+(u))\D u =
				\begin{cases}
					0, & x < -2,\\
					\frac{x + 2}{2} \ge 0 , & x \in [-2, 0),\\
					\frac{2 - x}{2} \ge 0 , & x \in [0, 2),\\
					0, & x \geq 2.
				\end{cases}
			\end{equation}
			Hence $X$ is increasing with respect to second order stochastic dominance. By inspection of the conditional distribution functions for $\bar X$, we see that $\bar X$ is stochastically monotone, and therefore also increasing with respect to second order stochastic dominance.
			Thus $X$ and $\bar X$ are co-monotone with respect to \emph{second order} stochastic dominance, but they are \emph{not} stochastically co-monotone (i.e.~co-monotone with respect to first order stochastic dominance).
			
			Let $\mu = \Law(X)$ and $\nu = \Law(\bar X)$. We now show that second order monotonicity is not sufficient for optimality of the Knothe--Rosenblatt rearrangement.
			We aim to solve \eqref{eq:bcot} for the quadratic cost; that is we want to find
			\begin{equation}
				\AW_2^2(\mu, \nu) = \inf_{\substack{\pi \in \cplba(\mu, \nu)\\\Law(X, \bar X) = \pi}} \E\!\left[|X_1 - \bar X_1|^2 + |X_2 - \bar X_2|^2\right]\!.
			\end{equation}
			
			We first compute the cost induced by the Knothe--Rosenblatt rearrangement $\pi^{\kr}_{\mu, \nu}$, shown in \Cref{subfig:kr-example}. Under this coupling, the first marginals are coupled monotonically, and so we have $\E[|X_1 - \bar X_1|^2] = 0$. Now, conditional on the event that $(X_1, \bar X_1) = (\frac12, \frac12)$, we couple the second marginals monotonically and compute $\E[|X_2 - \bar X_2|^2] = \frac12(|0 - 2|^2 + |0+2|^2) = 4$. Conditional on the complementary event $(X_1, \bar X_1) = (-\frac12, -\frac12)$, we find $\E[|X_2 - \bar X_2|^2] = \frac12(|2 - 0|^2 + |-2 + 2|^2) = 2$. We conclude that, for $\Law(X, \bar X) = \pi^\kr_{\mu, \nu}$,
			\begin{equation}
				\E\!\left[|X_1 - \bar X_1|^2 + |X_2 - \bar X_2|^2\right]\! = 0 + \frac12(4 + 2) = 3.
			\end{equation}
			
			Now consider an alternative bicausal coupling $\pi^\antitone_{\mu, \nu}$, defined by taking the \emph{antitone} coupling between the first marginals (i.e.~the opposite of the monotone coupling) and then proceeding as in the Knothe--Rosenblatt rearrangement by coupling the conditional second marginals monotonically, as shown in \Cref{subfig:at-example}. Under this coupling there is a non-zero contribution to the cost from the first marginals: $\E[|X_1 - \bar X_1|^2] = 1$. Conditional on $(X_1, \bar X_1) = (\frac12, - \frac12)$, we have $\E[|X_2 - \bar X_2|^2] = \frac12(|0 - 0|^2 + |0 + 2|^2) = 2$, and conditional on the complement, $\E[|X_2 - \bar X_2|^2] = \frac12(|2 - 2|^2 + |-2 + 2|^2) = 0$. Thus taking $\Law(X, \bar X) = \pi^\antitone_{\mu, \nu}$ gives a total cost of
			\begin{equation}
				\E\!\left[|X_1 - \bar X_1|^2 + |X_2 - \bar X_2|^2\right]\! = 1 + \frac12(2 + 0) = 2 < 3.
			\end{equation}
			Hence the Knothe--Rosenblatt rearrangement is not optimal.
		\end{example}
		
\section{General optimality of the synchronous coupling}\label{sec:mixing-assumptions}

	In the previous sections we have identified different classes of coefficients of the SDEs \eqref{eq:sde-correlated} for which the synchronous coupling is optimal for \eqref{eq:bcot-cont}.
	In fact, the same optimality result holds for SDEs whose coefficients belong to two different classes; that is they satisfy one of \Cref{ass:regular,ass:growth-disc,ass:zvonkin}.
	
	\begin{theorem}\label{thm:mixed-assumptions}
		Suppose that $(b, \sigma)$ and $(\bar b, \bar \sigma)$ each satisfy any one of \Cref{ass:regular,ass:growth-disc,ass:zvonkin}, where the two pairs of coefficients may each satisfy a different assumption. Write $\mu, \nu$ for the laws of $X, \bar X$, respectively. 
		Let $c \colon [0, T] \times \R \times \R \to \R$ satisfy \Cref{ass:cost}. Then the synchronous coupling attains the value $V_c(\mu, \nu)$ of the bicausal optimal transport problem \eqref{eq:bcot-cont}.
	\end{theorem}
	
	\begin{proof}
		First suppose that $(b, \sigma)$ satisfies either the global Lipschitz condition \Cref{ass:lipschitz} (this implies \Cref{ass:regular}) or \Cref{ass:growth-disc}. Then take $X^h$ to be the monotone Euler--Maruyama scheme defined in \cite[Definition 3.13]{BaKaRo22}, or the transformed monotone semi-implicit Euler--Maruyama scheme defined in \Cref{def:mono-si-em}, respectively.		
		If $(\bar b, \bar \sigma)$ also satisfies one of the above conditions, then taking $\bar X^h$ to be the respective numerical scheme, $(X, \bar X)$ and $(X^h, \bar X^h)$ satisfy \Cref{ass:comon-scheme}. Then we can conclude by \Cref{prop:general-optimality}.
		If $(\bar b, \bar \sigma)$ satisfies \Cref{ass:regular}, but not necessarily a global Lipschitz condition, then we can approximate $(\bar b, \bar \sigma)$ locally uniformly in space by a sequence of Lipchitz functions that have uniform linear growth. Then we conclude by applying the stability result \Cref{lem:stability} to the previous case.
				
		Now suppose that $(b, \sigma)$ satisfies either \Cref{ass:regular} or \Cref{ass:growth-disc}, and $(\bar b, \bar \sigma)$ satisfies \Cref{ass:zvonkin}. Then we adapt the proof of \Cref{thm:zvonkin-optimality} as follows. Take the transformation $\bar u$ to be as in the proof of \Cref{thm:zvonkin-optimality}, and replace the transformation $u$ with the identity. Then, following the proof of \Cref{thm:zvonkin-optimality}, we show that \eqref{eq:bcot-cont} is equivalent to a bicausal transport problem between the laws of transformed SDEs, whose coefficients satisfy either \Cref{ass:regular} or \Cref{ass:growth-disc}. From the cases considered above, we know that the synchronous coupling is optimal for the transformed problem, and by equivalence of the problems, the synchronous coupling is also optimal for \eqref{eq:bcot-cont}.
				
		Noting that the roles of $\mu$ and $\nu$ are interchangeable, all remaining cases are covered by \Cref{cor:regular} and \Cref{thm:optimality-growth-disc,thm:zvonkin-optimality}.
	\end{proof}

\section{Computation of the adapted Wasserstein distance}\label{sec:numerics}
	
	Our results immediately suggest an efficient method for computing the value of a bicausal optimal transport problem. This is in sharp contrast to the general situation in classical and constrained optimal transport problems, when the optimiser is not known. Therefore our results contribute to the efficiency of computations in applications, such as in the calculation of robust bounds for stochastic optimisation problems, as discussed in \Cref{sec:finance}. In \Cref{sec:examples} we present numerical examples, in which we explicitly compute adapted Wasserstein distances.
	
		In classical optimal transport, entropic regularisation is used in computation. Sinkhorn's algorithm can be used to efficiently compute the value of the regularised problem, but not of the original unregularised problem; see \citet{Cu13,PeCu19}. 
		For bicausal optimal transport in discrete time, \citet{PiWe21,EcPa22} introduce two different adapted analogues of Sinkhorn's algorithm. However, each of these algorithms either include a discretisation of the marginal measures or work only on a discrete state space, and generalising them to both continuous time and space would be prohibitively costly. 
	Under rather strong assumptions on the coefficients, \citet{BiTa19} consider a stochastic control problem for SDEs, which is shown to be equivalent to finding the adapted $2$-Wasserstein distance in \citep[Proposition 2.2]{BaKaRo22}. This control formulation gives rise to a Hamilton--Jacobi--Bellman equation, which one could solve numerically.
	
	\subsection{Computation via the synchronous coupling}\label{sec:computation-sync}
	
		Let $c \colon [0, T] \times \R \times \R \to \R$ satisfy \Cref{ass:cost}. Suppose that $(b, \sigma)$ satisfies one of \Cref{ass:regular,ass:growth-disc,ass:zvonkin}, and that $(\bar b, \bar \sigma)$ also satisfies one of these assumptions. Then, by \Cref{thm:mixed-assumptions}, the synchronous coupling is optimal for $V_c(\mu, \nu)$. Writing $(X, \bar X)$ for the solution of the SDEs \eqref{eq:sde-correlated} driven by a common Brownian motion $W$, we thus have
		\begin{equation}
			V_c(\mu, \nu) = \E\!\left[\int_0^T c_t(X_t, \bar X_t) \D t\right]\!.
		\end{equation}
		In particular, taking $c_t(x, y) = |x - y|^p$, for all $t \in [0, T]$, $x, y \in \R$, the adapted Wasserstein distance is given by
		\begin{equation}
			\AW_p^p(\mu, \nu) = \E\!\left[\int_0^T |X_t - \bar X_t|^p \D t\right]\!.
		\end{equation}
		For $N \in \N$, set $h = T/N$ and take $X^h$, $\bar X^h$ to be c\`adl\`ag processes adapted to the natural filtration of $W$ such that $X^h$ (resp.~$\bar X^h$) approximates $X$ (resp.~$\bar X$) strongly in $L^p$. Define $S^h \coloneqq \E[\int_0^T|X^h_t - \bar X^h_t|^p\D t]^{\frac 1p}$. Then
		\begin{equation}\label{eq:AW-decoupling}
			\bigl|\AW_p(\mu, \nu) - S^h\bigr| \leq \E\biggl[\int_0^T |X_t - X^h_t|^p \D t \biggr]^{\frac 1p} + \E\biggl[\int_0^T |\bar X_t - \bar X^h_t|^p \D t \biggr]^{\frac 1p} \xrightarrow{h \to 0} 0.
		\end{equation}
		Note that each term in the sum on the right hand side of \eqref{eq:AW-decoupling} depends on only one of the processes $X, \bar X$ and so these terms can be estimated separately. Thus the problem of computing the adapted Wasserstein distance is reduced to the well-studied problem of numerically solving two one-dimensional SDEs. Monte Carlo methods can then be applied to compute these approximations.
		
		There are now two error sources to consider in the numerical computation of $\AW_p(\mu, \nu)$: the convergence rate of the schemes $X^h, \bar X^h$, and the computational cost of Monte Carlo estimation for $S^h$. Given numerical schemes for which we have strong convergence rates, one can efficiently implement a multi-level Monte Carlo (MLMC) method to reduce the overall computational cost; see, for example, ~\citet{GiSz13}.

		\begin{remark} \label{rem:no-mon-needed}
			The bound \eqref{eq:AW-decoupling} shows that, for each SDE, we may take any numerical scheme that converges strongly in $L^p$. For example, for any SDE with Lipschitz coefficients, we can apply the Euler--Maruyama scheme defined in \Cref{sec:num-sdes}.
			In the case of discontinuous and exponentially growing drift, we proved strong $L^p$ convergence of the transformed (monotone) semi-implicit Euler--Maruyama scheme in \Cref{thm:implicit-transformed-convergence}, which allows us to compute the adapted Wasserstein distance when one of the SDEs has coefficients satisfying \Cref{ass:growth-disc}. When additional conditions hold, however, we may exploit the existing convergence results that are discussed in \Cref{rem:growth-disc} to improve computational efficiency. While stochastic monotonicity of the scheme is essential to the proof of optimality in \Cref{thm:optimality-growth-disc}, this condition is not needed for the computation of the adapted Wasserstein distance. Under \Cref{ass:zvonkin}, suitable numerical schemes are discussed in \Cref{rem:zvonkin}.
		\end{remark}
		
		\begin{example}
			As a toy example, consider the situation that $b, \bar b$ are Lipschitz and $\sigma, \bar \sigma$ are strictly positive constants. In this case, taking $X^h$, $\bar X^h$ as the standard Euler--Maruyama schemes for $X$, $\bar X$, respectively, \eqref{eq:AW-decoupling} implies that $|\AW_p(\mu, \nu) - S^h| \leq C(p, T) N^{-1}$. This follows from the order 1 strong convergence of the Euler--Maruyama scheme in the case of Lipschitz drift and additive noise, where it is equivalent to the Milstein scheme \cite{Mi75}.
		\end{example}
				
	\subsection{Numerical examples}\label{sec:examples}
	
	In each of the examples below, we consider an SDE with solution $X$ and, for $k \in \N$, the solution $\bar X^{(k)}$ of an SDE obtained by some perturbation of the coefficients of the original SDE, where the magnitude of the perturbation increases with the parameter $k$. We fix the terminal time $T = 1$ and consider the adapted Wasserstein distance $\AW_2(\mu, \nu^{(k)})$, where $\mu$ and $\nu^{(k)}$ denote the laws of $X$ and $\bar X^{(k)}$, respectively. Supposing that the coefficients of the SDEs for $X$ and $\bar X^{(k)}$ satisfy the assumptions of \Cref{thm:mixed-assumptions}, this distance is attained by the synchronous coupling. For $h > 0$, let $X^h$ (resp.~$\bar X^{(k),h}$) denote numerical approximations of $X$ (resp.~$\bar X^{(k)}$) that converge strongly in $L^2$ with known rates and are each driven by a common one-dimensional Brownian motion. We fix the step size $h = 2^{-12}$ for the numerical approximations and average over $2^{12}$ sample paths. In light of the bound \eqref{eq:AW-decoupling}, we can then approximate the squared adapted Wasserstein distance $\AW_2^2(\mu, \nu^{(k)})$ by the Monte Carlo estimate
	\begin{equation}\label{eq:aw-estimate}
		\widehat \AW_2^2(\mu, \nu^{(k)}) =  \frac{1}{2^{12}(2^{12}+1)}\sum_{i=1}^{2^{12}}\sum_{j=0}^{2^{12}}|X^h_{jh}(\omega_i) - \bar X_{jh}^{(k),h}(\omega_i)|^2.
	\end{equation}
	In each example below, we plot $\widehat\AW_2^2(\mu, \nu^{(k)})$ against the size of the perturbation.
	
	\begin{example}[An SDE with discontinuous drift]\label{ex:num-disc-drift}
		Consider SDEs
		\begin{equation}\label{eq:disc}
			\begin{split}			
				\ds X_t & = \ds W_t, \quad X_0 = x_0,\\
				\ds \bar X_t^{(k)} &= \frac{k}{10}\sign(\bar X_t^{(k)}) \D t +  \ds W_t, \quad \bar X^{(k)}_0 = x_0, \quad k\in\{1,\dots,10\}.
			\end{split}
		\end{equation}
		The solution $X$ of the first SDE in \eqref{eq:disc} can be approximated by the Euler--Maruyama scheme with strong convergence order $1$. For the SDEs in the second line of \eqref{eq:disc}, however, \citet{MuYa20b,PrScSz23,ElMuYa24} show an upper bound of $3/4$ for the strong convergence rate of any numerical scheme that uses finitely many approximations of the driving Brownian motion (such as the Euler--Maruyama scheme). This is one of the structural properties that distinguishes the class of SDEs with discontinuous drift. For the adapted Wasserstein distance between the induced measures, however, we observe no gap comparable to the gap between the strong convergence rates. Each of the SDEs in \eqref{eq:disc} satisfies \Cref{ass:growth-disc}, and so \Cref{thm:mixed-assumptions} implies that the synchronous coupling attains the adapted Wasserstein distance between the laws of the solutions. Thus, we can estimate this distance by \eqref{eq:aw-estimate}, taking $X^h, \bar X^{(k),h}$ to be the respective Euler--Maruyama schemes, each driven by a common Brownian motion. The relationship between this approximation of the adapted Wasserstein distance and the size of the jump in the drift is shown in \Cref{fig:disc}.
			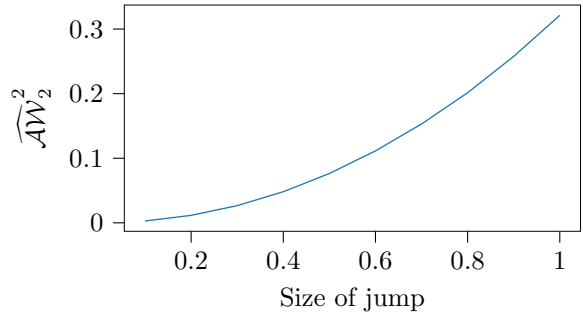
\begin{figure}[ht]
				\centering
\begin{tikzpicture}[baseline={(0,0)}]

\definecolor{darkgrey176}{RGB}{176,176,176}
\definecolor{steelblue31119180}{RGB}{31,119,180}

\begin{axis}[
scale = 0.8,
tick align=outside,
tick pos=left,
grid style={darkgrey176},
xmin=0.0550000000000001, xmax=1.045,
xtick style={color=black},
ymin=-0.0130603025651276, ymax=0.337186671227656,
ytick style={color=black},
xlabel = {Size of jump},
ylabel = {$\widehat\AW_2^2$},
axis lines = box
]
\addplot [semithick, steelblue31119180]
table {%
0.1 0.00286001442545347
0.2 0.0115987181020346
0.3 0.0266394332318566
0.4 0.0481866189570397
0.5 0.0763767368429628
0.6 0.111403013254899
0.7 0.153030928449135
0.8 0.201457845500375
0.9 0.257714189576319
1 0.321266354237075
};
\end{axis}

\end{tikzpicture}
				\caption{Approximation of the adapted Wasserstein distance for the SDEs \eqref{eq:disc}.} 
				\label{fig:disc}
			\end{figure}
	\end{example}
				
	\begin{example}[CIR model]\label{ex:cir}
		Consider a CIR process $(X_t)_{t \in [0, 1]}$, solving the SDE
		\begin{equation}\label{eq:cir}
			\ds X_t = \kappa(\eta - X_t)\D t + \gamma \sqrt X_t \D W_t, \quad X_0 = x_0,
		\end{equation}
		for some $\kappa \in \R$, $\eta \in [0, \infty)$, $\gamma, x_0 \in (0, \infty)$. We suppose that $2\kappa\eta \geq \gamma^2$, so that $X_t > 0$ for all $t \in [0, 1]$ by Feller's test; see, for example, \citep[Theorem 5.29]{KaSh91}. We call the parameter $\eta$ the mean reversion level, $\kappa$ the mean reversion speed, and $\gamma$ the diffusion parameter.
		Such a process is commonly used to model interest rates.
		
		Here we have a Lipschitz-continuous drift coefficient and a $1/2$-H\"older continuous diffusion coefficient. The classical result of \citet[Theorem 1]{YW} gives pathwise uniqueness for \eqref{eq:cir}. Thus the coefficients of \eqref{eq:cir} satisfy \Cref{ass:regular}. Consider also the process $\bar X^{(k)}$ solving the same SDE \eqref{eq:cir} with perturbed parameters $\bar \kappa, \bar \eta, \bar \gamma$ and $\bar X^{(k)}_0 = x_0$. Since the coefficients of this SDE also satisfy \Cref{ass:regular}, we have by \Cref{thm:mixed-assumptions} that the synchronous coupling attains the adapted Wasserstein distance between their laws.
		
		Since the CIR process remains positive, its approximation requires a modification to the standard Euler--Maruyama scheme. For example, we can take the symmetrised Euler--Maruyama scheme that is defined in \citet{BeBoDi08} by taking the absolute value at each step of the Euler--Maruyama scheme. Under the condition that $2\kappa \eta \geq \gamma^2$, \citep[Theorem 2.2]{BeBoDi08} shows that the symmetrised Euler--Maruyama scheme converges strongly in $L^2$ at rate $1/2$. Note that \cite[Theorem 1.1]{HeJe19} shows that this convergence rate does not hold when $2\kappa \eta < \gamma^2$.
		In this example we fix the parameters $\kappa =  \eta = \gamma = 1$ for the process $X$, and for the process $\bar X^{(k)}$ we perturb each of these parameters in turn, fixing the other two equal to those of the process $X$. For each perturbation, we approximate $X, \bar X^{(k)}$ by the numerical scheme described above, taking the same Brownian increments in both $X^h$ and $\bar X^{(k), h}$. We then estimate the adapted Wasserstein distance between the laws of $X$ and $\bar X^{(k)}$ by \eqref{eq:aw-estimate}.
					
		Figure \ref{fig:cir} shows an approximation of the adapted Wasserstein distance plotted against the absolute size of the perturbation for each parameter on the left-hand side, and the same on a log-scale on the right-hand side.
		As expected, we observe clearly in \Cref{fig:cir} that increasing the perturbation in each of the parameters increases the adapted Wasserstein distance between the reference measure and the measure induced by the process with perturbed parameters. We observe that a perturbation of the diffusion parameter has a bigger effect than that of the mean reversion level, which has a much bigger effect than a perturbation of the mean reversion speed. This behaviour depends, however, on the choice of the parameters for the process $X$.
		As the log-plot confirms, the adapted Wasserstein distance grows exponentially in the size of the perturbation.
		This observation confirms that the adapted Wasserstein distance successfully captures the difference between CIR models with different parameters.
		\begin{figure}[ht]
			\centering
\begin{tikzpicture}[baseline={(0,0)}]

\definecolor{darkgrey176}{RGB}{176,176,176}
\definecolor{green01270}{RGB}{0,127,0}
\definecolor{steelblue31119180}{RGB}{31,119,180}
\definecolor{darkorange25512714}{RGB}{255,127,14}
\definecolor{blueviolet}{RGB}{138,43,226}

\begin{axis}[
scale = 0.8,
tick align=outside,
tick pos=left,
x grid style={darkgrey176},
xmin=-0.00193603515625005, xmax=0.0524731445312501,
xtick style={color=black},
y grid style={darkgrey176},
ymin=-3.72773595943008e-05, ymax=0.00078300180256209,
ytick style={color=black},
xlabel = {Size of perturbation},
ylabel = {$\widehat\AW_2^2$},
axis lines = box,
ticklabel style={
        /pgf/number format/fixed,
        /pgf/number format/precision=4
},
scaled ticks=false
]
\addplot [semithick, steelblue31119180]
table {%
0.05 0.000445008837816177
0.0375 0.000250673572859574
0.0249999999999999 0.000111296213234569
0.015625 4.35038153545894e-05
0.00937499999999991 1.56549651150653e-05
0.00546875000000002 5.32832572989199e-06
0.00312500000000004 1.73961712344967e-06
0.00175781249999996 5.50469585604227e-07
0.0009765625 1.69890433240962e-07
0.000537109374999956 5.13931258627746e-08
};
\addplot [semithick, green01270, dash pattern=on 5.55pt off 2.4pt]
table {%
0.05 6.79976384591772e-05
0.0375 4.00477007179355e-05
0.0249999999999999 1.72227312866576e-05
0.015625 6.87277985716459e-06
0.00937499999999991 2.44186309605702e-06
0.00546875000000002 8.37421217604615e-07
0.00312500000000004 2.72210307470205e-07
0.00175781249999996 8.63504729858321e-08
0.0009765625 2.66130789064748e-08
0.000537109374999956 8.05686735335128e-09
};
\addplot [semithick, blueviolet, dash pattern=on 1.5pt off 2.475pt]
table {%
0.05 0.000745716386100436
0.0375 0.000417176730666245
0.0249999999999999 0.00018613835387795
0.015625 7.25257951063875e-05
0.00937499999999991 2.61501548875648e-05
0.00546875000000002 8.89005865898954e-06
0.00312500000000004 2.9044373240263e-06
0.00175781249999996 9.18701476455133e-07
0.0009765625 2.83598344106417e-07
0.000537109374999956 8.5780376460739e-08
};
\end{axis}

\end{tikzpicture}
			\quad
\begin{tikzpicture}[baseline={(0,0)}]

\definecolor{darkgrey176}{RGB}{176,176,176}
\definecolor{green01270}{RGB}{0,127,0}
\definecolor{steelblue31119180}{RGB}{31,119,180}
\definecolor{darkorange25512714}{RGB}{255,127,14}
\definecolor{blueviolet}{RGB}{138,43,226}

\begin{axis}[
scale = 0.8,
tick align=outside,
tick pos=left,
x grid style={darkgrey176},
xmin=-11.1895248953183, xmax=-3.99489967581922,
xtick style={color=black},
y grid style={darkgrey176},
ymin=-27.7120362762091, ymax=-9.56418291071711,
ytick style={color=black},
xlabel = {$\log_2(\mathrm{Size~of~perturbation})$},
ylabel = {$\log_2\bigl(\widehat\AW_2^2\bigr)$},
axis lines = box
]
\addplot [semithick, steelblue31119180]
table {%
-4.32192809488736 -11.1338783914548
-4.73696559416621 -11.9619024707892
-5.32192809488737 -13.1333078725897
-6 -14.4884985412509
-6.73696559416622 -15.9630201808039
-7.51457317282975 -17.5178862897424
-8.32192809488734 -19.132798754562
-9.15200309344509 -20.7928338096138
-10 -22.4888920495023
-10.8624964762502 -24.2138493558452
};
\addplot [semithick, green01270, dash pattern=on 5.55pt off 2.4pt]
table {%
-4.32192809488736 -13.8441558316385
-4.73696559416621 -14.6079210597193
-5.32192809488737 -15.8253265221055
-6 -17.1506748203581
-6.73696559416622 -18.6435862519539
-7.51457317282975 -20.1875431923416
-8.32192809488734 -21.808774967485
-9.15200309344509 -23.4652206789423
-10 -25.1632893315593
-10.8624964762502 -26.8871338505049
};
\addplot [semithick, blueviolet, dash pattern=on 1.5pt off 2.475pt]
table {%
-4.32192809488736 -10.3890853364213
-4.73696559416621 -11.2270536902647
-5.32192809488737 -12.3913370256653
-6 -13.7511462675398
-6.73696559416622 -15.2228209828457
-7.51457317282975 -16.7793756309492
-8.32192809488734 -18.3933098716027
-9.15200309344509 -20.0539005169548
-10 -21.7496475553169
-10.8624964762502 -23.4747771116163
};
\end{axis}

\end{tikzpicture}
			\caption{Approximation of the adapted Wasserstein distance for the CIR process; the dotted purple line shows the effect of a perturbed diffusion parameter $\gamma$, the solid blue line shows the effect of a perturbed mean reversion level $\eta$, and the dashed green line shows the effect of a perturbed mean reversion speed $\kappa$.} 
			\label{fig:cir}
		\end{figure}
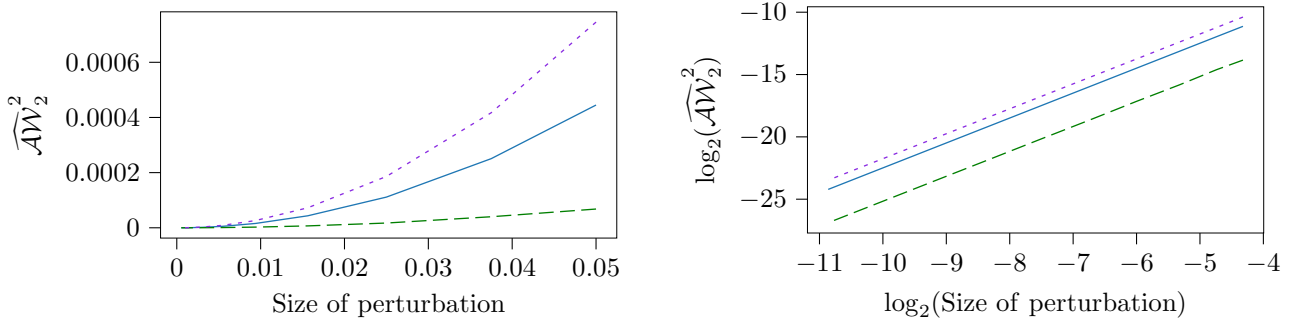
	\end{example}
	
\section{Application to robust optimisation}\label{sec:finance}
	We finally turn to applications in the robust approach to stochastic optimisation. Several results have been proven on stability of various stochastic optimisation problems \citep{AcBaZa20,BaBaBeEd19a} and (Lipschitz) continuity of optimal stopping problems \citep{AcBaZa20,BaBaBeEd19b,Al81}, in both discrete and continuous time, with respect to the adapted Wasserstein distance. In a discrete-time setup, the adapted Wasserstein distance has also been applied to quantify model-sensitivity of multi-period optimisation problems in \citep{AnPf14}, \citep{BaWi23}, and \citep{JiOb24}, with the latter also treating model-sensitivity of continuous-time problems.
	
	As a concrete application of the adapted Wasserstein distance to robust optimisation, we consider optimal stopping problems for continuous-time stochastic processes.
	Such problems are often analytically intractable and the numerical approximation of their solution may be computationally intensive; see, e.g.~\citep{BeChJeWe21} and the references therein.
	In \Cref{thm:optimal-stopping} we prove Lipschitz continuity of optimal stopping problems with respect to the adapted Wasserstein distance, which admits an efficient numerical approximation, following the approach of \Cref{sec:computation-sync}.
	Hence, our methods provide computationally feasible bounds on the discrepancy between the values of optimal stopping problems under different probability measures, making such problems tractable for applications.

	\subsection{Optimal stopping}		
		We consider the following class of optimal stopping problems. For an $\F$-adapted process $L \colon [0, T] \times \Omega \to \R$ and $p \geq 1$, define the value function $v^L \colon \Pc_p(\Omega) \to \R$ of the associated optimal stopping problem for any $\mu \in \Pc_p(\Omega)$ by
		\begin{equation}
			v^L(\mu) \coloneqq \inf_{\tau} \E^\mu[L(\tau (\omega), \omega)],                                                                                                                                                                                                                                                                                                                                                                                 
		\end{equation} where the infimum is taken over all stopping times with respect to $\F^\mu$, the completion of the canonical filtration under $\mu$.
		We prove the following Lipschitz estimate, which quantifies the model uncertainty in the choice of the measure $\mu \in \Pc_p(\Omega)$. 
		
		\begin{theorem}\label{thm:optimal-stopping}
			Let $p \geq 1$, let $L \colon [0, T] \times \Omega \to \R$ be $\F$-adapted and suppose that there exists a Lipschitz constant $C_L > 0$ such that, for all $t \in [0, T]$, $\gamma, \bar \gamma \in \Omega$,
			\begin{equation}
				|L(t, \gamma) - L(t, \bar \gamma)| \leq C_L \|\gamma - \bar \gamma\|_p.
			\end{equation}
			Then, for any $\mu, \nu \in \Pc_p(\Omega)$, we have the bound
			\begin{equation}\label{eq:optimal-stopping-bound}
				|v^L(\mu) - v^L(\nu)| \leq C_L \AW_p(\mu, \nu).
			\end{equation}
		\end{theorem}

		\begin{example}\label{ex:os-lip}
			Let $f \colon \R \to \R$ be Lipschitz and define $L \colon [0, T] \times \Omega \to \R$, for any $t \in [0, T]$, $\gamma \in \Omega$,  by
			\begin{equation}
				L(t, \gamma) \coloneqq f\Big(\int_0^t \gamma_s \D s\Big).
			\end{equation}
			Then $L$ satisfies the conditions of \Cref{thm:optimal-stopping} for any $p \geq 1$. In this example, the value $v^L$ of the corresponding optimal stopping problem has a financial interpretation as an American-style Asian option.
		\end{example}
		
		This result could also be derived from \citep[Proposition 4.4]{AcBaZa20}, but we rather provide a direct proof, following a similar approach to the proof of the analogous discrete-time result \cite[Lemma 7.1]{BaBaBeEd19b}.
		
		\begin{proof}[Proof of \Cref{thm:optimal-stopping}]
			Let $\pi \in \cplba(\mu, \nu)$ and let $(X, Y)$ be a pair of random variables on $\Omega \times \Omega$ with joint law $\pi$. As before, let $\F^X$ and $\F^Y$ denote the completions of the natural filtrations of $X$ and $Y$ under the measures $\mu$ and $\nu$, respectively. Fix $\varepsilon > 0$, and take an $\varepsilon$-optimal $\F^Y$-stopping time $\tau^Y$; i.e.
			\begin{equation}\label{eq:os-suboptimal}
				\E[L(\tau^Y, Y)] \leq v^L(\nu) + \varepsilon.	
			\end{equation}

			For each $u \in [0, 1]$, define $\theta_u \coloneqq \inf\bigl\{\, t \geq 0 : \pi(\tau^Y \leq t \mid \F^X_T) \geq u \,\bigr\}$. We claim that $\theta_u$ is an $\F^X$-stopping time. By continuity of the paths of processes $X$ and $Y$, we know that, for any $u \in [0, 1]$ and $t \in [0, T]$,
			\begin{equation}
				\{\theta_u \leq t\} = \bigl\{\pi\bigl(\tau^Y \leq t \mid \F^X_T\bigr) \geq u\bigr\}.
			\end{equation}
			So it suffices to show that, for each $t \in [0, T]$, the conditional probability $\pi(\tau^Y \leq t \mid \F^X_T)$ is $\F^X_t$-measurable. Since $\pi \in \cplba(\mu, \nu)$, the causality condition \eqref{eq:causality} states that, under $\pi$, $\F^Y_t$ is conditionally independent of $\F^X_T$, given $\F^X_t$. An equivalent formulation of this conditional independence condition (see, e.g.~\citet[Theorem 8.9]{Ka21}) is that for any $F \in \F^Y_t$,
			\begin{equation}\label{eq:cond-indep}
				\pi\bigl(F \mid \F^X_T \vee \F^X_t\bigr) = \pi\bigl(F \mid \F^X_t\bigr).
			\end{equation}
			In fact, $\F^X_t \subset \F^X_T$, and so $\F^X_T \vee \F^X_t = \F^X_T$. By the definition of a stopping time, $\{\tau^Y \leq t\} \in \F^Y_t$. Therefore \eqref{eq:cond-indep} gives us
			\begin{equation}
				\pi\bigl(\tau^Y \leq t \mid \F^X_T\bigr) = \pi\bigl(\tau^Y \leq t \mid \F^X_t\bigr),
			\end{equation}
			which is then $\F^X_t$-measurable. This proves our claim that $\theta_u$ is an $\F^X$-stopping time for each $u \in [0, 1]$. Thus
			\begin{equation}
				v^L(\mu) \leq \inf_{u \in [0, 1]}\E[L(\theta_u, X)] \leq \int_0^1 \E[L(\theta_u, X)] \D u.
			\end{equation}
			Applying the tower property of conditional expectation and Fubini's theorem to exchange the order of integration, we arrive at
			\begin{equation}\label{eq:os-upper-bound}
				v^L(\mu) \leq \E \!\left[\int_0^1 \E^\pi\bigl[L(\theta_u, X) \mid \F^X_T\bigr] \D u\right]\!.
			\end{equation}
			
			Now fix a path $\gamma \in \Omega$ of the process $X$. We can define a probability measure $\P_\gamma$ on $[0, T]$ by $\P_\gamma([0, t)) = \pi(\tau^Y \leq t \mid \F^X_T)(\gamma)$, for all $t \in [0, T]$. Then $u \mapsto \theta_u(\gamma)$ is the quantile function of $\P_\gamma$, and we see that
			\begin{equation}
				\begin{split}
					\E\bigl[L(\tau^Y, X) \mid \F^X_T\bigr](\gamma) & = \int_0^T L(t, \gamma) \D \P_\gamma(t) = \int_0^1 L(\theta_u(\gamma), \gamma) \D u = \int_0^1 \E\bigl[L(\theta_u, X) \mid \F^X_T\bigr](\gamma) \D u.
				\end{split}
			\end{equation}
			Taking the expectation with respect to $\pi$ and substituting into \eqref{eq:os-upper-bound}, we find
			\begin{equation}
				v^L(\mu) \leq \E^\pi\Bigl[\E\bigl[L(\tau^Y, X) \mid \F^X_T\bigr](\omega)\Bigr] = \E\!\left[L(\tau^Y, X)\right]\!.
			\end{equation}
			By \eqref{eq:os-suboptimal}, the Lipschitz continuity of $L$ in the second argument, and Jensen's inequality, we have
			\begin{equation}
			\begin{split}
				v^L(\mu) - v^L(\nu) & \leq C_L \cdot \E\!\left[\|X - Y\|_p\right]\! + \varepsilon \leq C_L \cdot \E{\left[\|X - Y\|_p^p\right]}^{\frac 1 p} + \varepsilon.
			\end{split}
			\end{equation}
			Exchanging the roles of $\mu$ and $\nu$ and noting that $\varepsilon > 0$ and $\pi \in \cplba(\mu, \nu)$ were chosen arbitrarily completes the proof.
		\end{proof}
				
\appendix
\section{Proofs from \Cref{sec:preliminaries}}\label{app:prelim}
	\begin{proof}[Proof of \Cref{prop:general-optimality}]
		For $h > 0$, \Cref{ass:it-1} implies that the Knothe--Rosenblatt rearrangement is optimal for \eqref{eq:bcot} between $\mu^h$ and $\nu^h$, by \Cref{prop:kr-optimal}. Taking a common Brownian motion $W = \bar W$ in \Cref{ass:it-2}, the Knothe--Rosenblatt rearrangement is equal to $\Law(X^h, \bar X^h)$, by \Cref{rem:kr-equivalent-def}. This together with \Cref{ass:it-3} allows us to 
		conclude that the limit \eqref{eq:limit-bcot} holds and the synchronous coupling attains the value \eqref{eq:bcot-cont} of the continuous-time bicausal transport problem. This final step is a straightforward modification of the proof of stability of bicausal optimal transport in \citep[Proposition 2.6, Corollary 2.7]{BaKaRo22}, and we omit the details here.
	\end{proof}
	
	\begin{proof}[Proof of \Cref{lem:stability}]
		Let $\pi \in \cplbc(\mu, \nu)$. Then by \cite[Proposition 2.2]{BaKaRo22} on the relation between bicausal couplings and correlated Brownian motions, there exists a correlated Brownian motion $(W, \bar W)$ such that $\pi = \Law(X, \bar X)$ when \eqref{eq:sde-correlated} is driven by $(W, \bar W)$. Taking the same driving Brownian motion, we write $\pi^N \coloneqq \Law(X^N, \bar X^N) \in \cplbc(\mu^N, \nu^N)$, for each $N \in \N$. Interpreting each such system of SDEs as a two-dimensional SDE, the stability result of \citet[Theorem 11.1.4]{StVa79} implies the weak convergence $\pi^N \rightharpoonup \pi$ as $N \to \infty$. Moreover, the uniform linear growth bound implies that, for any $p \geq 1$, there exists a constant $C > 0$ such that $\E^{\pi^N}[\sup_{t \in [0, T]}|(\omega_t, \bar \omega_t)|^p] \leq C$ for all $N \in \N$; see, e.g.~ \cite[Lemma 3.8]{GiSk79}. Thus $\E^{\pi^N}[\phi(\omega, \bar \omega)] \to \E^\pi[\phi(\omega, \bar \omega)]$, for any functional $\phi \colon \Omega \times \Omega \to \R$ that is continuous with order $p$ polynomial growth with respect to the uniform norm on $\Omega \times \Omega$; see, e.g.~\cite[Definition 6.8]{Vi09}. We claim that $(\omega, \bar \omega) \mapsto \int_0^T c_t(\omega_t, \bar \omega_t) \D t$ satisfies these conditions. Indeed the polynomial growth of order $p$ follows immediately from the bound \eqref{eq:poly-growth}, which holds uniformly in $t$. Continuity follows from \eqref{eq:poly-growth} and the continuity of the maps $(x, y) \mapsto c_t(x, y)$ for each $t \in [0, T]$ via the Lebesgue--Vitali dominated convergence theorem. Hence
		\begin{equation}\label{eq:stability-one-direction}
			\E^{\pi^N}\biggl[\int_0^Tc_t(\omega_t, \bar \omega_t) \D t\biggr] \xrightarrow{N \to \infty} \E^\pi\biggl[\int_0^Tc_t(\omega_t, \bar \omega_t) \D t\biggr].
		\end{equation}
		
		Now take a common one-dimensional Brownian motion driving \eqref{eq:sde-correlated} for each set of coefficients. Then we have
		\begin{equation}
			\pi^\sync_{\mu^N, \nu^N} = \Law(X^N, \bar X^N) \xrightarrow{N \to \infty} \Law(X, \bar X) = \pi^\sync_{\mu, \nu}.
		\end{equation}
		Given that $\pi^\sync_{\mu^N, \nu^N}$ minimises the left-hand side of \eqref{eq:stability-one-direction} over $\cplba(\mu^N, \nu^N)$ for each $N \in \N$, we have that $\pi^\sync_{\mu, \nu}$ minimises the right-hand side of \eqref{eq:stability-one-direction} over $\cplba(\mu, \nu)$ and the required convergence holds.
	\end{proof}
	
	\begin{proof}[Proof of \Cref{prop:time-dependent-lipschitz}]
		Under \Cref{ass:lipschitz}, we can show that the solutions of \eqref{eq:sde-correlated} and the corresponding monotone Euler--Maruyama schemes satisfy \Cref{ass:comon-scheme}, following the proofs of  \cite[Lemma 3.14, Lemma 3.15, Proposition 3.16]{BaKaRo22}. The result then follows from \Cref{prop:general-optimality}.
	\end{proof}
	
	\begin{proof}[Proof of \Cref{cor:regular}]
		For each SDE, existence of a weak solution is guaranteed by \citet[Chapter 3, Section 3]{Sk65}, and given that we assumed pathwise uniqueness, \citet[Corollary 1]{YW} implies that there exists a unique strong solution. We can approximate the coefficients $(b, \sigma)$ and $(\bar b, \bar \sigma)$ locally uniformly in space by coefficients satisfying \Cref{ass:lipschitz} and the conditions of \Cref{lem:stability}. We conclude by combining \Cref{prop:time-dependent-lipschitz} and \Cref{lem:stability}.
	\end{proof}
	
\section{Proofs from \Cref{sec:transformation}}\label{app:proofs-transform}

	\begin{proof}[Proof of \Cref{lem:phi-bar}]
		Direct computation of the first and second derivatives of $\bar \phi_k$ yield the first three points. We also observe that the first derivative is continuous and bounded, and hence $\bar \phi_k$ is Lipschitz. The second derivative is defined almost everywhere and is continuous and bounded on $(- \infty, \xi_k) \cup (\xi_k, \infty)$. Combined with the fact that the one-sided second derivatives at $\xi_k$ exist and are finite, this shows that $\bar \phi_k^\prime$ is Lipschitz. Finally, differentiating $\bar \phi_k^{\prime \prime}$ once again on $(- \infty, \xi_k) \cup (\xi_k, \infty)$ gives a bounded third derivative, and hence $\bar \phi_k^{\prime \prime}$ is piecewise Lipschitz with discontinuity point $\xi_k$.
	\end{proof}
	
	\begin{proof}[Proof of \Cref{lem:G-properties-growth-disc}]
		By \Cref{lem:phi-bar}, $\bar \phi_k (x) = 0$ for $|x - \xi_k| \geq c_0$ for each $k \in \{1, \dotsc, m\}$, and so we have $G(x) = x$ for $x \in \continuityset$. Moreover, for any $k \in \{1, \dotsc, m\}$ and $x \in [\xi_k - c_0, \xi_k + c_0]$, $G(x) = x + \alpha_k \bar \phi_k(x)$ and $G^\prime(x) = 1 + \alpha_k \bar \phi_k^\prime(x)$, implying that $G^\prime(x) \in [1 - 6 c_0 |\alpha_k|, 1 + 6 c_0 |\alpha_k|]$ by \Cref{lem:phi-bar}. Thus, as in \cite[Lemma 2.2]{LeSz17}, we observe that our choice of $c_0$ in \eqref{eq:c_0} guarantees that $G$ is strictly increasing on $\R$ and therefore admits a strictly increasing inverse $G^{-1} \colon \R \to \R$. From the definition of $G$ and the properties proved in \Cref{lem:phi-bar}, we immediately deduce that $G$ is bounded and Lipschitz with bounded Lipschitz first derivative $G^\prime$ and bounded piecewise Lipschitz almost-everywhere second derivative $G^{\prime \prime}$, with discontinuity points $\xi_1$, $\dotsc\,$,~$\xi_m$. By the inverse function theorem, $G^{-1}$ has the same properties.
	\end{proof}
		
	\begin{proof}[Proof of \Cref{lem:piecewise-plus-continuity}]
		Fix $c > 0$ and let $L_f$, $K_f$, $\gamma$, $\eta$ denote the constants from \Cref{it:A1} that appear in the one-sided Lipschitz and local Lipschitz conditions for $f$. By \Cref{rem:lip-compacts} and the continuity of $f$, we have that $f$ is Lipschitz on the compact interval $[\xi_1 - c, \xi_m + c]$ and hence satisfies \Cref{it:A1ii,it:A1iii} on $[\xi_1 - c, \xi_m + c]$. Suppose that  $x \in [\xi_1 - c, \xi_m + c)$, $y \in (-\infty, \xi_1 - c]$. Then, since $y < \xi_1 - c < x$, we have
		\begin{align}
			(x - y)(f(x) - f(y)) & = \frac{x - y}{x - (\xi_1 - c)}\bigl(x - (\xi_1 - c)\bigr)\bigl(f(x) - f(\xi_1 - c)\bigr)\\
			& \quad + \frac{x - y}{(\xi_1 - c) - y}\bigl((\xi_1 - c) - y\bigr)\bigl(f(\xi_1 - c) - f(y)\bigr)\\
			& \leq \frac{x - y}{x - (\xi_1 - c)}L_f|x - (\xi_1 - c)|^2 + \frac{x - y}{(\xi_1 - c) - y}L_f|(\xi_1 - c) - y|^2\\
			& = L_f (x - y)\bigl(x - (\xi_1 - c) + (\xi_1 - c) - y\bigr) = L_f |x - y|^2.
		\end{align}
		Similarly,
		\begin{align}
			|f(x) - f(y)| & \leq |f(x) - f(\xi_1 - c)| + |f(\xi_1 - c) - f(y)|\\
			& \leq K_f \bigl(e^{\gamma |x|^\eta} + e^{\gamma |\xi_1 - c|^\eta}\bigr)\bigl(x - (\xi_1 - c)\bigr) + K_f \bigl(e^{\gamma |\xi_1 - c|^\eta} + e^{\gamma |y|^\eta}\bigr)\bigl((\xi_1 - c) - y\bigr)\\
			& \leq K_f \bigl(2e^{\gamma |x|^\eta} + e^{\gamma |y|^\eta}\bigr)(x - y) + K_f \bigl(e^{\gamma |x|^\eta} + 2e^{\gamma |y|^\eta}\bigr)(x - y)\\
			& = 3 K_f \bigl(e^{\gamma |x|^\eta} + e^{\gamma |y|^\eta}\bigr)|x - y|.
		\end{align}
		The remaining cases are treated analogously, and we conclude that $f$ satisfies \Cref{it:A1ii,it:A1iii} globally on $\R$.
	\end{proof}

	\begin{proof}[Proof of \Cref{lem:transformed-coefficients}]
	 	First, as in \cite{LeSz17}, we note that the constants $\alpha_k$, $k \in \{1, \dotsc, m\}$ are chosen in \eqref{eq:alpha} such that $\tilde b$ is continuous on $\R$. Indeed by \Cref{lem:phi-bar}, for each $k \in \{1, \dotsc, m\}$, $G^\prime(\xi_k) = 1 + \alpha_k \bar \phi_k^\prime(\xi_k) = 1$, $G^{\prime \prime}(\xi_k+) = \alpha_k \bar \phi_k^{\prime \prime}(\xi_k+) = 2 \alpha_k$, $G^{\prime \prime}(\xi_k -) = - 2 \alpha_k$, and so
		\begin{equation}
			\begin{split}
				\bigl(b G^\prime + \tfrac12 \sigma^2 G^{\prime \prime}\bigr)(\xi_k+) - \bigl(b G^\prime + \tfrac12 \sigma^2 G^{\prime \prime}\bigr)(\xi_k-) & = b(\xi_k+) - b(\xi_k-) + 2\alpha_k \sigma^2(\xi_k) = 0.
			\end{split}
		\end{equation}
		Then, by \Cref{lem:G-properties-growth-disc}, $b G^\prime + \tfrac{1}{2} \sigma^2 G^{\prime \prime}$ is continuous on $\R$. Also by \Cref{lem:G-properties-growth-disc}, $G^{-1}$ is Lipschitz, and so $\tilde b$ is continuous on $\R$.
		
		Let $c > 0$. Then, on each finite interval $(\xi_1 - c, \xi_1)$, $(\xi_m, \xi_m + c)$, $(\xi_k, \xi_{k + 1})$, for $k \in \{1, \dotsc, m - 1\}$, we have that $b$, $\sigma$, and $G^\prime$ and $G^{\prime \prime}$ are bounded and Lipschitz by \Cref{rem:lip-compacts}, \Cref{it:A2}, and \Cref{lem:G-properties-growth-disc}, respectively. Thus on each such interval, $b G^\prime + \tfrac{1}{2} \sigma^2 G^{\prime \prime}$ is Lipschitz as the sum of products of bounded Lipschitz functions. Taking the composition with the Lipschitz function $G^{-1}$, we see that $\tilde b$ is Lipschitz, and thus satisfies \Cref{it:A1ii,it:A1iii}, on each interval $(\xi_1 - c, \xi_1)$, $(\xi_m, \xi_m + c)$, $(\xi_k, \xi_{k + 1})$, for $k \in \{1, \dotsc, m - 1\}$. On the intervals $(-\infty, \xi_1 - c)$ and $(\xi_m + c, + \infty)$, we have that $\tilde b = b$ and thus satisfies \Cref{it:A1ii,it:A1iii}. Since $\tilde b$ is continuous, \Cref{lem:piecewise-plus-continuity} implies statements (i) and (ii) of the lemma. Statement (iii) follows from statement (ii).
		
		Now, $\sigma G^\prime$ is globally Lipschitz, since it is the product of bounded Lipschitz functions on $[\xi_1 - c_0, \xi_m + c_0]$ and is equal to $\sigma$ on $(-\infty, \xi_1 - c_0] \cup [\xi_m + c_0, + \infty)$; cf.\ \cite[Lemma 2.5]{LeSz17}. Taking the composition with the Lipschitz function $G^{-1}$ preserves the Lipschitz property, and this concludes the proof of statement (iv).
		
		Finally, statement (v) of the lemma follows from statements (i) and (iv).
	\end{proof}	

\bibliographystyle{abbrvnat}
\bibliography{references.bib}

\end{document}